\documentclass[11pt,letterpaper]{article}
\usepackage{sparse-ising}
\usepackage{psmacros}
\usepackage{ising-macros}

\title{Fast Mixing in Sparse Random Ising Models}
\author{Kuikui Liu\thanks{MIT. Email: \texttt{liukui@mit.edu}}
  \and Sidhanth Mohanty\thanks{MIT. Email: \texttt{sidm@mit.edu}. Supported by CSAIL.}
  \and Amit Rajaraman\thanks{MIT. Email: \texttt{amit\_r@mit.edu}. Supported by an Akamai Presidential Fellowship.}
  \and David X. Wu\thanks{UC Berkeley. Email: \texttt{david\_wu@berkeley.edu}. Supported by NSF Graduate Research Fellowship DGE-2146752.}}
\date{\today}

\begin{document}
\maketitle

\begin{abstract}
    Motivated by the community detection problem in Bayesian inference, as well as the recent explosion of interest in spin glasses from statistical physics, we study the classical Glauber dynamics for sampling from Ising models with sparse random interactions. It is now well-known that when the interaction matrix has spectral diameter less than $1$, Glauber dynamics mixes in $O(n\log n)$ steps. Unfortunately, such criteria fail dramatically for interactions supported on arguably the most well-studied sparse random graph: the Erd\H{o}s--R\'{e}nyi random graph $\ER(n,d/n)$. There is a scarcity of positive results in this setting due to the presence of almost linearly many outlier eigenvalues of unbounded magnitude.

    We prove that for the \emph{Viana--Bray spin glass}, where the interactions are supported on $\ER(n,d/n)$ and randomly assigned $\pm \beta$, Glauber dynamics mixes in $n^{1+o(1)}$ time with high probability as long as $\beta \leq O\parens*{1/\sqrt{d}}$, independent of $n$. We further extend our results to random graphs drawn according to the $2$-community stochastic block model, as well as when the interactions are given by a ``centered'' version of the adjacency matrix. The latter setting is particularly relevant for the inference problem in community detection.
    Indeed, we build on this result to demonstrate that Glauber dynamics succeeds at recovering communities in the stochastic block model in a companion paper \cite{LMRRW24}.

    The primary technical ingredient in our proof is showing that with high probability, a sparse random graph can be decomposed into two parts --- a \emph{bulk} which behaves like a graph with bounded maximum degree and a well-behaved spectrum, and a \emph{near-forest} with favorable pseudorandom properties. We then use this decomposition to design a localization procedure that interpolates to simpler Ising models supported only on the near-forest, and then execute a pathwise analysis to establish a modified log-Sobolev inequality.

\end{abstract}

\thispagestyle{empty}
\setcounter{page}{0}
\newpage

\tableofcontents

\thispagestyle{empty}
\setcounter{page}{0}

\newpage


\section{Introduction}
In this paper, we tackle the problem of algorithmically sampling from the Gibbs distribution of an \emph{Ising model}, the quintessential model in statistical mechanics. These are probability distributions of the form
\begin{align*}
    \mu_{\beta}(\sigma) = \mu_{\beta J,h}(\sigma) \propto \exp\parens*{\frac{\beta}{2}\sigma^{\top} J \sigma + \langle h,\sigma\rangle}, \qquad \forall \sigma \in \{\pm1\}^{n},
\end{align*}
where $J \in \R^{n \times n}$ is a symmetric \emph{interaction matrix}, $h \in \R^{n}$ is an \emph{external field} vector, and $\beta \geq 0$ is the \emph{inverse temperature}, which scales the interaction strengths. There is a vast literature devoted to their study due to their numerous intimate connections with physics, statistical inference, combinatorial optimization, and the study of constraint satisfaction problems in theoretical computer science; see e.g., \cite{EKPS00, WJ08, MM09, Tal11, MNS15, FV17} and references therein. We consider the following fundamental question.

\begin{question*}
Under what conditions on $\beta$, $J$, and $h$ is (approximately) sampling from $\mu_{\beta}$ computationally tractable?
\end{question*}

This question has a long and rich history, and there is now a good understanding in the worst-case setting. Sampling from $\mu_{\beta}$ is easy if at least one of the following conditions hold:
\begin{itemize}
    \item $J$ is entrywise nonnegative (and all entries of $h$ have the same sign), i.e., the model is \emph{ferromagnetic} \cite{JS93}.
    \item The interaction strengths are sufficiently weak, e.g., \emph{Dobrushin's condition} $\beta \max_{i} \sum_{j} \abs{J_{ij}} < 1$ is satisfied \cite{Dob68,BD97}.
\end{itemize}
Conversely, if one lets $J = -A_{\bG}$ where $A_{\bG}$ is the adjacency matrix of a $d$-regular graph $\bG$, then there is a precise critical threshold $\beta_{c}(d) \approx 1/d$ such that the sampling problem is $\NP$-hard when $\beta > \beta_{c}(d)$ via reduction from $\mathsf{MAXCUT}$ \cite{SS12}, which corresponds to the $\beta \to +\infty$ limit. Notably, nearly-linear time samplers exist when $\beta < \beta_{c}(d)$ \cite{CLV21, AJKPV21b}, and so the problem can exhibit sharp \emph{computational phase transitions}.

Motivated by the stochastic block model from the study of community detection, as well as the theory of spin glasses in statistical physics, we focus on the \emph{average-case} setting where $J$ is a random matrix and the goal is to successfully sample from $\mu_{\beta}$ with probability $1 - o_{n}(1)$ over the randomness of $J$. Furthermore, since our aim is to obtain fast samplers, we consider \emph{Glauber dynamics} (or the \emph{Gibbs sampler}), a natural Markov chain over $\{\pm1\}^{n}$ used ubiquitously both in theory and in practice. Its evolution is described as follows: in each step, we select a uniformly random coordinate $i \in [n]$, and resample its assignment $\sigma_{i}$ conditioned on all remaining coordinates $\sigma_{-i}$.

A recent line of work \cite{BB19, EKZ22, CE22, KLR22, BBD23, AJKPV24} has identified that having spectral diameter bounded as
\begin{align}\label{eq:spectral-condition}
    \text{(Spectral Condition)} \qquad \beta \cdot \parens*{\lambda_{\max}(J) - \lambda_{\min}(J)} < 1
\end{align}
is sufficient to ensure that Glauber dynamics mixes in $O(n\log n)$ steps. More generally, polynomial-time samplers exist if $\beta J$ only has $O(1)$ many ``outlier'' eigenvalues exceeding $1$; see \cite{KLR22} for a precise statement. These types of conditions are appealing since many classical random matrix ensembles (e.g., $\mathsf{GOE}(n)$, the uniformly random $d$-regular graph, etc.) have tame bulk spectra. The spectral condition \cref{eq:spectral-condition} is also sharp for rapid mixing of Glauber dynamics, as witnessed by the Curie--Weiss model where $J = \frac{1}{n}\mathbf{1}\mathbf{1}^{\top}$. Stronger evidence for hardness of sampling when \cref{eq:spectral-condition} fails was recently given \cite{Kun23}.

\subsection{Diluted spin glasses}

Unfortunately, existing spectral criteria like \cref{eq:spectral-condition} fail badly for interaction matrices supported on the edges of arguably the most well-studied sparse random graph: the Erd\H{o}s–R\'{e}nyi random graph $\bG \sim \ER(n,d/n)$ with bounded average degree $d \leq O(1)$. For instance, consider the setting where all entries of $J$ are supported on a graph $\bG$, and randomly assigned $\pm1$ (or $\mathcal{N}(0,1)$) independently. By varying the choice of $\bG$, this recovers several famous models of spin glasses in statistical physics, e.g., the \emph{Sherrington--Kirkpatrick model} when $\bG = K_{n}$, and the \emph{Edwards--Anderson model} when $\bG = \Z^{d}$. When $\bG \sim \ER(n,d/n)$, sometimes known as the \emph{Viana--Bray model} (see e.g., \cite{Mon98, MP01, MP03, GT04}), such matrices automatically have $\Omega(n^{1-o(1)})$ many eigenvalues of magnitude $\omega(1)$ with high probability due to an abundance of ``stars'' with unusually high degrees, despite the overall sparsity of $\bG$. Hence, a direct application of existing results would yield efficient samplers only when $\beta \leq o_{n}(1)$. For comparison, if $\bG$ is a uniformly random $d$-regular graph, then \cref{eq:spectral-condition} yields rapid mixing as long as $\beta \leq O\parens*{1/\sqrt{d}}$, independent of $n$.

In this work, we overcome the abundance of large eigenvalues and obtain almost-linear mixing up to $\beta \leq O\parens*{1/\sqrt{d}}$, thus resolving the discrepancy between $\ER(n,d/n)$ and the uniformly random $d$-regular graph. As above, our first main result concerns mixing of Glauber dynamics for a diluted spin glass given by randomly signing the edges of a sparse Erd\H{o}s--R\'{e}nyi graph.
\begin{theorem}
    \label{th:main}
    Fix a constant $d > 1$. There exists a universal constant $C > 0$ such that for all $\beta \leq C/\sqrt{d}$ the following is true. Let $A$ be the random matrix such that each off-diagonal entry is equal to $0$ with probability $1-d/n$, and is equally likely to be $\pm 1$ with the remaining probability. With high probability over the draw of $A$, Glauber dynamics run on $\mu_{\beta A}$ mixes in time $n^{1+o_n(1)}$.
\end{theorem}
The bound on the constant $C$ we obtain is equal to $(3-\sqrt{8}) (1-o_d(1)) \approx 0.17-o_d(1)$. We expect that this bound can be improved using recent techniques for beating the spectral condition \eqref{eq:spectral-condition} for the Sherrington--Kirkpatrick model \cite{AKV24}. Based on replica-symmetry breaking predictions from statistical physics, e.g., \cite{GT04}, we further conjecture that Glauber dynamics mixes in polynomial time if $d \tanh^2(\beta)<1$, and takes exponential time to mix above this threshold.

Our results extend immediately to sparse random graphs drawn from the stochastic block model $\SBM(n, d, \lambda)$; see \cref{thm:main-mixing}. Roughly speaking, the key feature of such graphs which allows us to go beyond the spectral condition \cref{eq:spectral-condition} is \emph{sparsity} and \emph{localization} of the eigenvectors corresponding to large eigenvalues.

\begin{remark}
    We further believe that our methods extend to the regime where $d$ is slightly super-constant, in particular when $d < \log n$. In the setting where $d > \log n$, the spectral condition (\Cref{eq:spectral-condition}) can be used to bound the mixing time for $\beta = O\left(\frac{1}{\sqrt{d}}\right)$.
\end{remark}

\subsection{Centered graphs and community detection}
Besides the spin glass setting, we also consider the \emph{centered} setting, which is of particular relevance to community detection. Let us recall the ($2$-community) \emph{stochastic block model}, a heavily studied Bayesian model of community detection. We refer interested readers to \cite{Abb18} for further details on the history and importance of the model.
\begin{itemize}
    \item Let $\bsigma \in \{\pm1\}^{n}$ be a \emph{signal} vector drawn uniformly at random, which intuitively encodes the ground truth community assignments. In the language of Bayesian inference, the uniform distribution over $\{\pm1\}^{n}$ is the \emph{prior}.
    \item Given $\bsigma$, we draw a random graph $\bG$ by including an edge between $u,v \in [n]$ independently with probability $\frac{d + \lambda \sqrt{d}}{n}$ if $\bsigma(u) = \bsigma(v)$, and with probability $\frac{d - \lambda\sqrt{d}}{n}$ otherwise, where $d, \lambda \in \R$ are fixed constants.
\end{itemize}
The goal of statistical inference is to recover the true community assignments $\bsigma$ given access to the noisy observation $\bG$. The entirety of this generative process induces a distribution over graphs denoted $\SBM(n, d, \lambda)$, with the special case $\lambda = 0$ recovering the Erd\H{o}s--R\'{e}nyi graph $\ER(n,d/n)$. Note that having $\lambda \geq 0$ encourages the resulting random graph $\bG$ to have more intracommunity edges than intercommunity edges. The magnitude of $\lambda$ controls the \emph{signal-to-noise ratio (SNR)}, and is parameterized so that $\lambda^2 = 1$ corresponds to the famous \emph{Kesten--Stigum threshold} (see e.g., \cite{DKMZ11, MNS15}).

It is well-known that the inference problem can be reduced to the problem of sampling from the \emph{posterior distribution} over $\{\pm1\}^{n}$ given $\bG$, which for large $d$ is approximately given by the Ising Gibbs measure $\mu(\sigma) \propto \exp\parens*{\frac{\beta}{2} \sigma^{\top} \left( A_{\bG} - \frac{d}{n} \bone\bone^\top \right) \sigma}$, where $\beta = \beta(d,\lambda)$ is some explicit function monotone increasing in $\lambda$.\footnote{Technically, it is given by $\beta(d,\lambda,n) = \frac{1}{2}\ln\parens*{\frac{d+\lambda\sqrt{d}}{d-\lambda\sqrt{d}}} - \frac{1}{2}\ln\parens*{\frac{n-d+\lambda\sqrt{d}}{n-d-\lambda\sqrt{d}}}$, which tends to $\beta(d,\lambda) = \frac{1}{2} \ln\parens*{\frac{d+\lambda\sqrt{d}}{d-\lambda\sqrt{d}}} \approx \frac{\lambda}{\sqrt{d}}$ at an $O(1/n)$ rate as $n \to \infty$ for fixed $d, \lambda$.} A natural and longstanding open problem in the field is giving provable guarantees for recovering $\bsigma$ by simulating Glauber dynamics. One of the principal difficulties is that the regime $\lambda > 1$, where nontrivial recovery is information-theoretically possible, is also well within the regime in which worst-case mixing of Glauber dynamics is exponential in $n$. Hence, the vast majority of existing results in the literature give guarantees for algorithms not based on Markov chains at all (such as belief propagation and spectral algorithms, see, e.g., \cite{DKMZ11,Mas14,MNS14,AS15,BLM15,MNS15,MNS18,PY23} and references therein).

Towards understanding the performance of Glauber dynamics for inference, our second main contribution concerns sampling from the Ising model corresponding to the centered adjacency matrix $J = A_{\bG} - \E[A_{\bG}|\bsigma]$.\footnote{Note that (up to a diagonal shift), $\E[A_{\bG} | \bsigma] = \frac{d}{n} \bone\bone^{\top} + \frac{\lambda \sqrt{d}}{n} \bsigma\bsigma^{\top}$ and $\E[A_{\bG}] = \frac{d}{n}\bone\bone^{\top}$, so the former contains an uninformative eigenspace parallel to $\bone$, as well as the signal $\bsigma$.} In a companion paper \cite{LMRRW24},
we leverage this mixing result to prove that with high probability over $\bG$, a polynomial number of steps of Glauber dynamics for a rescaled version of the true posterior distribution achieves $\Omega(1)$ correlation with the ground truth $\bsigma$ (or $-\bsigma$) as long as $\lambda$ exceeds a certain (reasonable) constant greater than $1$.
\begin{theorem}
    \label{thm:main-centered}
    Fix constants $d > 1, \lambda \geq 0$. There exists a universal constant $C > 0$ such that for all $\beta \leq C/\sqrt{d}$ the following is true. Let $\ol{A}_{\bG} \coloneqq A_{\bG} - \E[A_{\bG}|\bsigma]$ be the centered adjacency matrix of $\bG \sim \SBM(n, d, \lambda)$. With high probability over the draw of $\bG$, Glauber dynamics run on $\mu_{\beta\ol{A}_{\bG}}$ mixes in time $n^{1+o_d(1)}$.
\end{theorem}

\subsection{Related work}
We now survey some of the literature relevant to the study of Ising models and sampling via Markov chains.

\parhead{Sampling from spin systems on sparse random graphs.}
In the ferromagnetic case where $J = A_{\bG}$, Mossel \& Sly \cite{ms13} established a mixing time of $n^{1 + \Theta(1/\log\log n)}$ for Glauber dynamics in the regime $d \tanh(\beta) < 1$ with high probability over $\bG \sim \ER(n,d/n)$. This is sharp as the mixing time necessarily scales exponentially in $n$ when $d \tanh(\beta) > 1$. More recently, when $J$ is supported on $\bG \sim \ER(n,d/n)$ and has i.i.d. Gaussian entries, Efthymiou \& Zampetakis \cite{EZ24} established almost-quadratic mixing of Glauber dynamics when $\beta \leq \sqrt{2\pi}/d$, and conjectured that their regime of $\beta$ is optimal for rapid mixing. While \cref{th:main} as stated only considers random signs, we believe that an adaptation of our techniques refutes this conjecture. For related results on the hardcore gas model (i.e., random independent sets) on $\ER(n,d/n)$, see \cite{EF23} and references therein.

In a recent work, Bauerschmidt, Bodineau \& Dagallier \cite{BBD23} proved that the Kawasaki dynamics for the ferromagnetic Ising model with fixed magnetisation on a random $d$-regular graph mixes rapidly for $\beta < O\left(1/\sqrt{d}\right)$.

\parhead{Stochastic localization and sampling.}
A common strategy to analyze a Markov chain is to decompose the stationary distribution as a mixture of simpler distributions.
To this end, we employ the \emph{stochastic localization} process.
This idea was first introduced in the context of studying isoperimetry in convex bodies by Eldan \cite{Eld13}, and has since seen numerous applications.
The work of Eldan, Koehler \& Zeitouni \cite{EKZ22} introduced it to the study of Glauber dynamics for Ising models, and used it to prove that any Ising model can be expressed as a ``nice'' mixture of rank-1 Ising models, a fact that was key to establishing the spectral condition for rapid mixing of Glauber dynamics.
Later, Chen \& Eldan \cite{CE22} provided a different proof of the same result by giving a stochastic localization process to decompose an Ising model into a mixture of product distributions; see also \cite{BB19}, which proved a similar result using a similar decomposition.

A parallel line of work uses stochastic localization as an algorithmic technique for sampling from spin glass models, rather than as a tool for analyzing Markov chains.
This line began with the work of El Alaoui, Montanari \& Sellke \cite{EAMS22} on an algorithm to sample from a distribution close to the Sherrington--Kirkpatrick model (SK model) in Wasserstein distance, when $\beta < 1$, based on discretizing the stochastic localization process.
In contrast, the spectral condition implies that Glauber dynamics succeeds only when $\beta < 1/4$, albeit under the more stringent total variation distance.
Analogous results have also been obtained for using stochastic localization to sample from mixed $p$-spin Ising models \cite{EAMS23}, parallel sampling algorithms \cite{AHLVXY23}, posteriors of spiked matrix models \cite{MW23}, and spherical spin glasses \cite{HMP24}.
See \cite{Mon23} for a survey on recent developments, and connections to diffusion models in machine learning.


\section{Technical overview}\label{sec:overview}

We now describe how we prove \Cref{th:main} on the rapid mixing of Glauber dynamics for an Ising model on a randomly signed sparse \erdos--\renyi graph.
The proof of \Cref{thm:main-centered} for centered stochastic block model graphs, while more involved, follows similar inspiration.

\parhead{Modified log-Sobolev inequality.}
The rapid mixing of Glauber dynamics follows from a \emph{modified log-Sobolev inequality} (MLSI), which is a functional inequality relating the global behavior of any function to its ``local differences'' in the Markov chain; see, e.g., \cite[Corollary 2.8]{bt06} for the connection between MLSI and mixing times.
Concretely, for a function $f:\{\pm1\}^n\to\R_{> 0}$ and probability distribution $\mu$, the quantities of interest to us are its:
\begin{itemize}
    \item {\bf Local entropy.}
    The \emph{local entropy} of $f$ is
    $\calE_{\mu}\parens*{f, \log f} \coloneqq \E_{\bx,\by}\bracks*{ (f(\bx) - f(\by)) \cdot \log\frac{f(\bx)}{f(\by)} }$ where $\bx\sim\mu$ and $\by$ is sampled by taking a single step of the Glauber dynamics chain from $\bx$.
    \item {\bf Global entropy.} The \emph{global entropy} of $f$ is $\Ent_{\mu}[f]\coloneqq \E_{\mu} f\log f - \E_{\mu} f \log \E_{\mu} f$.
\end{itemize}
We say that Glauber dynamics for $\mu$ satisfies a \emph{modified log-Sobolev inequality} (MLSI) if for any function $f:\{\pm1\}^n\to\R_{>0}$:
\[
    \calE_{\mu}(f,\log f) \gtrsim \Ent_{\mu}[f],
\]
where the $\gtrsim$ hides a $\poly(n)$ factor.

The following is the functional inequality we prove, which also implies rapid mixing.
\begin{theorem}[Informal version of \Cref{thm:main-mixing}] \label{thm:MLSI-sparse-ising}
    Let $A$ be the adjacency matrix of a randomly signed graph $\bG\sim \ER(n,d/n)$.
    There exists a constant $\beta = \Omega(1 / \sqrt{d})$ such that $\mu = \mu_{\beta A}$ satisfies an MLSI.
\end{theorem}
\begin{remark}
    In fact, we show that the polynomial factor hidden by $\gtrsim$ in the above is $n^{-1-o(1)}$.
    However, we do not stress this in the overview for the sake of simplicity.
\end{remark}

The general strategy to prove this inequality in our setting is similar to that of \cite{EKZ22,AJKPV21a,CE22}, which prove similar inequalities for Glauber dynamics on other Ising models, and can be summarized by the following.
\begin{displayquote}
    Decompose $\mu$ into a mixture $\rho$ over ``simpler'' distributions $\mu_z$ where $\mu = \E_{\bz\sim\rho} \mu_{\bz}$, and then establish the following chain of inequalities:
\end{displayquote}
\[
    \calE_{\mu}(f, \log f) \underset{\substack{\uparrow \\ \text{Conservation} \\ \text{of local} \\ \text{entropy}}}{\geq} \E_{\bz\sim\rho} \calE_{\mu_{\bz}}(f, \log f) \underset{\substack{\uparrow\\ \text{MLSI for} \\ \text{simple} \\ \text{distributions} }}{\gtrsim} \E_{\bz\sim\rho} \Ent_{\mu_{\bz}} [f]
    \underset{\substack{\uparrow\\\text{Conservation} \\ \text{of entropy}}}{\gtrsim} \Ent_{\mu} [f]\mper
\]
The first inequality is generic and does not use any properties of the structure of $\rho$; see, e.g., \cite[Page 19]{AJKPV21a}.
We reproduce the details here for completeness:
\begin{align*}
    \calE_{\mu}(f, \log f) &= \sum_{x\sim y} \frac{1}{n} \cdot \frac{\mu(x) \mu(y)}{\mu(x)+\mu(y)}\cdot\parens*{f(x)-f(y)}\cdot\log\frac{f(x)}{f(y)} \\
    &= \sum_{x\sim y} \frac{1}{n} \cdot \frac{\E_{\bz\sim\rho}\left[\mu_{\bz}(x)\right]\cdot \E_{\bz\sim\rho}\left[\mu_{\bz}(y)\right]}{\E_{\bz\sim\rho}\left[\mu_{\bz}(x)\right]+\E_{\bz\sim\rho}\left[\mu_{\bz}(y)\right]}\cdot\parens*{f(x)-f(y)}\cdot\log\frac{f(x)}{f(y)} \\
    &\ge \sum_{x\sim y} \frac{1}{n} \cdot \E_{\bz\sim\rho}\left[\frac{\mu_{\bz}(x) \mu_{\bz}(y)}{\mu_{\bz}(x)+\mu_{\bz}(y)}\right]\cdot\parens*{f(x)-f(y)}\cdot\log\frac{f(x)}{f(y)} \\
    &= \E_{\bz\sim\rho} \left[\calE_{\mu_{\bz}}(f, \log f)\right],
\end{align*}
where the inequality follows from concavity of the function $(a,b)\mapsto\frac{ab}{a+b}$ in the nonnegative quadrant.

The second inequality relies on the MLSI for the component measures, which we hope are easier to analyze. The third inequality necessarily relies on the properties of $\rho$.
We now describe our construction of a measure decomposition of $\mu$ when it is an Ising model with PSD interaction matrix $J$ and external field $h$, and then show how to establish the MLSI for the simple distributions, and conservation of entropy.\footnote{We can assume that $J\psdge 0$ without loss of generality, as adding an arbitrary diagonal matrix to $J$ does not change the Gibbs distribution.}

\parhead{Measure decomposition.}
Our measure decomposition is based on the \emph{Hubbard--Stratonovich transform}, an algebraic trick to express a Gibbs distribution arising from a quadratic Hamiltonian as a mixture of simpler distributions.
The measure decomposition takes the following form.
\begin{displayquote}
    Let $M$ be a PSD matrix that we shall call a \emph{control matrix}, let $\bx\sim\mu$, and let $\bz \coloneqq \bx + M^{-1/2}\bg$ for $\bg\sim\calN(0, I)$ independent of $\bx$.
    Let $\rho$ denote the distribution of $\bz$ and let $\mu_{\bz}$ denote the distribution of $\bx|\bz$.    
\end{displayquote}
We now perform a fairly mechanical calculation to determine the distribution of $\bx|\bz$, just to illustrate a favorable algebraic cancellation.
Let $P(x), P(z|x)$ and $P(x,z)$ be the joint densities of $\bx$, $\bz|\bx$ and $(\bx,\bz)$ respectively.
We have:
\begin{align*}
    \Pr[\bx = x|\bz = z] &\propto P(x, z) \\
    &= P(x)\cdot P(z|x) \\
    &\propto \exp\parens*{\frac{1}{2}x^{\top}Jx + \angles*{h,x}}\cdot\exp\parens*{-\frac{1}{2}(z-x)^{\top}M(z-x)} \\
    &\propto \exp\parens*{\frac{1}{2}x^{\top}(J-M)x + \angles*{h+Mz, x}}\mper
\end{align*}

\parhead{Designing the control matrix}
The upshot of the above is that the measures $\mu_{\bz}$ are Ising models with interaction matrix $J-M$ and some external field depending on $\bz$.
We would like to design $M$ such that
\begin{itemize}
    \item It is relatively easy to prove an MLSI for an Ising model with interaction matrix $J-M$.
    \item The mixture $\rho$ satisfies conservation of entropy.
\end{itemize}

Choosing $M = J$ makes the interactions disappear altogether, and makes $\mu_z$ a product distribution.
This is desirable as every product distribution satisfies an MLSI.
Indeed, this choice is made in the work of Chen and Eldan \cite[Corollary 51]{CE22} where they prove an MLSI for Glauber dynamics on Ising models when $J$ has spectral diameter at most $1$ (see \cref{eq:spectral-condition}).
They then use the bound on the spectral diameter to prove conservation of entropy for the resulting product decomposition of the Ising model.

Unfortunately, when $J$ is a matrix associated with a sparse \erdos--\renyi graph, the abnormally high-degree vertices cause a large number of outlier eigenvalues, which is a hurdle for their technique to establish conservation of entropy.
Nevertheless, studying their technique motivates our design of $M$.

\parhead{Conservation of entropy.}
Recall that \emph{conservation of entropy} for our decomposition refers to the following functional inequality.
For $f:\{\pm1\}^n\to\R_{>0}$: 
\[
    \E_{\bz\sim\rho} \Ent_{\mu_{\bz}}[f] \gtrsim \Ent_{\mu}[f]\mper
\]
The strategy for doing this, as carried out in \cite{EKZ22,CE22}, is to chart a continuous path between $\mu$ and $\mu_{\bz}$ and control the derivative of entropy along this path.
Concretely, there is a stochastic process $(\mu_t)_{0\le t\le 1}$ on the space of probability distributions called \emph{stochastic localization} with the following properties; see \cite[Proposition 39 and Lemma 40]{CE22}:\footnote{We refer the reader to \cite{EAM22} for a nice exposition on stochastic localization and why it indeed charts such a path.}
\begin{itemize}
    \item $\mu_0 = \mu$ and $\mu_1 = \mu_{\bz}$ for $\bz\sim\rho$.
    \item The function $F(t) = \E \Ent_{\mu_t}[f]$, where the expectation is taken over the randomness in the choice of $\mu_{t}$, satisfies the following lower bound on its derivative:
    \[
        \restr{\frac{\dif F}{\dif t}}{t} \ge -\sup_{h}\norm*{M^{1/2}\cdot\Cov\parens*{\mu_{J-tM,h}}\cdot M^{1/2}}_{\mathrm{op}} \cdot F(t)\mper \numberthis \label{eq:deriv-sl}
    \]
\end{itemize}
If we can show for any $h \in \R^{n}$ and $t\in[0,1]$ that $\norm*{M^{1/2}\cdot\Cov\parens*{\mu_{J-tM, h}}\cdot M^{1/2}}_{\mathrm{op}}\le q$ for some $q > 0$, then
\[
    F(1) \ge F(0)\cdot\exp\parens*{-q}\mcom
\]
which, in more familiar terms, is
\[
    \E_{\bz\sim\rho} \Ent_{\mu_{\bz}}[f] \ge \exp\parens*{-q} \Ent_{\mu}[f]\mper
\]
Our goal thus reduces to designing $M$ such that $\norm*{M^{1/2}\cdot\Cov\parens*{\mu_{J-tM,h}} \cdot M^{1/2}}_{\mathrm{op}}$ is small, and that it is easy to  prove an MLSI for Ising models with interaction matrix $J-M$.

Intuitively, the large spectral norm of $M$ itself can be blamed on the high-degree vertices.
To heuristically get a sense of where $\Cov\parens*{\mu_{J-tM,h}}$ might have large spectral norm, let us pretend $J$ is the interaction matrix of a tree, and consider the special case of $h=0$ and $t=0$. 
We remark that the final proof will \emph{uniformly} bound the covariance matrix for all external fields $h \in \R^n$ and times $t \in [0,1]$ with high probability over the graph, although we omit the details here for the sake of simplicity.

Under this pretense, via standard formulas for correlations in tree graphical models,
\[
    \Cov\parens*{\mu_{J,0}}_{i,j}\approx\Omega\parens*{\frac{1}{\sqrt{d}}}^{\mathrm{dist}(i,j)}\mper
\]
Let $n_{\ell}(i)$ denote the number of neighbors of vertex $i$ at distance exactly $\ell$ from $i$.
The entries of the $i$th row corresponding to the distance-$\ell$ neighbors of a vertex $i$ are at least $\parens*{1/d}^{\ell/2}$, and since the operator norm of a matrix is at least the $\ell_2$ norm of any row, we have:
\[
    \norm*{\Cov\parens*{\mu_{J,0}}}_{\mathrm{op}} \gtrsim \Omega\parens*{\sqrt{\frac{n_{\ell}(i)}{d^{\ell}}}}\mper
\]
This can be huge if vertex $i$ is abnormally high-degree, or more generally, if vertex $i$ has an abnormally large $\ell$-neighborhood.
In particular, if $n_{\ell}(i) > \parens*{Cd}^{\ell}$ for some appropriately chosen constant $C > 1$, that would reflect in a blow-up in the spectral norm of $\Cov\parens*{\mu_{J,0}}_{N_{\ell}(i)}$, the covariance matrix restricted to the $\ell$-neighborhood of $i$.
The idea then is to consider $M$ that is supported only on rows and columns away from such ``rogue'' vertices $i$.

In particular, we partition the graph into two pieces, one piece with vertices close to the rogue vertices, and another piece with vertices far away from rogue vertices, and design $M$ to be supported on the rows and columns from the second piece.

\parhead{Graph decomposition.}
Concretely, for a graph $G$:
\begin{itemize}
    \item For each vertex $v$, define $\ell(v)$ as the smallest value of $\ell$ such that for all $L\ge \ell$, the number of vertices at distance $\le L$ from $v$ is at most $(d(1+\eps))^L$.
    \item Set $H$ as $\bigcup_{v\in V} B_{\ell(v)}(v)$, where $B_{\ell(v)}(v)$ is the radius $\ell(v)$ ball around $v$.
\end{itemize}
The above gives a natural partition of the graph $G$ into $B \coloneqq G\setminus H$, which we call the \emph{bulk}, and $H$, which we call the \emph{near-forest}.
The point of this decomposition is that all the rogue vertices reside in the near-forest, and the vertices touched by edges in the bulk are ``tame''; see \Cref{fig:graph-decomposition} for an illustration.
\begin{figure}
    \centering
    \begin{tikzpicture}[
    vertex/.style={
        circle, 
        draw=black, 
        fill=black, 
        align=center, 
        minimum size=1mm,
        inner sep=0pt, 
        outer sep=0pt},
    center/.style={
        circle, 
        draw=black, 
        fill=black, 
        align=center, 
        minimum size=2mm,
        inner sep=0pt, 
        outer sep=0pt},
    outside vertex/.style={
        circle, 
        draw=gray, 
        fill=gray, 
        minimum size=1mm,
        inner sep=0pt,
        outer sep=0pt,
        scale=0.8},
    special ball/.style={
        circle, 
        draw=blue, 
        preaction={fill, white}, 
        opacity=0.5, 
        pattern=north east lines, 
        pattern color=blue!80, 
        minimum size=20mm},
    graph/.style 2 args={
        minimum width=#1cm,
        minimum height=#2cm,
        thick, 
        rounded corners=10mm, 
        fill=green!20, 
        opacity=0.4,
    },
    boundary edge/.style={
        gray, 
        dashed
    },
    bulk edge/.style={
        gray
    }
    ]

\pgfdeclarelayer{background}
\pgfdeclarelayer{foreground}
\pgfsetlayers{background,main,foreground}

\begin{scope}
    
    \begin{pgfonlayer}{background}
        \draw[graph] (0, 0) rectangle (10, 6) {};
        \def\points{0/1.5/1.5,1/2.7/4.2,2/7.0/3.2,3/8.3/2.8,4/1.9/2.2,5/1.1/5.0,6/2.2/2.0,7/8.5/1.1,8/2.1/3.4,9/1.3/4.5};
        \node[outside vertex] (b1) at (1.5, 2.2) {};
        \node[outside vertex] (b2) at (4.5, 1.5) {};
        \node[outside vertex] (b3) at (5.5, 2) {};
        \node[outside vertex] (b4) at (7.2, 2.2) {};
        \node[outside vertex] (b5) at (8, 2.5) {};
        \node[outside vertex] (b6) at (2, 4.5) {};
        \node[outside vertex] (b7) at (1.3, 5.2) {};
        \node[outside vertex] (b8) at (2.5, 4.8) {};
        \node[outside vertex] (b9) at (6.1, 3.7) {};

        \node[outside vertex] (p0) at (2.5, 4.8) {};
        \node[outside vertex] (p1) at (3, 5) {};
        \node[outside vertex] (p2) at (3.6, 4.3) {};
        \node[outside vertex] (p3) at (4.2, 5.2) {};
        \node[outside vertex] (p4) at (5, 4.2) {};
        \node[outside vertex] (p5) at (5.2, 3.5) {};
        \node[outside vertex] (p6) at (6, 3.2) {};
        \node[outside vertex] (p7) at (6.8, 3.5) {};
        \node[outside vertex] (p8) at (7.2, 2.2) {};
        \node[outside vertex] (p9) at (6.5, 1.6) {};
        \node[outside vertex] (p10) at (5.8, 1.2) {};
        \node[outside vertex] (p11) at (4.5, 1.5) {};
        \node[outside vertex] (p12) at (4, 1.5) {};
        \node[outside vertex] (p13) at (3, 1.8) {};

        \foreach \i in {0,...,12} {
            \draw[bulk edge] (p\i) -- (p\the\numexpr\i+1\relax);
        }
    
        \node at (0.75,0.75) {\huge$G$};
    \end{pgfonlayer}

    \begin{pgfonlayer}{foreground}
        \begin{scope}[shift={(3.5, 3)}]
            \node[special ball] (C) at (0, 0) {};

            \node[center] (c1) at (0, 0) {};
            \node[vertex] (v1) at (0.3, 0.2){};
            \node[vertex] (v2) at (-0.1, 0.6){};
            \node[vertex] (v3) at (-0.5, 0.1){};
            \node[vertex] (v4) at (-0.5, -0.2){};
            \node[vertex] (v5) at (0, -0.4){};
            \node[vertex] (v1-6) at (0.3, -0.5){};
            \node[vertex] (v1-7) at (0.4, -0.2){};
            \node[vertex] (v1-8) at (-0.2, 0.2){};

            \foreach \i in {1, 2, 3, 4, 5} {
                \draw (c1) -- (v\i);
            };
            \foreach \i in {6, 7, 8} {
                \draw (c1) -- (v1-\i);
            };

            \node[vertex] (v2-1) at (0.5, 0.4){};
            \node[vertex] (v2-2) at (0.7, 0.1){};
            \node[vertex] (v2-3) at (0.6, -0.1){};
            
            \foreach \i in {1, 2, 3} {
                \draw (v1) -- (v2-\i);
            };

            \foreach \i in {1, 2,4, 6, 7, 11} {
                \node[outside vertex] (o\i)at ({\i*360/12}:1.4) {};
            };

            \draw[boundary edge] (v2-1) -- (o1);
            \draw[boundary edge] (v2-2) -- (o11);
            \draw[boundary edge] (v2) -- (o4);
            \draw[boundary edge] (v3) -- (o6);
            \draw[boundary edge] (v4) -- (o7);
            \draw[boundary edge] (v4) -- (o6);
            \draw[bulk edge] (o2) -- (o1);
            \draw[bulk edge] (o4) -- (p1);
            \draw[bulk edge] (o1) -- (p4);
            \draw[bulk edge] (o2) -- (p2);
            \draw[bulk edge] (o2) -- (p3);

        \end{scope}
        
        \begin{scope}[shift={(6.5, 4.75)},scale=0.8,every node/.append style={transform shape}]
            \node[special ball] (E) at (0, 0) {};
            \node[center] (c1) at (0, 0) {};
            \node[vertex] (v1-overlap) at (0.5, 0.2){};
            \node[vertex] (v2) at (0.1, 0.7){};
            \node[vertex] (v3) at (-0.4, 0.2){};
            \node[vertex] (v4) at (-0.6, -0.3){};
            \node[vertex] (v4-overlap) at (0.2, -0.4){};
            \node[vertex] (v1-6) at (0, -0.5){};
            \node[vertex] (v1-7) at (0.3, -0.2){};
            \node[vertex] (v1-8) at (-0.3, 0.6){};

            \draw (c1) -- (v1-overlap);
            \draw (c1) -- (v2);
            \draw (c1) -- (v3);
            \draw (c1) -- (v4);
            \draw (c1) -- (v4-overlap);
            \foreach \i in {6, 7, 8} {
                \draw (c1) -- (v1-\i);
            };

            \draw[boundary edge] (v4) -- (b9);
            \draw[bulk edge] (p6) -- (p7) -- (b9) -- (p6);
        \end{scope}

        
        \begin{scope}[shift={(7.75, 4.5)},scale=0.8,every node/.append style={transform shape}]
            \node[special ball] (D) at (0, 0) {};
            \node[center] (c1) at (0, 0) {};
            \node[vertex] (v1) at (0.3, 0.2){};
            \node[vertex] (v2) at (-0.1, 0.6){};
            \node[vertex] (overlap) at (-0.8, -0.05){};
            \node[vertex] (v3) at (-0.3, 0.3){};
            \node[vertex] (v4) at (-0.1, -0.6){};
            \node[vertex] (v1-6) at (0.3, -0.3){};
            \node[vertex] (v1-7) at (0.5, -0.1){};
            \node[vertex] (v1-8) at (0.1, -0.6){};

            \draw (overlap) -- (v1-overlap);
            \draw (overlap) -- (v4-overlap);

            \foreach \i in {1, 2, 3, 4} {
                \draw (c1) -- (v\i);
            };
            \foreach \i in {6, 7, 8} {
                \draw (c1) -- (v1-\i);
            };

            \draw (c1) -- (overlap);
        \end{scope}

        \begin{scope}[shift={(8, 1.5)},scale=0.5,every node/.append style={transform shape}]
            \node[special ball] (E) at (0, 0) {};
            \node[center] (c1) at (0, 0) {};
            \node[vertex] (v1) at (0.7, 0.2){};
            \node[vertex] (v2) at (0.1, 0.6){};
            \node[vertex] (v3) at (-0.5, -0.5){};
            \node[vertex] (v4) at (-0.3, 0.3){};
            \node[vertex] (v5) at (0.1, -0.6){};
            \node[vertex] (v1-6) at (0.3, -0.3){};
            \node[vertex] (v1-7) at (0.5, -0.1){};
            \node[vertex] (v1-8) at (-0.5, 0.0){};

            \foreach \i in {1, 2, 3, 4, 5} {
                \draw (c1) -- (v\i);
            };
            \foreach \i in {6, 7, 8} {
                \draw (c1) -- (v1-\i);
            };

            \draw[boundary edge] (v2) -- (b5);
        \end{scope}

        \draw[bulk edge] (b1) -- (o6);
        \draw[bulk edge] (b1) -- (o7);
        \draw[bulk edge] (b2) -- (b3);
        \draw[bulk edge] (b2) -- (o11);
        \draw[bulk edge] (b3) -- (o11);
        \draw[bulk edge] (b4) -- (b5);
        \draw[bulk edge] (b6) -- (o4);
        \draw[bulk edge] (b7) -- (b6);
        \draw[bulk edge] (b6) -- (b8);
        \draw[bulk edge] (b7) -- (b8);
    \end{pgfonlayer}
    
    \end{scope}
    
    \begin{scope}[shift={(10.5, 2)}]
        \draw[thick] (0,0.75) rectangle (3.5,2.6);
        
        \node[anchor=west] at (0,2.25) {\textbf{Legend}};
    
        \node[anchor=west] at (0.25,1.75) {\tikz\draw[fill=green!20, opacity=0.5](0,0) rectangle (0.5,0.3); {Bulk $B$}};
        \node[anchor=west] at (0.25,1.25) {\tikz\fill[special ball] (0,0) rectangle (0.5,0.3); Near-forest $H$};
    \end{scope}
    \end{tikzpicture}
    \caption{An illustration of our graph decomposition (\Cref{lem:graph-decomp-overview}) for a sparse SBM with average degree $d = 3$ and $\eps = 0.5$. 
    The bulk $B$ (green) consists solely of vertices with degree at most $d(1+\eps) = 4.5$. 
    The near-forest $H$ (shaded blue circles) comprise the local neighborhoods of high-degree vertices (enlarged for emphasis) in $\bG$. 
    The local neighborhoods of nearby high-degree vertices can merge, as shown by the two blue circles in the top right.
    Nevertheless, with high probability, all the connected components in $H$ are trees with at most one additional edge.}
    \label{fig:graph-decomposition}
\end{figure}
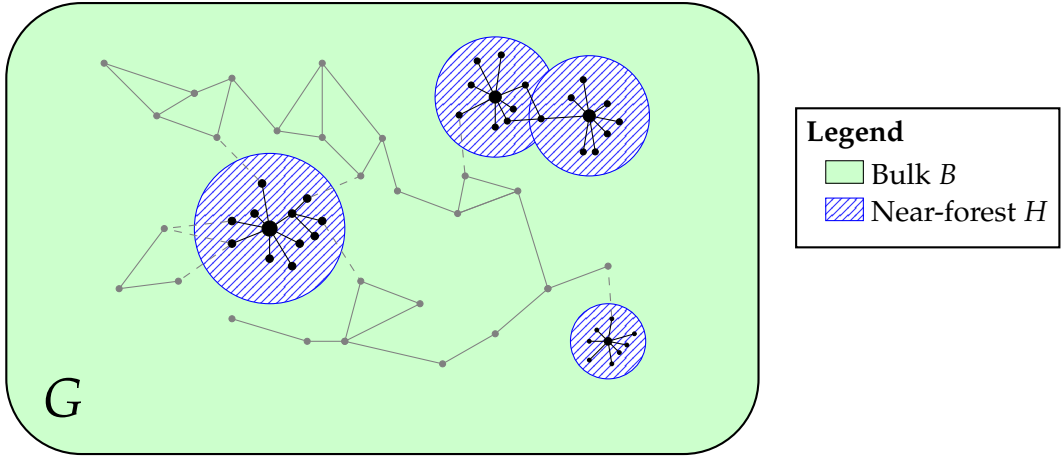
We prove the following in the context of sparse random graphs.
\begin{restatable}{lemma}{graphoverview}\label{lem:graph-decomp-overview}
    For $\bG\sim \ER(n,d/n)$ and random signing $\bc$ of the edges of $\bG$, there exists with high probability a subgraph $H$ and $\eps = O(\log d / d)^{1/3}$ such that:
    \begin{itemize}
        \item Every connected component of $H$ is a tree plus at most one additional edge.
        \item The diameter of each connected component of $H$ is at most $O\parens*{\frac{\log n}{\eps^3 d}}$.
        \item For any vertex $v$ and $\ell \ge 1$, the size of $B_{\ell}(v)$ is at most $\Delta \cdot \parens*{d(1+\eps)}^{\ell-1}$ for $\Delta = o(\log n)$.
        \item For any vertex $v$ and $\ell \ge 1$, the size of $B_{\ell}(v) \cap B$ is at most $(d(1+\eps))^{\ell}$ where $B$ is the set of all vertices touched by an edge in $\bG\setminus H$.
        \item The spectral radius of the signed adjacency matrix of $\bG\setminus H$ is at most $2\sqrt{d}(1+\eps)$.
    \end{itemize}
\end{restatable}
We defer a discussion of the proof of this graph decomposition lemma to later in this overview, and focus on the proof of mixing of Glauber dynamics for now.

The above statement motivates the choice of the control matrix as:
\[
    M\coloneqq\beta\cdot\parens*{ A_{\bG \setminus H} + 2\sqrt{d}(1+\eps)\cdot I_{\bG\setminus H}}\mper
\]

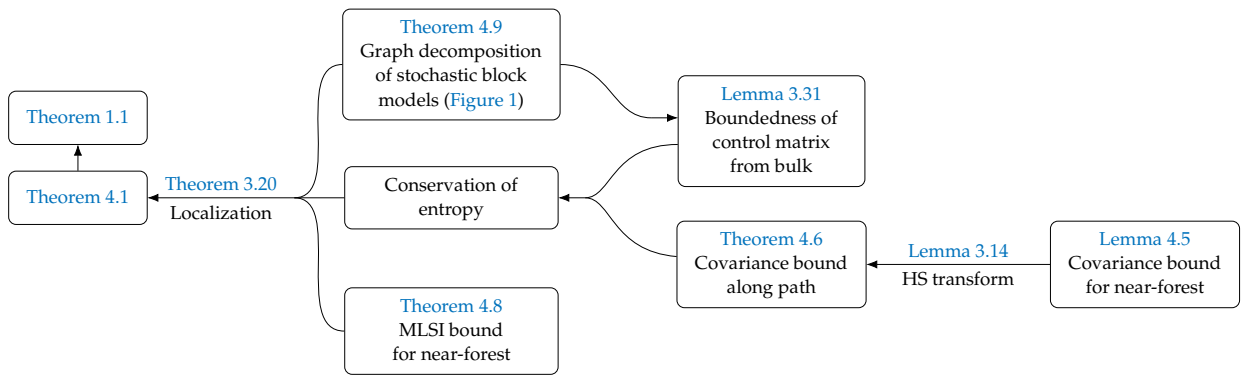
\begin{figure}
\resizebox{\linewidth}{!}{
\begin{tikzpicture}
\node (P01) [smallrect] at (-2, 2.5) {\small\begin{tabular}{c}\Cref{thm:main-mixing}\end{tabular}};
\node (P00) [smallrect] at (-2, 4) {\small\begin{tabular}{c}\Cref{th:main}\end{tabular}};

\node (P10) [bigrect] at (5, 5cm) {\small\begin{tabular}{c}\Cref{lem:random-graph-decomposition}\\Graph decomposition\\of stochastic block\\models (\Cref{fig:graph-decomposition})\end{tabular}};
\node (P11) [bigrect] at (5, 2.5cm) {\small\begin{tabular}{c}Conservation of\\entropy\end{tabular}};
\node (P12) [bigrect] at (5, 0cm) {\small\begin{tabular}{c}\Cref{thm:mlsi-bdd-growth-tree}\\MLSI bound\\for near-forest\end{tabular}};

\node (P20) [mediumrect] at (11,3.75cm) {\small\begin{tabular}{c}\Cref{lem:bulk-specrad-signed}\\Boundedness of\\control matrix\\from bulk\end{tabular}};
\node (P22) [mediumrect] at (11,1.25cm) {\small\begin{tabular}{c}\Cref{thm:entropy-conservation}\\Covariance bound\\along path\end{tabular}};

\node (P32) [smallrect] at (18,1.25cm) {\small\begin{tabular}{c}\Cref{lem:near-forest-cov}\\Covariance bound\\for near-forest\end{tabular}};

\draw[-{Latex[length=2mm]}] (P01) -- (P00);

\coordinate (Qf) at ([xshift=+2.75cm]P01.east); 
\draw (P10) .. controls (2,5) and (3,2.5) .. (Qf) -- (Qf); 
\draw (P11) -- (Qf);
\draw (P12) .. controls (2,0) and (3,2.5) .. (Qf) -- (Qf);
\draw[-{Latex[length=2mm]}] (Qf) -- node[above,midway] {\small \Cref{thm:sum-MLSI}} node[below,midway] {\small Localization} (P01);

\coordinate (Qf3) at ([xshift=-0.5cm,yshift=+0.25cm]P20.west);
\coordinate (Qf4) at ([yshift=+0.25cm]P20.west);
\draw (P10) .. controls (8,5) and (8,4) .. (Qf3) -- (Qf3);
\draw[-{Latex[length=2mm]}] (Qf3) -- (Qf4);

\coordinate (Qf2) at ([xshift=+0.5cm]P11.east);
\coordinate (P20temp) at ([yshift=-0.25cm]P20.west);
\draw (P20temp) .. controls (8,3.5) and (8,2.5) .. (Qf2) -- (Qf2);
\coordinate (P22temp) at ([yshift=+0.25cm]P20.west);
\draw (P22) .. controls (8,1.5) and (8,2.5) .. (Qf2) -- (Qf2);
\draw[-{Latex[length=2mm]}] (Qf2) -- (P11);

\draw[-{Latex[length=2mm]}] (P32) -- node[below,midway] {\small HS transform} node[above,midway] {\small\Cref{lem:interaction-sum-cov-bound}} (P22);

\end{tikzpicture}
}
\caption{A flow chart outlining our proof of an MLSI for diluted spin glasses. }
\label{fig:spinglass-flowchart}
\end{figure}

An outline of the proof is presented in \Cref{fig:spinglass-flowchart}. We remark that other proofs of MLSIs that utilize stochastic localization, such as that of \eqref{eq:spectral-condition}, also follow a similar outline, albeit with many steps being much simpler. We shall now concretely clarify our goals before proceeding. 
\begin{itemize}
    \item Recall that our goal in proving conservation of entropy was to give an upper bound on $\norm*{M^{1/2}\cdot\Cov\parens*{\mu_{J-tM, h}}\cdot M^{1/2}}_{\mathrm{op}}$.
    \item We would also like to prove that any Ising model with interaction matrix $J-M$ satisfies an MLSI.
    In light of \Cref{lem:graph-decomp-overview}, the structure of $J-M$ is a ``near-forest'' where every connected component is a low-diameter tree with at most one additional edge with ``not too large'' local neighborhoods.
    Our goal is to prove an MLSI for such near-forest Ising models.
\end{itemize}

\parhead{Covariance bound along localization path.}
Observe that $\norm*{M^{1/2}\cdot\Cov\parens*{\mu_{J-tM, h}}\cdot M^{1/2}}_{\mathrm{op}}$ is at most $\norm*{\Cov\parens*{\mu_{J-tM, h}} \cdot M}_{\mathrm{op}}$.
We show that this spectral norm is $O(1)$.

For simplicity, in this overview, we focus on the special case when $t = 0$ and $h = 0$.
Let us denote $\Cov\parens*{\mu_{J,0}}$ as $Q$.
First observe that:
\[
    \norm*{Q \cdot M}_{\mathrm{op}} \le \norm*{Q}_{\mathrm{op}}\cdot\norm*{M}_{\mathrm{op}}\mper
\]
To obtain a handle on $Q$, we perform a trick identical to how we constructed our measure decomposition, albeit for a completely different purpose now.
\begin{displayquote}
    Let $\bx\sim\mu_{J,0}$, and let $\bz\coloneqq \bx + M^{-1/2}\bg$ for $\bg\sim\calN(0,I)$.\footnote{Technically, $M$ needs to be slightly modified by adding a tiny diagonal on the entries outside $V(B)$, but we ignore that consideration here.}
\end{displayquote}
We obtain a bound on $Q$ by using $\Cov[\bz] \psdge \Cov[\bx]$, which lets us pass to giving a bound on $\Cov[\bz]$.
To this end, we prove that the distribution of $\bz$ is \emph{strongly log-concave} and give quantitative lower bound on its strong log-concavity.
A bound on the covariance then immediately follows from the Brascamp--Lieb inequality, which says:
\begin{displayquote}
    Suppose $\nabla^2 \log f(y) \psdge \Gamma$ for all $y$, where $f$ is the density of $\bz$. Then, $\Cov[\bz]\psdle\Gamma^{-1}$.
\end{displayquote}
A standard calculation reveals that it suffices to give a lower bound in the PSD order on the matrix $M - M\cdot\Cov\parens*{\mu_{J-M,0}}\cdot M$.
Since $M$ is supported only on the ``bulk'' $B\coloneqq \bG\setminus H$, we can write this as $M - M\cdot\Cov\parens*{\mu_{J-M,0}}_{B}\cdot M$, where $\Cov(\mu_{J-M,0})_{B}$ is the principal submatrix of the covariance matrix on the rows and columns corresponding to non-isolated vertices in $B$.\footnote{We need the lower bound $\Gamma$ to be strictly positive definite, which motivates the trick from the previous footnote of adding a small scaling of the identity matrix to the $V(H)$ block.
However, we ignore this technicality for the present discussion.
}
We now make the observation
\[
    M - M\cdot\Cov\parens*{\mu_{J-M,0}}\cdot M = M - M \cdot \Cov\left( \mu_{J-M,0} \right)_B \cdot M \psdge M - M^2 \cdot \norm*{\Cov\parens*{\mu_{J-M,0}}_B}_{\mathrm{op}}\mcom
\]
which motivates obtaining a bound on $\Cov\parens*{\mu_{J-M,0}}_{B}$.
The spectral norm of $M$ can be suitably bounded using existing results on the spectra of sparse random graphs by a small enough constant that scales linearly with $\beta$; see \Cref{sec:graph-prelims} for details.
We will show that $\norm*{\Cov\parens*{\mu_{J-M,0}}_{B}}_{\mathrm{op}}\le O(1)$.

\parhead{The Ising model on near-forests.}
We now discuss how to control the spectral behavior of $\Cov\parens*{\mu_{H, h}}$ where $H$ is a near-forest satisfying the conditions from \Cref{lem:graph-decomp-overview}, as well as how we prove the MLSI for such Ising models.

Observe that $\Cov\parens*{\mu_{H,h}}$ is a block diagonal matrix where each block corresponds to the covariance matrix of the Ising model of a connected component $F$, which is a tree with at most one excess edge.
Each $F$ has a specially identified boundary $\partial F$, which is the intersection of $F$ with $B$.
Thus, it is sufficient to bound the spectral norm of each $\Cov\parens*{\mu_{F, h}}_{\partial F}$.
Bounding the spectral norm of this matrix entails the following steps:
\begin{itemize}
    \item We reduce to proving a covariance norm bound to the case when $F$ is a tree, the interactions are nonnegative, and the external field is $0$, but omit the description of this reduction in the overview; see the proof of \Cref{lem:near-forest-cov} for details.
    \item There is a simple explicit formula for the covariance matrix of an Ising model on a tree with zero external field in terms of the interaction matrix: for a pair of vertices $i$ and $j$, and the unique path $P$ between $i$ and $j$, the covariance $\Cov\parens*{\mu_{F,0}}_{i,j}$ is equal to $\prod_{e\in P}\tanh(J_e)$.
    \item We use the trace moment method to bound the desired spectral norm.
    The explicit formula from the previous point makes such a calculation tractable.
    When performing this calculation, one needs to count closed walks in the tree, and the bound on the size of balls in $H$ guaranteed by \Cref{lem:graph-decomp-overview} lead us to the desired bound.
\end{itemize}

We now turn our attention to proving that the leftover near-forest in $J-M$ satisfies an MLSI for any choice of external field $h$.
In order to make the discussion smoother and concretely point out which models we show an MLSI for, we make the following definition (see \cref{def:pseudorandom-graph} for a formal version).
\begin{definition}[$(\Delta, D)$-pseudorandomness, informal]
    We say a graph $H$ is \emph{$(\Delta,D)$-pseudorandom} if for every vertex $v\in V(H)$, the number of vertices at distance at most $\ell$ from $v$ is at most $\Delta\cdot D^{\ell-1}$.
\end{definition}
It suffices for us to prove an MLSI for $(\Delta, D)$-pseudorandom near-forests where $\Delta = o(\log n)$ and $D = d(1+\eps)$.
To show an MLSI for pseudorandom near-forests, we reduce to exhibiting an MLSI for pseudorandom trees; see \Cref{sec:mlsi-tree-ising} for details.

We prove the following for pseudorandom trees.
\begin{theorem}[Informal version of \Cref{thm:mlsi-bdd-growth-tree}]
    Suppose $\Delta = o(\log n)$ and $0\le \gamma < 1$ is an absolute constant.
    Then, for any $(\Delta,D)$-pseudorandom tree $H$ with interaction matrix $J$ with all entries in $\parens*{-\frac{\gamma}{\sqrt{D}}, \frac{\gamma}{\sqrt{D}}}$ and any external field $h$, Glauber dynamics for $\mu_{J,h}$ satisfies an MLSI.
\end{theorem}

To prove the above, we employ the strategy of decomposing the measure into a mixture of product distributions using stochastic localization.
Recall from \eqref{eq:deriv-sl} that to control the derivative of the entropy along the stochastic localization path, it suffices to control $\norm*{ M \cdot \Cov\parens*{\mu_{J-tM, h}}}_{\mathrm{op}}$ where $M$ is the control matrix of choice for stochastic localization.
It is tempting to use $J + \lambda_{\min}(J) \cdot I$ as the control matrix.
However, the exact choice of the control matrix requires care because the parts of both this matrix and the covariance matrix that have high spectrum  are on the same coordinates.
In fact, we need to engineer the control matrix so that its high-spectral norm part partially cancels the corresponding part of the covariance matrix when they are multiplied.
We refer the reader to the proof of \Cref{lem:pseudorandom-ferro-mixing} for details.
We spend the rest of the overview discussing how we prove \Cref{lem:graph-decomp-overview}.

\parhead{Proving the existence of a graph decomposition for sparse random graphs.}
We next outline how to prove \Cref{lem:graph-decomp-overview}, which establishes nice properties of the graph decomposition into bulk and near-forest components for sparse stochastic block models. 
Before diving into the discussion, we restate the lemma for convenience.
\graphoverview*

As its name suggests, the bulk component has a tame spectrum, because every vertex in the bulk has bounded degree by construction.
On the other hand, because sparse random graphs are locally tree-like, the near-forest component $H$ consists of the union of balls around the high degree vertices. 
Let us see how to push this further to prove the other two properties of the graph decomposition.

To show that $H$ is a near-forest, a crucial ingredient is the following  structural property about small sets in sparse random graphs from \cite[Lemma 30]{BLM15}.
\begin{displayquote}
    For $\bG \sim \SBM(n,d,\lambda)$, with probability $1-o(1)$, any connected set of vertices with fewer than $\frac{1}{2} \cdot \log_d n$ vertices has at most one cycle.
\end{displayquote}
In fact, the above lemma can be used to rule out the existence of multiple cycles in any connected set of vertices (in particular the components of $H$), as long as the diameter of this set is $O(\log n)$.
The idea is to rule out the existence of the second cycle by considering the small set of vertices consisting of the two cycles, along with the shortest path between them. 
Because of the diameter bound, the total size of this set is also $O(\log n)$, so the above lemma applies --- if the hidden constant in the diameter bound is small enough, then the lemma implies that this subgraph has at most one cycle, a contradiction.

To summarize, we have boiled down the proof of \Cref{lem:graph-decomp-overview} to showing the diameter bound. 

\parhead{Bounding the diameter of the near-forest.}
We illustrate the intuition behind such a statement from the perspective of branching processes.
Consider the following exploration process, which simulates the ball-growing process that generates $H$.

At all stages of the exploration, we maintain a set of active vertices, unexplored vertices, and inactive vertices.
The active vertices represent the boundary of the connected component we are currently exploring.
\begin{enumerate}
    \item If there are no active vertices, select a vertex $v$ which has not been explored yet. If all vertices have been visited, terminate the process. 
    \item Select a vertex $v$ in the active set, and iteratively query its radius-$r$ neighborhoods. 
    \item If there is some vertex $u$ at distance at most $r$ which is $r$-heavy, that is, $\abs{B_{r}(u)} > (d(1+\eps))^r$, then expand out the active vertices to include $u$ by growing a ball of radius $2 \cdot \ell_u$ around $v$. Set all vertices on the interior of the active set to be inactive.
    \item If no active vertices remain, this means that a new connected component of $H$ has been fully generated. 
\end{enumerate}
Analyzing the exploration process directly is challenging, but one natural approach is to consider a branching process which simulates the above exploration process. 

For simplicity, let us restrict to the case where $\bG$ is a sparse random graph with constant average degree $d$. 
It is well known that for such a graph, the degree of any vertex $v$ has Poisson tails, in that for any $s \ge 1$, we have
\[
    \Pr[\deg{v} \ge sd] \le \exp(-\Omega(sd\log s)).
\]
In particular, the maximum degree of $\bG$ is $\Theta\parens*{\frac{\log n}{\log \log n}}$ with high probability.
In fact, using standard techniques from the theory of branching processes, we prove a similar bound (\Cref{lem:ball-tail}) for the sizes of radius-$r$ neighborhoods of $v$:
\[
    \Pr[\abs{B_{r}(v)} \ge (sd)^{r}] \le \exp(-\Omega(ds^{r})).
\]
Observe how the above tail bound decays doubly exponentially in $r$. Using this, it is not too difficult to show that the probability that a vertex is $r$-heavy also decays as $\exp\left( -\Omega(ds^r) \right)$. We also remark that the proof strategy of this tail bound is key in proving that with high probability, the size of \textit{any} ball of radius $\ell$ is at most $o(\log n) \cdot (d(1+\eps))^{\ell-1}$, which is required to attain an $n^{1+o(1)}$ mixing time.

The hope in a proof involving branching processes would be that although the number of vertices at depth at most $r$ is $O(d^r)$, the probability of any of these vertices being $r$-heavy decays doubly exponentially with $r$, so the radius of the largest ball that captures the vertex under consideration is very small --- it has tail bounds that decay doubly exponentially, and has expectation $e^{-\Omega(d)}$. Consequently, the number of active vertices is likely to shrink in each step of the exploration. Furthermore, the $O(\log n/d)$ diameter bound is tight using such a proof technique: a vertex has slightly abnormal degree and is $1$-heavy with probability $e^{-\Theta(d)}$, so with probability $1/n$, there are $O(\log n/d)$ such vertices that are all adjacent to each other.

However, this proof technique as stated is difficult to carry out, because the definition of heaviness involves not just subtrees rooted at the relevant vertex in the tree, but also subtrees rooted at ancestors and siblings, which become difficult to handle.
For this reason, the proof we give in the body of the paper involves a more direct moment method for the diameter bound, but the intuition remains the same. We now briefly describe the ideas involved therein.

\medskip
In our actual proof, we consider the associated \emph{cluster graph}, which has as its vertex set those vertices $v$ that have $\ell(v) > 0$, and edges between vertices if their corresponding balls intersect. Our goal then is to bound the diameter of this new cluster graph, where each edge is assigned weight equal to the sum of the radii of (the balls around) its endpoints. To do so, we consider the probability that a \textit{fixed} collection of vertices forms a long path. The key is that if these balls have small radii, then they must intersect in a specific fashion. However, each of the balls are essentially random (small!) subsets of the $n$ vertices, so this is extraordinarily unlikely. On the other hand, if the radii are somewhat large, the doubly exponential tail bound on the radius of heaviness kicks in, which again makes this event very unlikely.

\medskip
\parhead{Organization.}
We begin in \Cref{sec:prelims} with an overview of the numerous techniques we will require in our proof.

\Cref{sec:spinglass} is dedicated to proving the mixing time bound for the diluted spin glass. In \Cref{sec:decomposition}, we describe a graph decomposition, and state several results that will play a part in the proof. In \Cref{sec:spinglass-mlsi}, before proving any of these results, we begin by putting them together to prove an MLSI. In \Cref{sec:covariance-near-forest}, we bound the covariances of Ising models supported on near-forests using the trace method. In \Cref{sec:entropy-conservation-spinglass}, we then use this bound to prove a covariance bound that is used to show that entropy is conserved along the localization path. In \Cref{sec:mlsi-tree-ising}, we prove a modified log-Sobolev inequality for Ising models supported on near-forests, which we are left with at the end of the localization scheme.

Many of the results and ideas from the previous section make a return in \Cref{sec:sbm-results}, where we prove an MLSI for Ising models whose interaction matrix is a scaling of a centered adjacency matrix. We begin with a high-level overview of the proof strategy in the centered setting. In \Cref{sec:perturbed-near-forest}, we introduce some additional tools related to the spectral independence framework that are required for the proof, and then instantiate these tools to prove an MLSI for the Ising model obtained at the end of the localization scheme. In \Cref{sec:centered-adjacency}, we reuse many of the tools introduced in the previous subsection to show that entropy is conserved along the path of the localization scheme.

Finally, in \Cref{sec:extinction}, we prove that the decomposition of stochastic block models has the properties we claim.


\section{Preliminaries}
\label{sec:prelims}

\subsection{Notation}

    \begin{itemize}
        \item Given a Markov chain $P$ on $\Omega$, and $x \in \Omega$, $P_x$ is the distribution over $\Omega$ obtained by taking one step of the Markov chain from $x$.
        \item Given $\sigma \in \{\pm 1\}^n$, $\sigma^{\oplus i} \in \{\pm 1\}^n$ is defined as
        \[ \left(\sigma^{\oplus i}\right)_j = \begin{cases} \sigma_j, & j \ne i, \\ -\sigma_i, & j = i. \end{cases} \]
        \item Given a matrix $M$, we denote its operator norm by $\|M\| = \|M\|_{\mathrm{op}}$.
        \item Given a distribution $\mu$ over $\Omega$ and a function $f : \Omega \to \R_{> 0}$,
        \[ \Ent[f] = \Ent_{\mu}[f] \defeq \E_{\mu}[f \log f] - \E_{\mu}[f] \log \E_{\mu}[f]. \]
        \item For a random variable $X$ over $\R^k$, we use $\Cov[X]$ to denote the covariance $\E[XX^\top] - \E[X]\E[X]^\top$. For a probability distribution $\mu$ over $\R^k$, we use $\Cov(\mu)$ to denote the covariance of the associated random variable with law $\mu$.
        \item Given a connected graph $G = (V,E)$, its \emph{tree-excess} is $|E|-|V|+1$.
        \item Given a graph $H$, we use $V(H)$ and $E(H)$ to denote its vertex and edge sets respectively. We use $V'(H)$ to refer to the subset of vertices of $V(H)$ adjacent to at least one other distinct vertex, the set of non-isolated vertices.
        \item For a matrix $M \in \R^{n \times n}$ and a set of vertices $S \subseteq [n]$, by $M_S$ we mean the $n \times n$ matrix where every row and column outside $S$ is zeroed out, i.e., the principal submatrix of $M$ restricted to $S$ but preserving the dimension. If $H$ is a graph on $[n]$, then by $M_H \in \R^{n \times n}$ we mean $M_{V'(H)}$, i.e., the restriction of $M$ to the non-isolated vertices of $H$.
        \item We often abuse notation and identify probability measures $\mu$ with their density functions.
        \item Given a matrix $C$ that is possibly not full-rank, and $\bz \sim \calN(0,I)$, we mean by $C^{-1/2} \bz$ the distribution supported on $\mathrm{Ker}(C)^\perp$ with density at $z$ proportional to $\exp\left(-\frac{1}{2} \cdot z^\top C z\right)$.
    \end{itemize}

\subsection{Markov chains}
    
    In this section, we have a distribution $\mu$ over $\Omega$, and shall look at how one can analyze Markov chains with stationary distribution $\mu$. Throughout, we work solely with time-reversible Markov chains. 
    First, we describe a natural Markov chain to sample from distributions over $\{\pm 1\}^n$.
    \begin{definition}[Glauber dynamics]
        Let $\Omega = \{\pm 1\}^n$, and suppose that $\mu$ is supported on all of $\Omega$. The corresponding \emph{Glauber dynamics} Markov chain $P = P_\mu$ is described by the following transition law from any state $\sigma$:
        \begin{itemize}
            \item Choose index $i$ uniformly at random from $[n]$.
            \item Go to $\sigma^{\oplus i}$ with probability $\displaystyle\frac{\mu(\sigma^{\oplus i})}{\mu(\sigma)+\mu(\sigma^{\oplus i})}$, and stay at $\sigma$ otherwise.
        \end{itemize}
    \end{definition}

    \begin{fact}
        The Markov chain $P_\mu$ is ergodic and time-reversible (with respect to $\mu$), so has $\mu$ as its stationary distribution.
    \end{fact}

    Our goal is to bound the ``mixing time'' of Glauber dynamics on certain distributions.

    \begin{definition}
        Given a Markov chain $P$ with stationary distribution $\mu$, and a distribution $\nu$ that is absolutely continuous with respect to $\mu$, the associated mixing time initialized at $\nu$ is defined by
        \[ \tmix(P,\eps;\nu) = \min\left\{ t > 0 : d_{\mathrm{TV}}(P^t\nu,\mu) < \eps \right\}. \]
        We define the \emph{mixing time} to be
        \[ \tmix(P,\eps) = \max_{x \in \Omega}\tmix(P,\eps;\delta_x). \]
    \end{definition}

    It is well-known that the mixing time of Glauber dynamics (and Markov chains more generally) can be controlled by log-Sobolev inequalities. 

    \begin{definition}[Dirichlet Form]
        Given a distribution $\mu$ over $\{\pm 1\}^n$, a time-reversible Markov chain $P$ with stationary distribution $\mu$, and functions $f,g : \{\pm 1\}^n \to \R$, the associated \emph{Dirichlet form} is defined as
        \[ \mathcal{E}(f,g) = \mathcal{E}_{P}(f,g) \defeq \underset{\substack{x \sim \mu \\ y \sim P_x}}{\E} \left[ (f(x) - f(y))(g(x) - g(y)) \right]. \]
    \end{definition}

    \begin{definition}[Modified log-Sobolev Inequality]
        Given a time-reversible Markov chain $P$ on $\{\pm 1\}^n$ with stationary distribution $\mu$, $P$ is said to satisfy a \emph{modified log-Sobolev inequality} (MLSI) with constant $C$ (possibly depending on $n$) if for any function $f : \{\pm 1\}^n \to \R_{> 0}$,
        \[ \mathcal{E}(f,\log f) \ge C \Ent[f]. \]
        In particular, $\MLSI$ is the best (largest) such constant $C$.
    \end{definition}

    \begin{fact}[{\cite[Corollary 2.8]{bt06}}]
        \label{fact:mixing-mlsi}
        Given a time-reversible Markov chain $P$ with stationary distribution $\mu$, if $P$ satisfies a modified log-Sobolev inequality with constant $\MLSI$, then
        \[ \tmix(P,\epsilon) \le \frac{1}{\MLSI}\left( \log \log \frac{1}{\mu_*} + \log \frac{1}{\epsilon} \right), \]
        where $\mu_* = \min_{x : \mu(x) > 0} \mu(x)$.
    \end{fact}

    \begin{fact}[{\cite[Lemma 2.5]{Goe04}}]
        \label{fact:mlsi-product}
        Let $\mu$ be a distribution on $\{\pm 1\}^n$, and $(C_i)_{i=1}^{k}$ a partition of $[n]$ such that there exist distributions $\mu^{(i)}$ on $\{\pm 1\}^{C_i}$ with $\mu = \bigotimes_{i=1}^{K} \mu^{(i)}$. For the Glauber dynamics Markov chain, if each $\mu^{(i)}$ satisfies an MLSI with constant $\rho^{(i)}$, then $\mu$ satisfies an MLSI with constant $\rho = \frac{1}{n} \min_{i} \rho^{(i)} |C_i|$. 
    \end{fact}

    We will also require the Holley-Stroock perturbation principle, which shows that MLSIs and variances do not change too much if we multiply the densities by a constant.\footnote{Their result works with log-Sobolev inequalities instead of modified log-Sobolev inequalities.
    Nevertheless, the proof is easily adapted.}

    \begin{lemma}[{\cite[p. 1186]{HS86}, see also \cite[Lemma 1.2]{Led01}}]\label{lem:mlsi-bdd-density-relate}
      Let $\mu,\nu$ be two probability measures over $\{\pm1\}^{n}$, and suppose that for some $c > 1$ (possibly depending on $n$), $\frac{1}{c} \leq \frac{\nu(\sigma)}{\mu(\sigma)} \leq c$ uniformly over all $\sigma \in \{\pm1\}^{n}$.
      \begin{enumerate}
          \item Suppose Glauber dynamics for $\mu$ has MLSI constant at least $\MLSI$. Then, Glauber dynamics for $\nu$ has MLSI constant at least $\MLSI / c^2$.
          \item For any function $f$,
          \[ \frac{1}{c} \cdot \Var_{\mu}[f] \le \Var_{\nu}[f] \le c \cdot \Var_{\mu}[f]. \]
      \end{enumerate}
      Also note that if $\nu \propto e^{-f}$ and $\mu \propto e^{-g}$, and $\|f-g\|_{\infty} \le \wt{c}$, then $e^{-2\wt{c}} \le \frac{\nu(\sigma)}{\mu(\sigma)} \le e^{2\wt{c}}$.
    \end{lemma}

    We make frequent use of the following simple observation about covariance matrices.
    \begin{observation}\label{fact:var-cov}
        For any distribution $\mu$ on $\R^n$ with finite second moment, and vector $v \in \R^n$,
        \[ v^\top \Cov(\mu) v = \Var[\langle v,x\rangle]. \]
        In particular, 
        \[
            \sup_{\substack{v \in \mathbb{S}^{n-1} \\ \supp(v) \subseteq S}} \Var_{x \sim \mu}[\angles{v, x}] = \norm{\Cov(\mu)_S}.
        \]
    \end{observation}

\subsection{Ising models}
    In this section, we have a symmetric \emph{interaction matrix} $J\in\R^{n\times n}$ and an \emph{external field} $h\in\R^n$. 
    \begin{definition}[Ising model]
        The \emph{Ising model} corresponding to $J$ and $h$ is the probability distribution $\mu_{J, h}$ on $\{\pm1\}^n$, where
        \[
            \mu_{J, h}(\sigma) \propto \exp\parens*{\frac{1}{2}\sigma^\top J \sigma + \langle h,\sigma\rangle}.
        \]
        Its \emph{partition function} is
        \[
            Z_{J, h} \coloneqq \sum_{\sigma\in\{\pm1\}^n} \exp\parens*{\frac{1}{2}\sigma^{\top} J\sigma + \angles*{h,\sigma}}.    
        \]
    \end{definition}

    \begin{observation}
        \label{obs:diag-invariant}
        A simple observation we frequently make use of is that for any diagonal matrix $D$, $\mu_{J, h} = \mu_{J+D, h}$.
    \end{observation}

    We rely on a well-known trick that lets us decompose an Ising model into a mixture of simpler Ising models.
    \begin{theorem}[Hubbard--Stratonovich transform]    \label{thm:HS}
        For any positive semidefinite matrix $C$ in $\R^{n\times n}$:
        \[
            \mu_{J, h} = \E_{z\sim\nu} \mu_{J-C, Cz+h}
        \]
        where $\nu$ is a distribution over $\R^n$ with density
        \[
            \nu(z) \propto \exp\parens*{-\frac{1}{2}z^{\top} C z + \log Z_{J-C, Cz+h}}.
        \]
        Further, $\Cov\parens*{\mu_{J, h}} \psdle \Cov(\nu)$.
    \end{theorem}
    We will only use the HS transform with $C$ being a positive definite matrix.
    \begin{proof}
        Let $\bsigma\sim\mu_{J, h}$ and $\by\sim\calN(0,I)$ be independent, and define $\bz \coloneqq \bsigma + C^{-1/2} \by$.
        The joint probability density function $p$ of $(z,\sigma)$ satisfies:
        \begin{align*}
            p(z,\sigma) &\propto \exp\parens*{-\frac{1}{2}(z-\sigma)^{\top} C (z-\sigma) + \frac{1}{2}\sigma^{\top}J\sigma + \angles*{h,\sigma}} \\
            &= \exp\parens*{ -\frac{1}{2} z^{\top} C z + \angles*{Cz,\sigma} + \frac{1}{2}\sigma^{\top} (J-C) \sigma + \angles*{h,\sigma} }
        \end{align*}
        This tells us:
        \begin{align*}
            \mu_{J,h}(\sigma) &\propto \int_{\R^n} p(z,\sigma) dz \\
            &\propto \int_{\R^n} \exp\parens*{ -\frac{1}{2} z^{\top} C z + \angles*{Cz,\sigma} + \frac{1}{2}\sigma^{\top} (J-C) \sigma + \angles*{h,\sigma} } dz \\
            &= \int_{\R^n} \mu_{J-C, Cz+h}(\sigma) \cdot \exp\parens*{ -\frac{1}{2}z^{\top}Cz + \log Z_{J-C, Cz+h} } dz,
        \end{align*}
        which completes the proof of the first part.
        The PSD inequality $\Cov\parens*{\mu_{J, h}} \psdle \Cov(\nu)$ follows from writing $\Cov(\nu)$ as $\Cov\parens*{\mu_{J, h}} + C^{-1}$.
    \end{proof}

    The Brascamp--Lieb inequality is relevant to bounding the covariance matrix of distributions arising from the Hubbard--Stratonovich transform.
    \begin{theorem}[Brascamp--Lieb inequality] \label{thm:brascamp-lieb}
        For any distribution with density $\nu\propto\exp(-V)$ for some $V:\R^n\to\R$, if there is a matrix $\Gamma\psdg 0$ such that $\nabla^2 V(z)\psdge\Gamma$ uniformly, then
        \(
            \Cov(\nu) \psdle \Gamma^{-1}.
        \)
    \end{theorem}

    Using the Hubbard-Stratonovich transform with the Brascamp-Lieb inequality allows us to bound covariances using the following lemma.

    \begin{lemma}   \label{lem:interaction-sum-cov-bound}
        Let $M_1$ and $M_2$ be interaction matrices in $\R^{n\times n}$, with $M_1$ positive definite, and let $h$ be an external field in $\R^n$.
        Suppose for some positive definite matrix $\Gamma$, the following holds for all $z\in\R^n$:
        \[
            M_1 - M_1\cdot \Cov\parens*{\mu_{M_2, M_1z + h}} \cdot M_1 \psdge \Gamma.
        \]
        Then $\Cov\parens*{\mu_{M_1+M_2, h}} \psdle \Gamma^{-1}$.
    \end{lemma}
    \begin{proof}
        First, observe that by applying \Cref{thm:HS} with $C = M_1$, we obtain $\Cov\parens*{\mu_{M_1+M_2, h}} \psdle \Cov(\nu)$ where $\nu(z)\propto\exp\parens*{-\frac{1}{2}z^{\top}M_1z + \log Z_{M_2, M_1 z + h}}$.
        It then suffices to bound $\Cov(\nu)$.
        Towards doing so, we show that $\nu$ is log-concave, and then use the Brascamp--Lieb inequality (\Cref{thm:brascamp-lieb}).
    
        We can verify that:
        \[
            -\nabla^2\log\nu(z) = M_1 - \nabla^2 \log Z_{M_2, M_1z + h} = M_1 - M_1\cdot \Cov\parens*{\mu_{M_2, M_1z+h}} \cdot M_1.
        \]
        By assumption, for any $z\in\R^n$, $-\nabla^2\log\nu(z) \psdge \Gamma\psdg 0$, which establishes that $\nu$ is log-concave.
        Consequently, by \Cref{thm:brascamp-lieb} we get $\Cov(\nu)\psdle\Gamma^{-1}$.
    \end{proof}

    To bound the above quantity, the following lemma will be useful.
    
    \begin{fact}
        \label{prop:prop-hs-trick}
        Let $L,M$ be symmetric matrices such that $\|M\| \le \alpha$ and $0 \psdl \eta_1 \psdle L \psdle \frac{1-\eta_2}{\alpha}$. Then,
        \[ L - LML \psdge \eta_1\eta_2. \]
    \end{fact}

\subsection{Localization schemes}   \label{sec:loc-schemes}

    We will make use of \emph{stochastic localization} to relate our Markov chain to simpler Markov chains \cite{Eld13,CE22}.
    We use a restricted form of stochastic localization, that we state below.
    \begin{definition}[Stochastic localization for Ising models]    \label{def:stoc-loc}
        A \emph{stochastic localization scheme} is a stochastic process $\parens*{\mu_{t}}_{0\le t\le 1}$ on the space of Ising models that is parameterized by a (possibly time-varying) \emph{control matrix} $(C_t)_{0 \le t \le 1}$ and an \emph{initialization} $\mu_{J,h}$, described by the following stochastic differential equation.
        \begin{equation}
            \dif\parens*{\frac{\mu_{t}}{\mu_{J,h}}} = \frac{\mu_{t}}{\mu_{J,h}} \angles*{ x-\E_{\mu_{t}} x, C_t \cdot \dif B_t } \label{eq:stoc-loc}
        \end{equation}
        where $(B_t)_{0\le t\le 1}$ is the standard Brownian motion in $n$ dimensions.
    \end{definition}
    While \eqref{eq:stoc-loc} may appear rather mysterious at first, the measures $\mu_{t}$ take on a natural form.
    \begin{fact}[{\cite[Facts 13 and 14]{CE22}}]    \label{fact:stoc-loc-form}
        The distribution $\mu_t$ is the Ising model
        \(
            \mu_{J_t, h_t}
        \)
        where $J_t = J - \int_{0}^{t} C_s^2 \dif s$ and $h_t = \int_0^t C_s \cdot \dif B_s + C_s^2 \E_{\mu_s} [x] \cdot \dif s$.
        Further, $\mu_t$ is a martingale, in that for any set $A \subseteq \{\pm 1\}^n$, $\E[\mu_t(A)|\mu_s] = \mu_s(A)$ for $0\le s\le t\le 1$.
    \end{fact}

    The following is a mild generalization of \cite[Proposition 39 and Lemma 40]{CE22} that tells us entropy conservation for stochastic localization; this stronger version is implicitly proved in \cite[Lemma 40]{CE22}. 
    \begin{lemma}   \label{lem:entropic-stability}
        Consider a stochastic localization process $(\mu_t)_{t \ge 0}$ with respect to (possibly time-varying) control matrix $(C_t)_{t \ge 0}$, started at an Ising model $\mu_0 = \mu_{J,h}$. Suppose that for every $h'\in\R^n$ and $t \in [0,1]$,
        \[
            \norm*{C_t \cdot \Cov\parens*{\mu_{J_t, h'}} \cdot C_t} \le \alpha_t.
        \]
        Then, for any positive-valued $\varphi$,
        \[
            \E \left[\Ent_{\mu_{1}}[\varphi]\right] \ge \exp\parens*{-\int_{0}^{1} \alpha_t \dif t} \cdot \Ent_{\mu_{J,h}}[\varphi].
        \]
    \end{lemma}

    The following result articulates how the measure decomposition provided by stochastic localization is instrumental in proving a modified log-Sobolev inequality for the Glauber dynamics Markov chain.
    \begin{theorem}[{\cite[Theorem 47]{CE22}}]  \label{thm:anneal}
        Let $J$ be an interaction matrix and $h$ be an external field.
        Let $(\mu_t)_{0\le t\le 1}$ be a stochastic localization scheme initialized at $\mu_{J,h}$.
        Suppose that
        \begin{enumerate}
            \item \label{item:entropy-conservation} We have the entropy conservation inequality
            \[ \frac{\E[\Ent_{\mu_1}[f]]}{\Ent_{\mu_{J,h}}[f]} \ge \eps \]
            for all $f : \Omega \to \R_+$.
            \item \label{item:MLSI-target} Almost surely,
            \[ \MLSI(P_{\mu_1}) \ge \delta. \]
        \end{enumerate}
        Then, $\mu_{J,h}$ has an MLSI coefficient at least $\eps\delta$.
    \end{theorem}

    Using the above two results easily yields the following.

    \begin{theorem} \label{thm:sum-MLSI}
        Let $M_1,M_2$ be matrices in $\R^{n \times n}$ with $M_1$ positive semidefinite, and suppose that for every external field $h' \in \R^{n}$ and $t\in[0,1]$,
        \begin{enumerate}
            \item $\left\| M_1^{1/2} \cdot \Cov(\mu_{(1-t)M_1+M_2,h'}) \cdot M_1^{1/2} \right\| \le q$, and
            \item Glauber dynamics for $\mu_{M_2,h'}$ satisfies $\MLSI\parens*{\mu_{M_2,h'}} \ge \delta$,
        \end{enumerate}
        then, for any external field $h \in \R^{n}$,
        \[
            \MLSI\parens*{\mu_{M_1+M_2,h}} \ge \delta \cdot e^{-q}.    
        \]
    \end{theorem}
    
    \begin{proof}
        We will use \Cref{thm:anneal} by choosing the stochastic localization to be that with control matrix $M_1^{1/2}$.
        \noindent The desired statement follows once we verify that the following are true:
        \begin{enumerate}
            \item \label{item:entropy-con} $\displaystyle\frac{\E \left[\Ent_{\mu_1}[f]\right]}{\Ent_{\mu_{J,h}}[f]} \ge e^{-q}$ for every $f:\{\pm1\}^n\to\R_+$, and
            \item \label{item:final-MLSI} $\MLSI\parens*{\mu_1} \ge \delta$ almost surely.
        \end{enumerate}
        \Cref{item:entropy-con} follows from \Cref{lem:entropic-stability}, and \Cref{item:final-MLSI} is true by assumption, which completes the proof.
    \end{proof}

\subsection{Graphs} \label{sec:graph-prelims}

    For the rest of this section, let $G$ be a graph on $n$ vertices and $m$ edges, and let $c$ be a weight function on this graph.
    We will use $A_c$ to denote the adjacency matrix of $G$, weighted according to $c$, and $D_c$ as the weighted degree function where $D_c[i,i] = \sum_{j\ne i} A_c[i,j]$.
    \begin{definition}[Nonbacktracking matrix]
        The \emph{nonbacktracking matrix} $B_c$ is a $2m\times 2m$ matrix indexed by directed edges in $G$ where:
        \[
            B_{c}[uv, xy] =
            \begin{cases}
                c(xy) &\text{ if }v=x\text{ and }u\ne y \\
                0 &\text{otherwise.}
            \end{cases}
        \]
        We will use $B_G$ to refer to the \emph{unweighted} nonbacktracking matrix, i.e., when $c(xy)$ is equal to $1$ for every edge in $G$.
    \end{definition}

    The nonbacktracking matrix is related to a matrix known as the Bethe Hessian, and we use this relationship to obtain information about the spectrum of the adjacency matrices of bounded degree subgraphs of a random graph arising from the stochastic block model.
    \begin{definition}[Bethe Hessian]
        \label{def:bethe-hessian}
        The \emph{Bethe Hessian} is defined as:
        \(
            \BH_{c} \coloneqq D_{c\wh{c}} - A_{\wh{c}} + I
        \)
        where $\wh{c} = \frac{c}{1-c^2}$.
        When all the weights are equal to $t$, the Bethe Hessian can be written as:
        \[
            \BH_G(t) \coloneqq \frac{(D_G-I)t^2 - At + I}{1-t^2}.
        \]
    \end{definition}

    The following is a classic statement that relates the eigenvalues of the Bethe Hessian and those of the nonbacktracking matrix; see, e.g., \cite[Proof of Theorem 5.1]{FM17}).
    \begin{lemma}
        For any $\alpha\in(0,1)$, the spectral radius $\rho(B_G) \le \frac{1}{\alpha}$ if and only if $\BH_G(t)\psdg 0$ for all $t\in(-\alpha,\alpha)$.
    \end{lemma}

    A simple observation is that when $G$ is a tree, its nonbacktracking matrix is nilpotent, and hence its spectral radius is equal to $0$, which gives us the following consequence.
    \begin{corollary}
        \label{cor:bethe-tree-psd}
        When $G$ is a tree, for all $t\in(-1,1)$, we have $\BH_G(t)\psdg 0$.
    \end{corollary}

    We define $A_G^{(s)}$ as the $n\times n$ matrix where the $i,j$ entry contains the number of length-$s$ nonbacktracking walks between $i$ and $j$.
    The following is a classic fact; see, e.g.,~\cite[Equation (2.1)]{ST96}.

    \begin{fact}
        \label{fact:bethe-inverse}
        $\displaystyle \BH_G(t)^{-1} = \sum_{s\ge 0} A_G^{(s)}t^s$.
    \end{fact}

    We will need to understand the spectrum of the nonbacktracking matrix of a random graph drawn from the stochastic block model, as well as random graphs with randomly signed edge weights.
    For the sequel, let $\bG\sim\SBM(n,d,\lambda)$ with hidden community vector $\bsigma\in\{\pm1\}^n$, and assume $d > 1$.
    Consider the adjacency matrix centered after conditioning on $\bsigma$, i.e., the matrix $\ol{A}_{\bG} \coloneqq A_{\bG} - \E[A_{\bG}|\bsigma]$.
    Let $\bc$ be the weight function on the edges of the complete graph prescribed by $\ol{A}_{\bG}$.

    The following bounds on the nonbacktracking matrix of a random graph drawn from the stochastic block model are effectively due to \cite{BLM15}, but the exact form we will need is the version in \cite{LMR22}.
    \begin{lemma}[Consequence of {\cite[Theorem 7.4]{LMR22}}]
        For $\eps > 0$:
        \[
            |\lambda|_{\max}(B_{\bc}) \le (1+\eps)\cdot\sqrt{d}
        \]
        with probability $1-o_n(1)$.
    \end{lemma}

    By an identical argument to the proof of \cite[Theorem 5.1]{FM17}, the following is true.
    \begin{lemma}
        For any $t\in\parens*{-\frac{1}{\sqrt{d}},\frac{1}{\sqrt{d}}}$,
        \(
            \BH_{{\bc}\cdot t} \psdge 0
        \)
        with probability $1-o_n(1)$.
    \end{lemma}

    Recalling that the centered adjacency matrix is given by $\ol{A}_{\bG} = A_{\bG} - \E A_{\bG} | \bsigma$, and using the fact that the entries of $\E A_{\bG}|\bsigma$ are of magnitude $O\parens*{\frac{d}{n}}$, and the fact that all degrees in $\bG$ are at most $O(\log n)$ with high probability, we get the following relationship.
    \begin{corollary}   \label{cor:bethe-bounds}
        For any $t\in\parens*{-\frac{1}{\sqrt{d}}, \frac{1}{\sqrt{d}}}$, we get that with probability $1-o_n(1)$:
        \[
            \parens*{D_{\bG}-I}t^2 - A_{\bG}t + \E[A_{\bG}|\bsigma] \cdot t(1-t^2) + I (1 + o(1)) \psdge 0.
        \]
        Further, by using the fact that $\norm*{\E[A_{\bG}|\bsigma}] \le d\cdot(1+o(1))$, we get:
        \[
            \wt{\BH}_{\bG}(t) \coloneqq \parens*{D_{\bG}-I}t^2 - \ol{A}_{\bG} t + I \cdot\parens*{1 + d|t|^3 + o(1)} \psdge 0\mper
        \]
    \end{corollary}

    We are finally ready to prove the main result about the spectrum of the bulk in our decomposition.
    \begin{lemma}
        \label{lem:bulk-spectral-radius}
        Let $S\subseteq[n]$ such that the degree of every vertex in $S$ within $\bG$ is at most $(1+\eps)d$, and suppose that $\sqrt{d(1+\eps)} + \frac{1}{1+\eps} \le \sqrt{d}(1+\eps)$.
        Then
        \[
            \norm*{\parens*{\ol{A}_{\bG}}_S} \le 2\sqrt{d}(1+\eps)
        \]
        with probability $1-o_n(1)$.
        Consequently, the spectral diameter of $\parens*{\ol{A}_{\bG}}_S$ is at most $4\sqrt{d}(1+\eps)$.
    \end{lemma}
    \begin{proof}
        We shall use \Cref{cor:bethe-bounds}. Set $t = \frac{1}{(1+\eps)\sqrt{d}}$. Observing that any principal submatrix of the PSD matrix $\wt{\BH}_{\bG}(t) \psdge 0$ is also PSD, and $\|(D_G)_S\| \le d(1+\eps)$, we get that
        \begin{align*}
            \lambda_{\mathrm{max}}\left((\ol{A}_{\bG})_S\right) &\le \frac{1 + dt^3 - t^2 + d(1+\eps)t^2}{t} \\
                &\le \sqrt{d}(1+\eps) + \frac{1}{1+\eps} + \sqrt{d(1+\eps)} \le 2\sqrt{d}(1+\eps).
        \end{align*}
        Doing the same for the matrix $\wt{\BH}_{\bG}(-t) \psdge 0$,
        \begin{align*}
            \lambda_{\mathrm{min}}\left((\ol{A}_{\bG})_S\right) &\ge \frac{1 + dt^3 - t^2 + d(1+\eps)t^2}{-t} \\
                &\ge - \sqrt{d}(1+\eps) - \frac{1}{1+\eps} - \sqrt{d(1+\eps)} \ge -2\sqrt{d}(1+\eps). \qedhere
        \end{align*}
    \end{proof}
    Similarly, there are well-established bounds on the spectral radius of the nonbacktracking matrix of randomly signed \erdos--\renyi graphs.
    \begin{lemma}[{Consequence of \cite[Theorem 2]{SM22}}]
        Let $\bG\sim \SBM(n, d, \lambda)$ and let $\bc$ be a random assignment of $\pm1$-valued signs to the edges of $\bG$.
        Then, with probability $1-o_n(1)$,
        \[
            |\lambda|_{\max}\parens*{B_{\bc}} \le (1+o_n(1))\cdot\sqrt{d}\mper
        \]
    \end{lemma}
    An identical argument to the proof of \Cref{lem:bulk-spectral-radius} can be used to establish the following.
    \begin{lemma}   \label{lem:bulk-specrad-signed}
        Let $\bG\sim \SBM(n, d, \lambda)$ and let $\bc$ be a random assignment of $\pm1$-valued signs to the edges of $\bG$. With probability $1-o_n(1)$, the following holds.
        Let $S\subseteq[n]$ such that the degree of every vertex in $S$ within $\bG$ is at most $(1+\eps)d$.
        Then,
        \[
            \norm*{A_{\bc}} \le 2\sqrt{d}(1+\eps).
        \]
    \end{lemma}
    \begin{remark}
        Notice that one difference between \Cref{lem:bulk-spectral-radius} and \Cref{lem:bulk-specrad-signed} is that the latter does not require a constraint on $d$ and $\eps$.
        This owes to the fact that in the former case there are nonzero entries both at the scale of $\Theta(1)$ and $\Theta(1/n)$ whereas all nonzero entries in the latter case are in $\{\pm1\}$.
    \end{remark}


\section{Glauber dynamics on near-forest decompositions}
\label{sec:spinglass}

This section is dedicated to proving the following result.

\begin{restatable}{theorem}{spinglassmixing}
    \label{thm:main-mixing}
    Fix $d > 1$. Let $A$ be the adjacency matrix of a randomly signed graph $\bG \sim \SBM(n, d, \lambda)$. There exists a universal constant $\beta > 0$ such that with high probability over the draw of $\bG$, for any external field $h$, we have $\MLSI(\mu_{\beta A/\sqrt{d}, h}) \ge n^{-1 - o(1)}$. 
\end{restatable}
    
Given a sparse random graph $\bG$, our goal is to show that for any external field $h$, Glauber dynamics for the randomly signed Ising model on $\bG$ mixes rapidly in the appropriate temperature regime.

In \Cref{sec:decomposition}, we describe the graph decomposition, and state several results that will play a part in the proof. In \Cref{sec:spinglass-mlsi}, before proving any of these results, we begin by putting them together to prove an MLSI --- the localization scheme we use will be described therein. In \Cref{sec:covariance-near-forest}, we bound the covariances of Ising models supported on near-forests using the trace method --- these bounds are used in \Cref{sec:entropy-conservation-spinglass}, where we establish a covariance bound used in showing that entropy is conserved along the localization path. In \Cref{sec:mlsi-tree-ising}, we prove a modified log-Sobolev inequality for Ising models supported on near-forests, which we are left with at the end of the localization scheme.


\subsection{On decomposing random graphs}
\label{sec:decomposition}
    In this section, we describe an explicit class of decompositions of an Ising Hamiltonian into the sum of two Hamiltonians which are supported on two disjoint sets of edges, which we term the \emph{bulk} and the \emph{near-forest} components. 
    We then use a localization scheme to anneal the original Ising model into an Ising model with interactions supported only on the near-forest component.
    We will verify the two conditions needed to apply \cref{thm:sum-MLSI} to imply rapid mixing of Glauber dynamics on sparse random graphs.

    \begin{definition}[Near-forest]
        We say a graph $H$ is a \emph{near-forest} if all its connected components have at most a single cycle. Equivalently, each component has tree-excess at most $1$.
    \end{definition}

    A crucial ingredient in our proof is that sparse stochastic block model graphs can be decomposed into the union of two graphs, a \emph{bulk} and a \emph{near-forest}, where the bulk has bounded degree and the near-forest satisfies some pseudorandom properties.
    The following definitions capture the pseudorandom property the near-forest satisfies and the precise form of the decomposition.
    \begin{definition}[Pseudorandom graph]\label{def:pseudorandom-graph}
        We say a graph $H$ is \emph{$(\Delta,D)$-pseudorandom} with respect to $S \subseteq V(H)$ if it satisfies the following:
        \begin{enumerate}
            \item For any $r\ge 1$, and any vertex $u\in V(H)$, the size of the $r$-ball around $u$ is at most $D^{r-1}\cdot\Delta$. Note in particular that this implies that the maximum degree is at most $\Delta$.
            \item For any $r\ge 0$ and $u\in V(H)$, $\abs*{\braces*{v\in S: \dist(u,v)\le r}} \le D^r$. 
        \end{enumerate}
    \end{definition}
    While the above statement bounds the sizes of balls, most of our proofs will only need bounds on the sizes of neighborhoods. We are now ready to formally state our graph decomposition.
    \begin{definition}[Graph decomposition]
        For a graph $G$, let $B, H$ be subgraphs of $G$ with possibly overlapping sets of vertices, which we refer to as the \emph{bulk} and \emph{near-forest} components.
        We say that $(B, H)$ is a \emph{$(\Delta, D)$-decomposition} of $G$ if the following properties hold.
        \begin{enumerate}
            \item $E(G) = E(B) \sqcup E(H)$, i.e. the bulk and near-forest components partition the edges of $G$. 
            \item Let $\partial H = V'(B) \cap V'(H)$ denote the vertex boundary of $H$, the vertices in $H$ incident to an edge in $E(B)$. Then $H$ is a $(\Delta, D)$-pseudorandom near-forest with respect to $\partial H$.
        \end{enumerate}

        If such a $(B, H)$ exists, we say that $G$ admits a $(\Delta, D)$-decomposition.
    \end{definition}
    
    Many of our results hold generically assuming that $G$ admits a $(\Delta, D)$-decomposition (or is itself a $(\Delta, D)$-pseudorandom near-forest);  we refer the readers back to \Cref{fig:graph-decomposition} for an illustration of such a decomposition. In \Cref{lem:random-graph-decomposition}, we show that sparse stochastic block model graphs admit such a decomposition.

    As a crucial first step towards analyzing the class of Ising models supported on such graphs, we establish the following covariance bounds for \emph{any} Ising model with bounded interactions supported on a pseudorandom near-forest; we defer the proof to \Cref{sec:covariance-near-forest}. 
    \begin{restatable}{lemma}{nearforestcov}
        \label{lem:near-forest-cov}
        Let $H$ be a $(\Delta,D)$-pseudorandom near-forest with respect to $S \subseteq V(H)$.
        For any Ising model with interactions $J$ supported on $E(H)$ with values in $\bracks*{-\frac{\gamma}{\sqrt{D}}, \frac{\gamma}{\sqrt{D}}}$, along with an arbitrary external field $h$, we have
        \[
            \norm*{\Cov\parens*{\mu_{J,h}}_S} \le \frac{e^{2\gamma/\sqrt{D}}}{(1-\gamma)^2}.
        \]
        Additionally,
        \[
            \norm*{\Cov\parens*{\mu_{J,h}}} \le \frac{e^{2\gamma/\sqrt{D}}}{(1-\gamma)^2} \cdot \frac{\Delta}{D}.
        \]
    \end{restatable}

    As a consequence of \Cref{lem:near-forest-cov}, we can establish conservation of entropy and a good MLSI constant for the annealed model. We prove these implications in \Cref{sec:entropy-conservation-spinglass,sec:mlsi-tree-ising}, respectively.

    \begin{restatable}{theorem}{entropyconservation}
        \label{thm:entropy-conservation}
        Let $(B, H)$ be a $(\Delta, D)$-decomposition of $G$, and let $J_B, J_H \in \R^{n \times n}$ be interaction matrices whose off-diagonal entries are supported on $E(B)$ and $E(H)$, respectively. Suppose that there are constants $\eta, K \in (0, 1)$ such that $\eta I_B \psdle J_B \psdle K I_B$, and that all interactions in $J_H$ are in $\left[ - \frac{\gamma}{\sqrt{D}} , \frac{\gamma}{\sqrt{D}} \right]$ for some $\gamma \in [0, 1)$. Finally, suppose that 
        \[
            K \cdot \frac{e^{2\gamma/\sqrt{D}}}{(1-\gamma)^2} \le 1 - \delta
        \]
        for a positive constant $\delta > 0$.
        There exists a positive constant $C$ not depending on $\Delta$ such that for any $0 \le t \le 1$ and external field $h \in \R^{n}$,
        \[ \Cov(\mu_{(1-t)J_B + J_H,h}) \psdle C (I_B + \Delta I_{H \setminus \partial H}). \]
    \end{restatable}
    
    \begin{remark}
        The assumptions on the spectrum of $J_B$ may look artificial, but they are in fact without loss of generality. Indeed, by \cref{obs:diag-invariant}, one can shift $J_B$ by a constant multiple of $I_B$ to ensure it is PSD lower bounded by $\eta$ when restricted to the bulk. 
    \end{remark}
    \begin{restatable}{theorem}{mlsitree}
        \label{thm:mlsi-bdd-growth-tree}
        Let $H$ be a $(\Delta,D)$-pseudorandom near-forest, and $J \in \R^{n \times n}$ an interaction matrix supported on $H$ with all interactions in $\left[ - \frac{\gamma}{\sqrt{D}} , \frac{\gamma}{\sqrt{D}} \right]$ for some $\gamma \in [0,1)$. There is a constant $c$ (depending on $\gamma$) such that for any external field $h \in \R^{n}$,
        \[ \MLSI(\mu_{J,h}) \ge \frac{1}{ne^{c\Delta}} . \]
    \end{restatable}

    To apply \Cref{thm:entropy-conservation,thm:mlsi-bdd-growth-tree} for concrete graphs to prove rapid mixing results, we need to establish that these graphs admit a $(\Delta, D)$-decomposition.

    For sparse random graphs, we explicitly construct the near-forest decomposition as outlined in \Cref{sec:overview} and obtain the following explicit bounds on the parameters $\Delta$ and $D$; the analysis is given in \Cref{sec:extinction}.
    \begin{restatable}{theorem}{randomgraphdecomp}   \label{lem:random-graph-decomposition}
        Let $\bG\sim\SBM(n, d, \lambda)$ for any $d\ge 1$.
        There exist universal constants $C_1$ and $C_2$ such that for any $\eps \ge C_1 \cdot \left( \frac{C_2 + \log d}{d} \right)^{1/3}$, with high probability over $\bG$, we can partition the edge set $E(\bG)$ as $E(\bG_1) \sqcup E(\bG_2)$ for graphs $\bG_1$ and $\bG_2$ on $[n]$ where:
        \begin{enumerate}
            \item \label{item:max-degree-bound} The maximum degree of $\bG_1$ is at most $(1+\eps)d$.
            \item \label{item:near-forest-property} $\bG_2$ is a $\parens*{\Delta, D}$-pseudorandom near-forest with respect to $S\coloneqq V'(\bG_1)\cap V'(\bG_2)$ for $\Delta = o(\log n)$ and $D = (1+\eps)d$.
        \end{enumerate}
    \end{restatable}
    
    \begin{remark}
        We will henceforth refer to $\bG_1$ as the \emph{bulk} part in the decomposition, and to $\bG_2$ as the \emph{near-forest} part of the decomposition.
    \end{remark}

    \begin{remark}
        The bound on the maximum degree of $\bG_1$ is important for obtaining a bound on the spectral norm of the interaction matrix restricted to the edges in $\bG_1$.
        The bound on $\Delta$ of $o(\log n)$ is crucial to obtain a bound of $n^{1+o(1)}$ on the mixing time via \Cref{thm:mlsi-bdd-growth-tree}.
    \end{remark}

\subsection{Proving an MLSI for diluted spin glasses}
\label{sec:spinglass-mlsi}

    In this section, we apply \Cref{thm:entropy-conservation,thm:mlsi-bdd-growth-tree,lem:random-graph-decomposition} to obtain a lower bound on the MLSI constant for Glauber dynamics on randomly signed sparse random graphs. We restate the theorem here for convenience.

    \spinglassmixing*
    \begin{proof}
        Consider the bulk $\bG_1$ and near-forest $\bG_2$ guaranteed by \Cref{lem:random-graph-decomposition}, and let $J_B = \frac{\beta}{\sqrt{d}} A_{\bG_1}$ and $J_H = \frac{\beta}{\sqrt{d}} A_{\bG_2}$. 
        We will verify the conditions of \Cref{thm:entropy-conservation} for some $\beta = \Omega(1)$. We want to verify that 
        \begin{align}
            K \cdot \frac{e^{2\gamma/\sqrt{D}}}{(1-\gamma)^2} \le 1-\delta 
            \label{eq:conservation-condition}
        \end{align}
        
        Before we dive into the calculations, let us first observe that from \Cref{lem:random-graph-decomposition}, we have that $\bG_2$ is a $(o(\log n),d(1+\eps))$-pseudorandom near-forest. 
        Thus, the interaction strengths on $J_H$ are in $\left[ -\frac{\beta}{\sqrt{d}} , \frac{\beta}{\sqrt{d}} \right] = \left[ - \frac{\beta\sqrt{1+\eps}}{\sqrt{d(1+\eps)}} , \frac{\beta\sqrt{1+\eps}}{\sqrt{d(1+\eps)}} \right]$.
        We will thus check \eqref{eq:conservation-condition} for $\gamma = \beta\sqrt{1+\eps}$. We have
        \[ \frac{e^{2\gamma/\sqrt{D}}}{(1-\gamma)^2} = \frac{e^{2\beta/\sqrt{d}}}{(1-\beta\sqrt{1+\eps})^2}. \]
    
        Turning now to $K$, \Cref{lem:bulk-specrad-signed} furnishes the bound
        \[
            -2\beta(1+\eps) I_B \psdle J_B \psdle 2\beta(1+\eps) I_B.
        \]
        For some ``small'' $\eta > 0$ that we fix later, we can define $\wt{J}_B = J_B + (\eta + 2\beta(1+\eps))I_B$ without changing the Ising model to ensure
        \[
            \eta I_B \psdle \wt{J}_B \psdle \left(\eta + 4\beta(1+\eps)\right) I_B.
        \] 
        Thus, we choose $K = \eta + 4\beta(1+\eps)$. Hence, the LHS of \eqref{eq:conservation-condition} is at most
        \begin{align*}
            K \cdot \frac{e^{2\beta/\sqrt{d}}}{(1-\beta\sqrt{1+\eps})^2} \le e^{2\beta/\sqrt{d}} \cdot \frac{\eta + 4\beta(1+\eps)}{(1-\beta(1+\eps))^2}.
        \end{align*}
        If we have
        \begin{align*}
            \beta < \min\left\{ 1 , \frac{1+2e^{2/\sqrt{d}} - \sqrt{4e^{4/\sqrt{d}} + 4e^{2/\sqrt{d}}}}{1+\eps} \right\} = (3-\sqrt{8})(1 - o_d(1)) \approx 0.17 \cdot (1-o_d(1)),
        \end{align*}
        sufficiently small constants $\eta$ and $\delta$ may be chosen so the above is less than $1-\delta$.
    
        Now, we can apply \Cref{thm:entropy-conservation} to show that 
        \[ 
            \Cov(\mu_{(1-t)J_B + J_H,h}) \psdle C (I_B + o(\log n) \cdot I_{H \setminus \partial H}), 
        \]
        and in turn 
        \[
            \exp(-\norm{\wt{J}_B^{1/2} \cdot \Cov(\mu_{(1-t)J_B + J_H,h}) \cdot \wt{J}_B^{1/2}}) \ge \Omega(1).
        \]
        On the other hand, \Cref{thm:mlsi-bdd-growth-tree} shows that for any $h'$ we have 
        \[
            \MLSI(\mu_{J_H, h'}) \ge \frac{1}{n} \cdot e^{-o(\log n)} = n^{-1-o(1)}.
        \] 
        Finally, applying \Cref{thm:sum-MLSI} with $M_1 = \wt{J}_B$ and $M_2 = J_H$ allows us to conclude that $\MLSI(\mu_{\beta A/\sqrt{d}, h}) \ge n^{-1 - o(1)}$.
    \end{proof}
    Combined with \Cref{fact:mixing-mlsi}, we immediately conclude \Cref{th:main}.

    \subsection{Covariance bounds on near-forests}\label{sec:covariance-near-forest}
    In this section we establish \Cref{lem:near-forest-cov}, which bounds the covariance of an Ising model supported on a pseudorandom near-forest.
    \nearforestcov*
    Our strategy to prove \Cref{lem:near-forest-cov} is:
    \begin{enumerate}
        \item Reduce to proving the covariance norm bound in the case when $H$ is a tree, the interactions are nonnegative, and there is no external field.
        \item Use an explicit formula for the covariance in tree Ising models with a zero external field, and bound the spectral norm of that matrix using the trace moment method.
    \end{enumerate}
    The first ingredient in this reduction is the following simple observation.
    \begin{observation} \label{obs:sign-conjugation}
        For any interaction matrix $J$, external field $h$, and diagonal matrix of signs $D$, we have:
        \[
            \norm*{\Cov\parens*{\mu_{J,h}}} = \norm*{\Cov\parens*{\mu_{DJD, Dh}}}\mper
        \]
        In fact, $\MLSI(\mu_{J,h}) = \MLSI(\mu_{DJD,Dh})$.
    \end{observation}
    \begin{proof}
        Note that if $\sigma\sim\mu_{J,h}$, the vector $D\sigma$ is distributed as $\mu_{DJD,Dh}$.
        Thus, $\Cov\parens*{\mu_{DJD,Dh}} = D\cdot\Cov\parens*{\mu_{J,h}}\cdot D$. The first observation now follows since $D$ is an orthogonal matrix. For the second part, we have that for any function $f : \{\pm 1\}^n \to \R_{>0}$, if we define $g : \{\pm 1\}^n \to \R_{>0}$ by $g(\sigma) = f(D\sigma)$, then the entropies and Dirichlet forms of $f$ and $g$ with respect to $\mu_{J,h}$ and $\mu_{DJD,Dh}$ (respectively) are equal.
    \end{proof}

    We now use the above observation to pass from an Ising model on a forest with arbitrary interactions to one with nonnegative interactions on all but one edge. 
    \begin{lemma}   \label{lem:all-but-one-neg}
        For any connected graph $H$ and spanning tree $T$ of $H$, interaction matrix $J$ supported on $H$, and an external field $h$, there is a diagonal matrix $D$ such that $DJD$ has nonnegative interactions on the edges of $T$.
    \end{lemma}
    \begin{proof}
        Arbitrarily root $T$ at a vertex $r$.
        Construct $D$ by choosing $D_{v,v} = \prod_{e'\in P} \sign(J_{e'})$ where $P$ is the unique path from $r$ to $v$ in $T$.
        For any edge $uv\in E(T)$, $(DJD)_{u,v} = J_{u,v} D_{u,u} D_{v,v} = J_{u,v} \sign\parens*{J_{u,v}} \ge 0$.
    \end{proof}

    We will need the following fact about ferromagnetic Ising models with external fields.

    \begin{lemma}[{\cite[Corollary 1.3]{DSS23}}]   \label{lem:cov-ferro-external}
        For any ferromagnetic interaction matrix $J\in\R_{\ge 0}^{n\times n}$ and external field $h\in\R^n$, we have for any $u,v\in[n]$:
        \[
            \Cov\parens*{\mu_{J,h}}_{u,v} \le \Cov\parens*{\mu_{J,0}}_{u,v}.
        \]
    \end{lemma}

    \begin{fact}
        \label{fact:ferro-nonneg-correlation}
        For any ferromagnetic interaction matrix $J \in \R_{\ge 0}^{n \times n}$, external field $h \in \R^n$, and $u,v \in [n]$, $\Cov\left( \mu_{J,h} \right)_{u,v} \ge 0$.
    \end{fact}
    \begin{proof}
        Ferromagnetic Ising models with arbitrary external fields satisfy the hypothesis of the FKG inequality \cite[Theorem 2.16]{gri06}.
        The statement then follows from the FKG inequality applied to the functions $\sigma_i - \E[\sigma_i]$ and $\sigma_j - \E[\sigma_j]$.
    \end{proof}

    \begin{corollary}
        \label{cor:cov-ferro-external-psd}
        Let $M$ be a matrix with all non-negative entries. For any ferromagnetic interaction matrix $J\in\R_{\ge 0}^{n\times n}$ and external field $h\in\R^n$, we have
        \[
            \left\|\Cov\parens*{\mu_{J,h}} \cdot M\right\| \le \left\|\Cov\parens*{\mu_{J,0}} \cdot M\right\|.
        \]
        In particular,
        \[ \left\|\Cov(\mu_{J,h})\right\| \le \left\|\Cov(\mu_{J,0})\right\|. \]
    \end{corollary}
    \begin{proof}
        This is immediate by \Cref{fact:ferro-nonneg-correlation,lem:cov-ferro-external}, and the Perron-Frobenius theorem.
    \end{proof}

    Next, we prove a bound on the operator norm of matrices that encode when two vertices are a certain distance apart. 

    \begin{lemma}
        \label{lem:trace-method-best-method}
        Let $H$ be a $(\Delta,D)$-pseudorandom tree with respect to $S \subseteq V(H)$. Let $\ell \ge 1$, and set $A^{(\ell)} = A^{(\ell)}_H$ to be the matrix with $(A^{(\ell)})_{uv} = 1$ if vertices $u$ and $v$ are distance exactly $\ell$ apart, and $0$ otherwise. Also let $A = A^{(1)}$ be the adjacency matrix of $H$. Then,
        \begin{align*}
            \left\| A^{(\ell)} \right\| &\le D^{\ell/2} \cdot (\ell+1) \cdot \left( \frac{\Delta}{D} \right) \\
            \left\| A^{(\ell)}_S \right\| &\le D^{\ell/2} \cdot (\ell+1) \\
            \left\| A^{(\ell)} A \right\| &\le D^{\ell/2} \cdot (\ell+1) \cdot \left( \frac{2\Delta}{\sqrt{D}} \right) \mper
        \end{align*}
    \end{lemma}
    \begin{proof}
        We shall bound these operator norms using the trace moment method. Let $R$ be $S$ or $V(H)$, so we are interested in bounding $\left\|A^{(\ell)}_R\right\|$. Let us define $\calW_{\ell,m,R}$ to be the set of all closed walks in $H$ of length $\ell\cdot m$ that can be expressed as the concatenation of $m$ nonbacktracking walks\footnote{In a tree, nonbacktracking walks are equivalent to shortest paths.} between a pair of vertices in $R$ of length-$\ell$ each.
        Henceforth, we refer to each of these length-$\ell$ nonbacktracking walks as a \emph{linkage}.
        For any even integer $m > 0$, we have:
        \[ \left\| A^{(\ell)}_R \right\|^m \le \Tr\parens*{ \parens*{A^{(\ell)}_R}^m } = \abs*{\calW_{\ell,m,R}}. \]
        For each vertex $r\in R$, we count the number of walks that start and end at $r$.
        The $t$-th linkage, if started at a vertex $v$, is composed of $s_t$ steps ``upwards'' away from $v$, followed by $\ell-s_t$ steps ``downwards''.
        Since $H$ is a tree, and the walks we are counting start and end at the same vertex, exactly half of the steps in any walk are upwards, and exactly half are downwards, that is,
        $$
            \sum_{t=1}^m s_t = \sum_{t=1}^m \ell-s_t = \ell\cdot m / 2\mper
        $$
        Given $s_t$, the number of choices for the $t$-th linkage is at most $D^{\ell-s_t-1} \cdot \Delta_R$ where $\Delta_R$ is $D$ when $R = S$ and $\Delta$ when $R = V(H)$.
        \[
            \abs*{\calW_{\ell,m,R}} \le |R| \cdot \sum_{s_1,\dots,s_{m}}  \prod_{t=1}^m D^{\ell-s_t}\cdot\frac{\Delta_R}{D} = |R| \cdot \sum_{s_1,\dots,s_m} D^{\ell\cdot m / 2} \cdot \parens*{\frac{\Delta_R}{D}}^m \le |R| \cdot D^{\ell\cdot m / 2} \cdot \parens*{\frac{\Delta_R}{D}}^m \cdot (\ell+1)^m \mper
        \]
        Thus, taking $m \to \infty$ so $|R|^{1/m} \to 1$, we get that
        \[ \norm*{ A^{(\ell)}_R } \le D^{\ell/2} \cdot \left( \frac{\Delta_R}{D} \right) \cdot (\ell+1). \]
        The first two inequalities follow by plugging in $\Delta_S = D$ and $\Delta_{V(H)} = \Delta$.\\

        The proof of the third inequality is very similar.
        The main difference is that each of the linkages is not a length-$\ell$ nonbacktracking walk, but instead a length-$(\ell+1)$ walk, of which the first $\ell$ steps are non-backtracking. That is, the $t$-th ``linkage'' is now composed of $s_t$ steps upwards, followed by $\ell-s_t$ steps downwards, followed by a single step upwards or downwards.
        If the final step is downwards, this is just $s_t$ steps upwards then $\ell - s_t + 1$ steps downwards. 
        If the final step is upwards, then this can be thought of as changing $s_{t+1}$ to $s_{t+1} + 1$.
        We encode this with $\eps_t, \eps'_t \in \{0,1\}$, wherein the $t$-th linkage consists of $s_t + \eps_t$ steps upwards followed by $\ell - s_t + \eps'_t$ steps downwards. Since $\eps'_t+\eps_{t+1}=1$, the sequence $(\eps'_t)$ is determined by $(\eps_t)$, and $\sum \eps_t + \eps'_t = m$. Consequently,
        \[
            \sum_{t} s_t + \eps_t = \sum_{t} \ell - s_t + \eps'_t = \frac{(\ell+1)m}{2}\mper
        \]
        Setting $\mathcal{W}'_{\ell,m}$ as the set of walks enumerated by this trace moment method calculation,\footnote{In this case, $R$ is fixed to be $V(H)$.} we have
        \begin{align*}
            |\mathcal{W}'_{\ell,m}| &\le V(H) \cdot \sum_{s_1,\ldots,s_m,\eps_1,\ldots,\eps_m} \prod_{t=1}^{m} D^{\ell - s_t + \eps'_t} \cdot \frac{\Delta}{D} \\
                &= V(H) \cdot \sum_{s_1,\ldots,s_m,\eps_1,\ldots,\eps_m} D^{(\ell+1)m/2} \cdot \left(\frac{\Delta}{D}\right)^m \\
                &\le V(H) \cdot (\ell+1)^m \cdot 2^m \cdot D^{(\ell+1)m/2} \cdot \left( \frac{\Delta}{D} \right)^m.
        \end{align*}
        Again, taking $m\to\infty$ so $V(H)^{1/m} \to 1$, we get that 
        \[ \left\| A^{(\ell)} A \right\| \le D^{(\ell+1)/2} \cdot \frac{\Delta}{D} \cdot 2 (\ell+1) \]
        as claimed.
    \end{proof}

    As our final ingredient, we now prove a version of \Cref{lem:near-forest-cov} for trees with positive interactions and no external field.

    \begin{lemma} \label{lem:tree-cov}
        Let $H$ be a $(\Delta,D)$-pseudorandom tree with respect to $S\subseteq V(H)$. Let $A$ be the adjacency matrix of $H$.
        If $J$ is a collection of interactions supported on the edges of $H$ with values in $\bracks*{0, \frac{\gamma}{\sqrt{D}}}$ for some $\gamma \in [0,1)$, then
        \begin{align*}
            \norm*{\Cov\parens*{\mu_{J,0}}_S} &\le \frac{1}{(1-\gamma)^2} \\
            \norm*{\Cov\parens*{\mu_{J,0}}} &\le \frac{1}{(1-\gamma)^2}\cdot\frac{\Delta}{D} \\
            \norm*{\Cov\left(\mu_{J,0}\right) \cdot A} &\le \frac{1}{(1-\gamma)^2} \cdot \frac{2\Delta}{\sqrt{D}} \mper
        \end{align*}
    \end{lemma}
    While only the first two of these bounds are needed in the current section, the third will come in handy in \Cref{sec:mlsi-tree-ising} when we show an MLSI for the Ising model obtained after running stochastic localization, which is supported on a pseudorandom near-forest.
    \begin{proof}
        We start with the first two bounds, using the following explicit formula for the covariance matrix.
        \[
            \Cov\parens*{\mu_{J,0}} = \sum_{\ell\ge 0} X^{(\ell)},
        \]
        where $X^{(\ell)}_{i,j}$ is $0$ if there is no length-$\ell$ path between $i$ and $j$ in $H$, and is otherwise equal to $\prod_{e\in P(i,j)} \tanh(J_e)$ where $P(i,j)$ is the unique such path.
        For $R\subseteq V(H)$, will use the bound:
        \begin{equation}
            \label{eq:trace-moment-cov-bound-ez}
            \norm*{\Cov\parens*{\mu_{J,0}}_R} \le \sum_{\ell\ge 0} \norm*{X^{(\ell)}_R}.
        \end{equation}
        Consider the matrix $Y^{(\ell)}$ obtained by replacing each nonzero entry of $X^{(\ell)}$ with $\parens*{\frac{\gamma}{\sqrt{D}}}^{\ell}$.
        Since the interactions are all nonnegative, $X^{(\ell)}$ is a nonnegative matrix, and $Y^{(\ell)}$ is entrywise at least $X^{(\ell)}$ since $\tanh(x) \le x$ for $x \ge 0$. 
        Consequently,
        \begin{equation}
            \label{eq:trace-moment-termwise-bound-ez}
            \norm*{X^{(\ell)}_R} \le \norm*{Y^{(\ell)}_R}.
        \end{equation}
        \Cref{lem:trace-method-best-method} now implies that
        \[ \norm*{Y^{(\ell)}_R} \le \left( \frac{\gamma}{\sqrt{D}} \right)^{\ell} \cdot D^{\ell/2} \cdot (\ell+1) \cdot \left(\frac{\Delta_R}{D}\right) = \gamma^{\ell} \cdot (\ell+1) \cdot \left(\frac{\Delta_R}{D}\right). \]
        Therefore,
        \[
            \norm*{\Cov\parens*{\mu_{J,0}}_R} \le \frac{\Delta_R}{D} \cdot \sum_{\ell \ge 0} (\ell+1) \gamma^{\ell} \le \frac{1}{(1-\gamma)^2}\cdot\frac{\Delta_R}{D}\mcom
        \]
        which recovers the desired bounds by plugging in $\Delta_R = \Delta$ for $R = V(H)$, and $\Delta_R = D$ for $R=S$.\\

        For the third inequality, we again have
        \[ \Cov\left(\mu_{J,0}\right)\cdot A = \sum_{\ell \ge 0} \norm*{X^{(\ell)} \cdot A}. \]
        Again, the non-negativity of $X^{(\ell)}$ and $A$ implies that 
        \[ \norm*{X^{(\ell)} \cdot A} \le \norm*{Y^{(\ell)} \cdot A}. \]
        The third bound in \Cref{lem:trace-method-best-method} now implies that
        \begin{align*}
            \norm*{Y^{(\ell)} \cdot A} &\le \left(\frac{\gamma}{\sqrt{D}}\right)^{\ell} D^{(\ell+1)/2} \cdot (\ell+1) \cdot 2 \cdot \left( \frac{\Delta}{D} \right) \\
                &= \gamma^{\ell} \cdot (\ell+1) \cdot 2 \cdot \frac{\Delta}{\sqrt{D}}.
        \end{align*}
        Therefore,
        \[ \norm*{\Cov\left(\mu_{J,0}\right) \cdot A} \le \sum_{\ell \ge 0} \gamma^{\ell} \cdot (\ell+1) \cdot 2 \cdot \frac{\Delta}{\sqrt{D}} \le \frac{1}{(1-\gamma)^2} \cdot \frac{2\Delta}{\sqrt{D}}, \]
        completing the proof.
    \end{proof}

    We are now ready to prove \Cref{lem:near-forest-cov}.
    \begin{proof}[Proof of \Cref{lem:near-forest-cov}]
        The matrix $\Cov\parens*{\mu_{J,h}}$ is a block diagonal matrix with a block per connected component of $H$, and so without loss of generality, it suffices to achieve a bound on the spectral norm in the case where $H$ is a connected near-forest (a near-tree) --- a tree with possibly one extra edge inserted. If $H$ has tree-excess $0$, the desideratum is trivial by \Cref{fact:ferro-nonneg-correlation,lem:cov-ferro-external,lem:tree-cov}.
        Hence, let us assume that $H$ has tree-excess $1$.

        Let $T$ be an arbitrary ($(\Delta,D)$-pseudorandom) spanning tree of $H$, such that the edge set of $H$ consists of those in $T$ along with an extra edge $ij$. By \Cref{obs:sign-conjugation} and \Cref{lem:all-but-one-neg}, it suffices to consider the case where $J$ has nonnegative interactions on all edges in $T$, and possibly a negative interaction on $ij$. Also let $J'$ be the interaction matrix which is identical to $J$ except that $J'_{ij} =J'_{ji} = 0$. Note that $\left|\sigma^\top J \sigma - \sigma^\top J' \sigma\right| \le \frac{\gamma}{\sqrt{D}}$ for any $\sigma \in \{\pm 1\}^n$. Let $R$ be equal to $S$ or $V(H)$, with $\Delta_S = D$ and $\Delta_{V(H)} = \Delta$, and set $v$ to be a unit vector that witnesses the operator norm of $\Cov(\mu_{J,h})_{R}$. Then,
        \begin{align*}
            \|\Cov\left( \mu_{J,h} \right)_R\| &= v^\top \Cov\left( \mu_{J,h} \right) v  \\
            &= \Var_{\sigma \sim \mu_{J,h}}(\langle v,\sigma\rangle) \tag*{(\text{\Cref{fact:var-cov}})} \\
            &\le e^{2\gamma/\sqrt{D}} \cdot \Var_{\sigma \sim \mu_{J',h}} (\langle v,\sigma\rangle) \tag*{(\text{\Cref{lem:mlsi-bdd-density-relate}})} \\
            &\le e^{2\gamma/\sqrt{D}} \cdot \| \Cov\left( \mu_{J',h} \right)_R \| \tag*{(\text{\Cref{fact:var-cov}})} \\
            &\le e^{2\gamma/\sqrt{D}} \cdot \| \Cov\left( \mu_{J',0} \right)_R \| \tag*{(\text{\Cref{cor:cov-ferro-external-psd}})} \\
            &\le \frac{e^{2\gamma/\sqrt{D}}}{(1-\gamma)^2} \cdot \frac{\Delta_R}{D} \tag*{(\text{\Cref{lem:tree-cov}})}
        \end{align*}
        as desired.
    \end{proof}

\subsection{Entropy conservation for near-forest decompositions}
\label{sec:entropy-conservation-spinglass}
 
    In this section, we shall bound the covariance matrix of the distribution obtained along the localization path for Ising models supported on a graph which admits a $(\Delta, D)$-decomposition, towards satisfying the first requirement for applying \Cref{thm:sum-MLSI}. 

    Let us now prove the main result of this section, which is restated here for convenience.
    \entropyconservation*
    \begin{proof}
        Let $\rho > 0$ be a positive constant that we fix later.
        We use \Cref{lem:interaction-sum-cov-bound}, with the instantiation
        \[ M_1 = (1-t) J_B + \frac{t}{2}I_B +  \frac{\rho}{\Delta} I_{H \setminus \partial H} \text{ and } M_2 = J_H. \]
        Since $\eta I_B \psdle (1-t) J_B + \frac{t}{2}I_B \psdle K I_B$, we can reduce to the case where $t=0$.
        Now, set $\delta' = \delta/2$.
        For $t=0$ and an arbitrary external field $h'$, we have
        \begin{align*}
            M_1 \cdot \Cov(\mu_{M_2,h'}) \cdot M_1 &= \left( J_B + \frac{\rho}{\Delta} I_{H \setminus \partial H} \right) \cdot \Cov(\mu_{M_2,h'}) \cdot \left( J_B + \frac{\rho}{\Delta} I_{H \setminus \partial H} \right) \\
                &\psdle \left( 1 + \delta' \right) \cdot J_B \cdot \Cov(\mu_{M_2,h'}) \cdot J_B + \left( 1 + \frac{1}{\delta'} \right) \cdot \frac{\rho^2}{\Delta^2} I_{H \setminus \partial H} \cdot \Cov(\mu_{M_2,h'}) \cdot I_{H \setminus \partial H} \\
                &= \left(1 + \delta'\right) J_B \cdot \Cov(\mu_{M_2,h'}) \cdot J_B + \left(1 + \frac{1}{\delta'}\right) \cdot \frac{\rho^2}{\Delta^2} \Cov(\mu_{M_2,h'})_{H \setminus \partial H}.
        \end{align*}
        Above, the inequality follows from the observation that for symmetric matrices $A_1,A_2$, a positive semidefinite matrix $M$, and any $\alpha > 0$,
        \[ (A_1 + A_2) M (A_1 + A_2) \psdle \left(1+\alpha\right) A_1 M A_1 + \left(1 + \frac{1}{\alpha}\right) A_2 M A_2. \]
        We now make the crucial observation that we can restrict the covariance to the bulk in the first term, that is, 
        \[ J_B \cdot \Cov(\mu_{M_2,h'}) \cdot J_B = J_B \cdot \Cov(\mu_{M_2,h'})_B \cdot J_B. \]
        Indeed, observe that $(J_B)_{ij}$ must be zero for any $i,j$ where $i$ or $j$ lies in $H \setminus \partial H$, the interior of the near-forest. In particular, any such $i$ or $j$ is not in the bulk.

        We are interested in showing the positive definiteness of
        \[ 
            \left(J_B - (1+\delta') \cdot J_B \cdot \Cov(\mu_{M_2,h'})_B \cdot J_B\right) + \frac{\rho}{\Delta}\left( I_{H \setminus \partial H} - \left( 1 + \frac{1}{\delta'}\right) \cdot \frac{\rho}{\Delta} \cdot \Cov(\mu_{M_2,h'})_{H \setminus \partial H} \right). 
        \]
        Note that had we not added the extra $(\rho/\Delta) I_H$ when defining $M_1$, the second term above would be absent, making the matrix entries zero on the rows and columns $V'(H) \setminus V'(B)$, which precludes a meaningful lower bound by a full-rank matrix.\\

        By assumption, $M_2$ is supported on a $(\Delta, D)$-pseudorandom near-forest, and the interaction strengths are in $\left[ -\frac{\gamma}{\sqrt{D}} , \frac{\gamma}{\sqrt{D}} \right]$. Furthermore, since $\gamma < 1$, the first part of \Cref{lem:near-forest-cov} yields
        \begin{align*}
            \left\| \Cov\left( \mu_{M_2,h'} \right)_B \right\| \le \frac{e^{2\gamma/\sqrt{D}}}{(1-\gamma)^2} \eqdef f(\gamma).
        \end{align*}

        Consequently, \Cref{prop:prop-hs-trick} and the assumption on $\gamma$ shows that
        \begin{align*}
             J_B - \left( 1 + \delta' \right) \cdot J_B \cdot \Cov(\mu_{M_2,h'}) \cdot J_B &\psdge \eta \cdot \left(1 - \left( 1 + \delta' \right) K f(\gamma)\right) I_B  \\
             &\psdge \eta \cdot \delta' \cdot I_B.
        \end{align*}

        Similarly, the second part of \Cref{lem:near-forest-cov} shows that
        \[ 
            \left\|\Cov(\mu_{M_2,h'})_{H \setminus \partial H}\right\| \le \left\|\Cov(\mu_{M_2,h'})_{H}\right\| \le \frac{\Delta}{D} \cdot f(\gamma).
        \]
        
        Selecting $\rho \coloneqq \frac{D}{2 f(\gamma) \cdot \left( 1 + \frac{1}{\delta'} \right)}$, we see that
        \[ 
            I_{H \setminus \partial H} - \left(1 + \frac{1}{\delta'}\right) \cdot \frac{\rho}{\Delta} \cdot \Cov(\mu_{M_2,h'})_{H \setminus \partial H} \psdge \frac{1}{2} I_{H \setminus \partial H}.
         \]
        
        Thus, we have that
        \begin{align*}
            M_1 - M_1 \cdot \Cov(\mu_{M_2,h'}) \cdot M_1 &\psdge \eta \delta' \cdot  I_B + \frac{\rho}{2\Delta} \cdot I_{H \setminus \partial H} \\
            &\psdge c\left(I_B + \frac{1}{\Delta}I_{H \setminus \partial H}\right),
        \end{align*}
        for some universal constant $c$ (that does not depend on $\Delta$).
        The claim follows by applying \Cref{lem:interaction-sum-cov-bound} and setting $C = c^{-1}$. 
    \end{proof}

\subsection{MLSI for Ising models on pseudorandom near-forests}
\label{sec:mlsi-tree-ising}

    Our goal in this section will be to prove a modified log-Sobolev inequality for the Ising model obtained after stochastic localization of the original Ising model, which we recall for convenience.

    \mlsitree*

    \Cref{fact:mlsi-product} implies that it suffices to prove an MLSI of $\frac{1}{|C| e^{\Delta}}$ for each component $C$ of this near-forest, a ``near-tree''. By a simple comparison argument, it in fact suffices to prove an MLSI for pseudorandom trees with ferromagnetic interactions --- we shall elaborate on this in the proof.

    \begin{restatable}{lemma}{pseudorandomferromixing}
        \label{lem:pseudorandom-ferro-mixing}
         Let $H$ be a $(\Delta,D)$-pseudorandom tree on $n$ vertices, and $J$ an interaction matrix supported on $H$ with all interactions in $\left[0, \frac{\gamma}{\sqrt{D}}\right]$ for some $\gamma \in [0,1)$.
        There is a constant $c$ (depending on $\gamma$) such that for any external field $h \in \R^n$,
        \[ \MLSI(\mu_{J,h}) \ge \frac{1}{ne^{c\Delta}}\mper \]
    \end{restatable}
    
    \begin{remark}
        In fact, our proof works in a slightly broader range of the interactions, only requiring that for some $\gamma \in [0,1)$, all the interactions are bounded (in absolute value) by $\tanh^{-1} \left( \frac{\gamma}{\sqrt{D}} \right)$. We conjecture that this is optimal. 
    \end{remark}
    
    In this section, we only provide a proof of the above in the setting where all the nonzero interactions are exactly equal to $\frac{\gamma}{\sqrt{D}}A$. The more general proof is provided in \Cref{appendix-a}.
    
    \begin{proof}
        Let $H$ be a tree as in the lemma statement, $A$ its adjacency matrix, and $J = \frac{\gamma}{\sqrt{D}}$ the corresponding interaction matrix. Also set $D_H$ be the (diagonal) degree matrix of $H$.

        To prove this, we shall once again use \Cref{lem:entropic-stability}. The control matrix will be of the form $C = (J+E)^{1/2} \psdge 0$ for a non-negative diagonal matrix $E$. Due to the ferromagneticity of the system, $C^2$ has all non-negative entries. At time $0 \le t \le 1$, $J_t$ will have the form $(1-t) J - tE$ -- in particular, at time $t=1$, $\mu_t$ is almost surely a product distribution so $\MLSI(\mu_t) = \Omega(1/n)$. In the context of \Cref{lem:entropic-stability}, our goal is to bound, for all external fields $h \in \R^n$, the operator norm
        \begin{equation}
            \label{eq:bethe-hessian-coincidence-1}
            \norm*{C \cdot \Cov\left( \mu_{(1-t)J,h} \right) \cdot C} \le \norm*{\Cov\left( \mu_{(1-t)J,h} \right) \cdot C^2 }.
        \end{equation}
        By \Cref{cor:cov-ferro-external-psd} (using the matrix $C^2$), it suffices to bound the above for $h = 0$. For $h = 0$, we recall that if $A^{(\ell)}$ is the matrix where the $uv$-th entry is $1$ if $u$ and $v$ are distance exactly $\ell$ apart and $0$ otherwise,
        \[ \Cov\left( \mu_{(1-t)J,0} \right) \le \sum_{\ell \ge 0} \left( \frac{\gamma}{\sqrt{D}} \right)^{\ell} A^{(\ell)} \]
        in the sense that the matrix on the right is entry-wise at least the matrix on the left.
        In particular,
        \begin{equation}
            \label{eq:bethe-hessian-coincidence-2}
            \norm*{\Cov\left( \mu_{(1-t)J,h} \right) \cdot C^2 } \le \norm*{ \sum_{\ell \ge 0} \left( \frac{\gamma}{\sqrt{D}} \right)^{\ell} A^{(\ell)} \cdot C^2 }.
        \end{equation}
        For ease of notation, denote $s_0 = \frac{\gamma}{\sqrt{D}}$. Recalling \Cref{def:bethe-hessian}, set
        \[ \BH(s) = \frac{(D_H - I) s^2 - As + I}{1-s^2}, \]
        so by \Cref{fact:bethe-inverse},
        \[ \BH\left( s_0 \right)^{-1} = \sum_{\ell \ge 0} \left( \frac{\gamma}{\sqrt{D}} \right)^\ell A^{(\ell)}. \]
        We choose
        \begin{align}
            \label{eq:control-matrix-tree-mlsi}
            C = \left((1-s_0^2) \cdot \BH\left( - s_0 \right)\right)^{1/2} .
        \end{align}
        Now,
        \begin{align*}
            \sum_{\ell \ge 0} \left( \frac{\gamma}{\sqrt{D}} \right)^{\ell} A^{(\ell)} \cdot C^2 &= (1-s_0^2) \cdot \BH\left( s_0 \right)^{-1} \cdot \BH\left( -s_0 \right) \\
                &= (1-s_0^2) \cdot \BH(s_0)^{-1} \cdot \left( \BH(s_0) + \frac{2s_0}{1-s_0^2} \cdot A \right) \\
                &= (1-s_0^2) \cdot I + 2 s_0 \cdot \BH(s_0)^{-1} \cdot A.
        \end{align*}
        Using the third bound in \Cref{lem:tree-cov} with \eqref{eq:bethe-hessian-coincidence-1} and \eqref{eq:bethe-hessian-coincidence-2}, we get the desired claim that
        \begin{align*}
            \norm*{C \cdot \Cov\left( \mu_{(1-t)J,h} \right) \cdot C} &\le \norm*{ (1-s_0^2) \cdot I + 2 s_0 \cdot \BH(s_0)^{-1} \cdot A } \\
                &\le 1 + \frac{8}{(1-\gamma)^2} \cdot \frac{\Delta}{D}.
        \end{align*}
        Given that the product measure obtained at time $1$ satisfies an MLSI with constant $1/n$, it follows by \Cref{lem:entropic-stability,thm:anneal} that
        \begin{align*}
            \MLSI(\mu_{J,h}) &\ge \frac{1}{n} \cdot \exp\left( - \int_{0}^{1} \sup_{h \in \R^n} \norm*{ C \cdot \Cov\left( \mu_{(1-t)J, h} \right) \cdot C } \,dt \right) \\
                &\ge \frac{1}{n} \cdot \exp\left( - \left( 1 + \frac{8}{(1-\gamma)^2} \cdot \frac{\Delta}{D} \right) \right) = \frac{1}{ne^{O(\Delta)}}
        \end{align*}
        as desired.
    \end{proof}

    Before moving on, we make a handful of remarks about the above proof, specifically the choice of control matrix in \eqref{eq:control-matrix-tree-mlsi}:
    \begin{itemize}
        \item The choice is well-defined due to \Cref{cor:bethe-tree-psd}, which implies that $\BH\left(-s_0\right) \psdg 0$.
        \item It is important that we provide a negative argument to $\BH$ so that $C^2$ is equal to some positive scaling of the interaction matrix (which in turn is a scaling of the adjacency matrix because all the nonzero interactions are equal) plus a diagonal matrix -- this is required for stochastic localization to eventually kill the interaction matrix. Note that as a consequence of this, $C^2$ has all non-negative entries as desired.
        \item It is also important that the precise negative argument we provide is equal to $-s_0$ in order to be able to express $C^2$ as $\BH(s_0)$ plus some scaling of the adjacency matrix -- otherwise, there would also be a term corresponding to the degree matrix in the definition of the Bethe Hessian. This would cause issues in the methods we use to bound the operator norm. 
    \end{itemize}

    The above lemma easily yields a similar result even when interactions are allowed to be negative, as long as their absolute values satisfy the same bound.

    \begin{corollary}
        \label{lem:pseudorandom-nonferro-mixing}
         Let $H$ be a $(\Delta,D)$-pseudorandom tree on $n$ vertices, and $J$ an interaction matrix supported on $H$ with all interactions in $\left[- \frac{\gamma}{\sqrt{D}}, \frac{\gamma}{\sqrt{D}}\right]$.
        For any external field $h$, 
        \[ \MLSI(\mu_{J,h}) \ge \frac{1}{ne^{O(\Delta)}}. \]
    \end{corollary}
    \begin{proof}
        This immediately follows using \Cref{lem:pseudorandom-ferro-mixing} in conjunction with \Cref{obs:sign-conjugation} and \Cref{lem:all-but-one-neg}, since $H$ is a tree.
    \end{proof}

    Given the above, \Cref{thm:mlsi-bdd-growth-tree} is near-immediate using the Holley-Stroock perturbation principle \Cref{lem:mlsi-bdd-density-relate}.

    \begin{proof}[Proof of \Cref{thm:mlsi-bdd-growth-tree}]
        As remarked earlier, by \Cref{fact:mlsi-product}, it suffices to prove an MLSI of $\Omega\left(\frac{1}{|C|e^{\Delta}}\right)$ for each component $C$ of $H$, which is a near-tree. Let $\mu$ be the corresponding Ising model on this component, with interaction matrix $\restr{J}{C}$ and external field $\restr{h}{C}$. Let $T$ be an arbitrary spanning tree of $C$ (which must also be $(\Delta,D)$-pseudorandom), and $\nu$ the Ising model supported on $T$ with the same interaction strengths as in $\mu$. Observe that because $C$ has at most one edge in addition to those in $T$, and the interaction strengths are bounded by $\frac{\gamma}{\sqrt{D}}$ in absolute value, we have
        \[ e^{-\frac{2\gamma}{\sqrt{D}}} \le \frac{\mu(\sigma)}{\nu(\sigma)} \le e^{\frac{2\gamma}{\sqrt{D}}}. \]
        Furthermore, $\nu$ satisfies an MLSI by \Cref{lem:pseudorandom-nonferro-mixing}, so \Cref{lem:mlsi-bdd-density-relate} completes the proof.
    \end{proof}



\section{Glauber dynamics for centered stochastic block models}\label{sec:sbm-results}

In this section, we establish the following result proving an MLSI for Ising models where the interactions are proportional to the centered adjacency matrix $\ol{A}_{\bG}$ for a sparse random graph $\bG$.

\begin{restatable}{theorem}{centeredmixing}
    \label{thm:main-centered-mixing}
    Fix constants $d > 1, \lambda \geq 0$. There exists a universal constant $\beta > 0$ such that the following is true. Let $\ol{A}_{\bG} \coloneqq A_{\bG} - \E[A_{\bG}|\bsigma]$ be the centered adjacency matrix of $\bG \sim \SBM(n, d, \lambda)$. With high probability over the draw of $\bG$, for any external field $h$, $\MLSI(\mu_{(\beta/\sqrt{d})\ol{A}_{\bG}, h}) \ge n^{-1-o_d(1)}$. 
\end{restatable}
Given this, \Cref{fact:mixing-mlsi} immediately yields \Cref{thm:main-centered}.

Due to the presence of the weak (but nonzero) centering interactions, the analysis turns out to be significantly more delicate than that in \Cref{sec:spinglass}. To prove an MLSI, one still uses \Cref{thm:sum-MLSI}, where the initial localization scheme anneals away the bulk interactions.
However, the annealed Ising model now has small but nonzero interactions that are supported beyond the near forest $H$. 

Thus, we need to prove analogous versions of \Cref{thm:entropy-conservation,thm:mlsi-bdd-growth-tree} for entropy conservation and an MLSI for the annealed model.
We state the results below and prove them in subsequent sections.

\begin{restatable}{theorem}{centeredadjacency}
    \label{thm:centered-adjacency}
    Let $\bG \sim \SBM(n, d, \lambda)$ for any constant $d\ge 1$, and let $B = \bG_1$ and $H = \bG_2$ be the bulk component and near-forest component respectively that are guaranteed by \Cref{lem:random-graph-decomposition}.     
    There exists a constant $\beta^* > 0$ such that with high probability over the draw of $\bG$ and any external field $h$, the following holds for all $\beta < \beta^*$. Set $J_B = \frac{\beta}{\sqrt{d}}(A_{\bG_1} - \E[A| \bsigma]_{\bG_1})$ and $J_R = \frac{\beta}{\sqrt{d}}\ol{A}_{\bG} - J_B$. There exists a positive constant $C$ such that for any $0 \le t \le 1$ and external field $h \in \R^n$,
    \[ \Cov(\mu_{(1-t)J_B+ J_R,h}) \psdle C(I_{B} + \Delta I_{H \setminus \partial H}). \]
\end{restatable}

\begin{restatable}{theorem}{nearforestjunk}
    \label{thm:near-forest-junk-mlsi}
    Let $J_R \in \R^{n \times n}$ be the interaction matrix for the perturbed near-forest as defined in \Cref{thm:centered-adjacency}. Then, for any external field $h$,
    \[
        \MLSI(\mu_{J_R, h}) \ge \frac{1}{n^{1+o_d(1)}}\mper
    \]
\end{restatable}


Given the above lemmas, the proof of \Cref{thm:main-centered-mixing} is exactly like that of \Cref{thm:main-mixing}.

\begin{proof}[Proof of \Cref{thm:main-centered-mixing}]
    By \Cref{thm:centered-adjacency}, for sufficiently low inverse temperature $\beta$, it holds with high probability that all $0 \le t \le 1$ and external fields $h \in \R^n$,
    \[ \Cov(\mu_{(1-t)J_B+ J_R,h}) \psdle C(I_B + o(\log n) \cdot I_{H \setminus \partial H}). \]
    With $\wt{J}_B = J_B + 2\beta(1+\eps) I_B$,
    \Cref{lem:bulk-spectral-radius} implies that 
    \[ 0 \psdle \wt{J}_B \psdle 4\beta(1+\eps) I_B. \]
    Thus,
    \[ \exp \left( - \left\| \wt{J}_B^{1/2} \cdot \Cov\left( \mu_{(1-t)J_B + J_R, h} \right) \cdot \wt{J}_B^{1/2} \right\| \right) = \Omega(1). \]
    On the other hand, \Cref{thm:near-forest-junk-mlsi} says that for all $h$,
    \[ \MLSI\left( \mu_{J_R,h} \right) \ge \frac{1}{n^{1+o_d(1)}}. \]
    \Cref{thm:sum-MLSI} completes the proof.
\end{proof}

\subsection{Technical overview for centered stochastic block models}

We now elucidate the high-level strategy for proving \Cref{thm:centered-adjacency,thm:near-forest-junk-mlsi}. 
The annealed Ising model is $\mu_{J_R, h}$ for some external field $h$; here $J_R$ stands for the remainder of the interactions outside of the bulk. 
We will refer to $J_R$ as the \emph{perturbed near-forest} interactions, and $\mu_{J_R, h}$ as the perturbed near-forest Ising model.
Let us introduce some notation to separate out the different pieces of $J_R$. 

Let $H' = V'(H)$ denote the set of non-isolated vertices in $H$, and let $B' \coloneqq V(B) \setminus H'$ denote the set of vertices on the interior of the bulk (we slightly abuse notation here for notational simplicity).
Furthermore, for any two matrices $M, N \times \R^{n \times n}$, the Hadamard product $M \circ N \in \R^{n \times n}$ denotes the entrywise product of $M$ and $N$. 

We can decompose $J_R = M_H + E_H + E$, with the latter defined explicitly as 
\begin{align}
    J_R &= \underbrace{\frac{\beta}{\sqrt{d}}(\ol{A}_{H} \circ A_H)}_{M_{H}} - \underbrace{\frac{\beta}{\sqrt{d}}\E[A|\bsigma]_{H} \circ (\bone_{H'} \bone_{H'}^\top - A_H)}_{E_H} - \underbrace{\frac{\beta}{\sqrt{d}} \left(\E[A|\bsigma]_{H', B'} + \E[A|\bsigma]_{B', H'}\right) }_{E} \label{eq:new-Jb},
\end{align}
where
\[ \E[A|\bsigma]_{H', B'} = \frac{d}{n}\left(\bone_{H'} \bone_{B'}^\top + \frac{\lambda}{\sqrt{d}}\bsigma_{H'}\bsigma_{B'}^\top\right) = \E[A|\bsigma]_{B', H'}^\top. \]
When $\bG$ is a sparse \erdos-\renyi graph (instead of the more general stochastic block model),
\begin{itemize}
    \item $\frac{\sqrt{d}}{\beta} \cdot M_H$ is equal to $\left( 1 - \frac{d}{n} \right)$ on the edges within the near-forest $H$.
    \item $\frac{\sqrt{d}}{\beta} \cdot E_H$ is equal to $\frac{d}{n}$ on all the non-edges in the near-forest $H$.
    \item $\frac{\sqrt{d}}{\beta} \cdot E$ is equal to $\frac{d}{n}$ on all pairs between the near-forest $H$ and the interior of the bulk $B'$.
\end{itemize}
In particular, $J_R$ is equal to zero on the $B' \times B'$ pairs.
 
Note that if only $M_H$ were present, then we would be back in the setting of \Cref{lem:near-forest-cov}.
Accordingly, we shall show that the perturbation terms $E$ and $E_H$ do not affect the covariances or MLSIs by too much.

Crucially, observe that the perturbation interactions are $O\left(\frac{d}{n}\right)$, and nonzero if and only if the edge is incident to a vertex of $H$. 
In particular, for any isolated vertex, corresponding to a vertex in $B'$, the perturbation vanishes. 
Moreover, we show that the size of each connected component is $o(n)$, which means that the perturbation interactions are genuinely small relative to the scale of connected components. 
This motivates introducing a Markov chain whose states correspond to the nontrivial connected components $C_i$ of $H$ with size strictly greater than $1$.

Since the perturbation interactions are small, we can show that the spectral norm version of Dobrushin's condition \cite{Hay06} is satisfied for this Markov chain.
In the end, this allows us to pass to the simpler Ising model supported on a single connected component $C_i$, where we can apply our results from \Cref{sec:decomposition}.

As in the technical overview, we outline the structure of our proof in \Cref{fig:sbm-flowchart}. The new elements in the proof compared to the simpler spin glass setting are on the bottom right of the chart --- a bound on a certain Dobrushin influence matrix is the primary new tool that is required.

\begin{figure}
\resizebox{\linewidth}{!}{
\begin{tikzpicture}
\node (P01) [smallrect] at (-2, 2.5) {\small\begin{tabular}{c}\Cref{thm:main-centered-mixing}\end{tabular}};
\node (P00) [smallrect] at (-2, 4) {\small\begin{tabular}{c}\Cref{thm:main-centered}\end{tabular}};

\node (P10) [bigrect] at (5, 5cm) {\small\begin{tabular}{c}\Cref{lem:random-graph-decomposition}\\Graph decomposition\\of stochastic block\\models (\Cref{fig:graph-decomposition})\end{tabular}};
\node (P11) [bigrect] at (5, 2.5cm) {\small\begin{tabular}{c}Conservation of\\entropy\end{tabular}};
\node (P12) [bigrect] at (5, 0cm) {\small\begin{tabular}{c}\Cref{thm:near-forest-junk-mlsi}\\MLSI bound for\\perturbed near-forest\end{tabular}};

\node (P20) [mediumrect] at (11,5cm) {\small\begin{tabular}{c}\Cref{lem:bulk-spectral-radius}\\Boundedness of\\control matrix\\from bulk\end{tabular}};
\node (P21) [mediumrect] at (11,2.5cm) {\small\begin{tabular}{c}\Cref{thm:centered-adjacency}\\Covariance bound\\along path\end{tabular}};

\node (P32) [bigrect] at (15,0cm) {\small\begin{tabular}{c}\Cref{cor:dobrushin}\\Bound on influence\\due to perturbations\end{tabular}};
\node (P33) [bigrect] at (15,-2.5cm) {\small\begin{tabular}{c}\Cref{thm:mlsi-bdd-growth-tree}\\MLSI bound for\\near-forest\end{tabular}};

\node (P31) [mediumrect] at (18,2.5cm) {\small\begin{tabular}{c}Proof of \Cref{thm:centered-adjacency}\\Covariance bound\\for perturbed near-forest\end{tabular}};

\node (P41) [bigrect] at (21,0cm) {\small\begin{tabular}{c}\Cref{lem:near-forest-cov}\\Covariance bound\\for near-forest\end{tabular}};

\draw[-{Latex[length=2mm]}] (P01) -- (P00);

\coordinate (Qf) at ([xshift=+2.75cm]P01.east); 
\draw (P10) .. controls (2,5) and (3,2.5) .. (Qf) -- (Qf); 
\draw (P11) -- (Qf);
\draw (P12) .. controls (2,0) and (3,2.5) .. (Qf) -- (Qf);
\draw[-{Latex[length=2mm]}] (Qf) -- node[above,midway] {\small  \Cref{thm:sum-MLSI}} node[below,midway] {\small Localization} (P01);

\coordinate (Qf3) at ([xshift=-0.5cm,yshift=+0.25cm]P20.west);
\coordinate (Qf4) at ([yshift=+0.25cm]P20.west);
\coordinate (Qf5) at ([yshift=+0.25cm]P10.east);
\draw[-{Latex[length=2mm]}] (Qf5) -- (Qf4);

\coordinate (Qf2) at ([xshift=+0.5cm]P11.east);
\coordinate (P20temp) at ([yshift=-0.25cm]P20.west);
\draw (P20temp) .. controls (8,4.75) and (8,2.5) .. (Qf2) -- (Qf2);
\draw (P21) -- (Qf2);
\draw[-{Latex[length=2mm]}] (Qf2) -- (P11);

\coordinate (Qf6) at ([xshift=-4cm]P32.west);
\draw[-{Latex[length=2mm]}] (Qf6) -- (P12);
\draw (P32) -- node[above,midway] {\small\Cref{lem:parameter-dep}} (Qf6);
\draw (P33) .. controls (10,-2.5) and (10,0) .. (Qf6) -- (Qf6);

\coordinate (Qf7) at ([yshift=-0.5cm]P31.south);
\draw[-{Latex[length=2mm]}] (Qf7) -- (P31);
\draw (P32) .. controls (15,1) and (18,1) .. (Qf7) -- (Qf7);
\draw (P41) .. controls (21,1) and (18,1) .. (Qf7) -- (Qf7);

\draw[-{Latex[length=2mm]}] (P31) -- node[below,midway] {\small HS transform} node[above,midway] {\small\Cref{lem:interaction-sum-cov-bound}} (P21);

\end{tikzpicture}
}
\caption{A flow chart outlining our proof of an MLSI for the centered adjacency matrix. The main difference in the structure of the proof compared to \Cref{fig:spinglass-flowchart} are the three boxes on the bottom right, which handle the interactions present on nonedges of the graph due to centering.}
\label{fig:sbm-flowchart}
\end{figure}
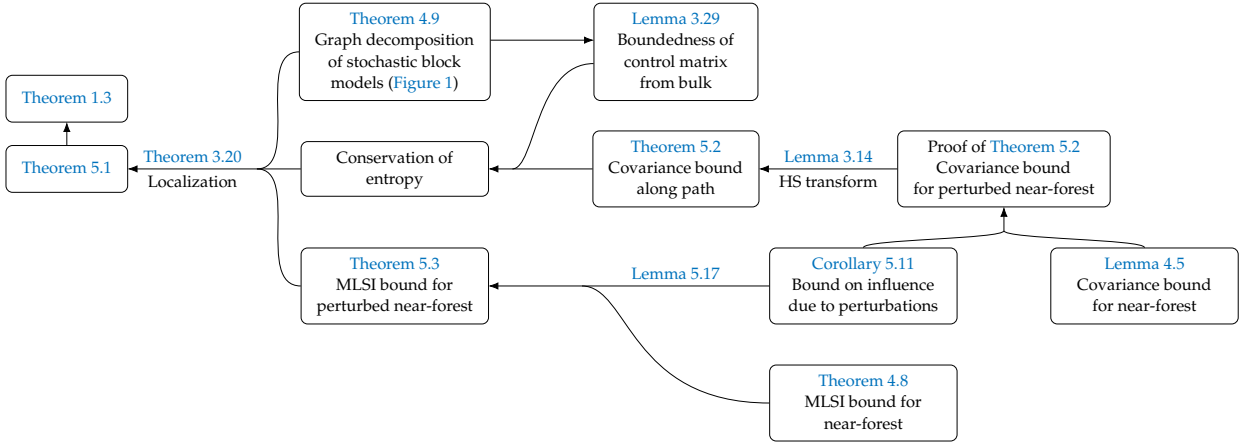

To rigorously execute the above plan, we will need some additional machinery.
We introduce these tools in \Cref{sec:perturbed-near-forest} to prove \Cref{thm:near-forest-junk-mlsi}. 
Using the same tools, we then prove \Cref{thm:centered-adjacency} in \Cref{sec:centered-adjacency}.

\subsection{A crash course on spectral independence}

We first introduce some additional tools that we use to establish the MLSI for the perturbed near-forest. 
These ingredients are stated somewhat abstractly, but we will next point out what the definitions correspond to in our application and tie it all together when we apply them to sparse random graphs.

\parhead{Influence matrices and spectral independence.}
We work in the following general setup.
Let $\nu$ be a distribution on a generic product domain $\calP = \calP_1\times\dots\times\calP_k$.
A key object of interest associated to $\nu$ is its \emph{Dobrushin influence matrix}, which is very standard in the literature on mixing times of Glauber dynamics.
\begin{definition}[Dobrushin Influence matrix \cite{Dob68, BD97}]
    The \emph{Dobrushin influence matrix} of $\nu$, denoted $\dobinf_{\nu}$, is a $k\times k$ matrix where:
    \[
        \dobinf_{\nu}[i,j] \coloneqq \max_{\substack{\sigma_{-i}, \tau_{-i} \\ \text{that agree on coordinates }\\ [k]\setminus \{i,j\} }} d_{\mathrm{TV}}\parens*{ \nu|\sigma_{-i}, \nu|\tau_{-i} }\mper
    \]
\end{definition}

The following was proved in \cite{AJKPV24}, with a slightly weaker version presented in \cite{Liu21,BCCPSV22}.
\begin{theorem}[{\cite[Theorem 43]{AJKPV24}}]    \label{thm:dobrushin-to-specind}
    If $\norm*{\dobinf_{\nu}} \le \delta < 1$, then $\nu$ is $\parens*{\frac{1}{1-\delta}}$-spectrally independent.
\end{theorem}

In order to prove the MLSI for the perturbed Ising model, we will need to show that the block Glauber dynamics on connected components introduced earlier has entropy factorization. 
To this end, we will need the notion of a \emph{tilt}.
\begin{definition}[Tilts]
    A \emph{tilt} $\lambda = (\lambda_1,\dots,\lambda_k)$ is a collection of functions $\lambda_i:\calP_i\to\R_{> 0}$ for $i \in [k]$, and the \emph{tilted distribution} $\nu\ast\lambda$ is:
    \[
        (\nu\ast\lambda)(x) \propto \nu(x)\cdot \prod_{i=1}^k \lambda_i(x_i)\mper
    \]
\end{definition}

The following is a consequence of \cite[Theorem 4 and Theorem 5]{AJKPV21a} and \cite[Proposition 20 and Remark 70]{AASV21}.
\begin{lemma}   \label{thm:specind-to-entind}
    Suppose for every tilt $\lambda$, the distribution $(\nu\ast\lambda)$ is $\gamma$-spectrally independent for some constant $\gamma > 0$. Then $\nu$ satisfies the following conservation of entropy property.
    For any function $f:\calP\to\R_{> 0}$\,,
    \[
        \E_{i\sim[k]}\E_{\bx_{-i}\sim\nu_{-i}} \Ent_{\nu_i}[f] \ge \frac{1}{k^{\gamma}}\cdot\Ent_{\nu}[f],
    \]
    where $\nu_i$ is the conditional marginal of site $i$ conditioned on all other sites.
\end{lemma}

The following is an immediate corollary of \Cref{thm:dobrushin-to-specind,thm:specind-to-entind}.
\begin{corollary}   \label{cor:entind-dob}
    Suppose for every tilt $\lambda$, the influence matrix satisfies:
    \[
        \norm*{\dobinf_{\nu\ast\lambda}} \le \delta < 1.
    \]
    Then for every function $f:\calP\to\R_{> 0}$\mcom
    \[
        \E_{i\sim[k]}\E_{\bx_{-i}\sim\nu_{-i}} \Ent_{\nu_i}[f] \ge \frac{1}{k^{\frac{1}{1-\delta}}}\cdot\Ent_{\nu}[f]\mper
    \]
\end{corollary}

\subsection{MLSI for Ising models on perturbed near-forests} \label{sec:perturbed-near-forest}

We now establish an MLSI for perturbed Ising models which satisfy certain assumptions (\Cref{ass:parameters}). Following this, we verify these assumptions for our concrete setting (\Cref{lem:parameter-dep}). For convenience, we restate our goal.

\nearforestjunk*

Many of the tools introduced in this subsection are reused in the next subsection, where we bound the covariance of the Ising models obtained along the localization path.

\parhead{MLSI for perturbed Ising models.} Consider a collection of connected components $C_1,\dots,C_m$, as well as a collection of isolated vertices $v_1,\dots,v_{m'}$, and set $n = \sum_{i} \abs{C_i} + m'$. We study a ``base'' Ising model with interaction matrix $J \in \R^{n \times n}$ supported on the edges of the $(C_i)$ and diagonal, along with an arbitrary external field $h \in \R^n$. 

A key assumption we make is the following.
\begin{assumption}  \label{ass:parameters}
    For each $C_i$, and for any external field $h$, the Ising model $\mu_{J_{C_i},h}$ on $\{\pm1\}^{V(C_i)}$ satisfies an MLSI with constant $\zeta$.
    Additionally, suppose that $\eta, \delta, \theta > 0$ are small enough parameters (which can be shrinking with $n$) such that:
    \begin{align*}
        \max_i |C_i| &\le n^{\theta} \\
        \eta \cdot \sum_i |C_i|^2 &< \delta \\
        2\cdot \eta^2\cdot m' \cdot \sum_i |C_i|^2 &< \delta\mper
    \end{align*}
    We will be interested in perturbed Ising models of the form $J+E$, where every entry of $E$ is at most $\eta$, and $E_{ij} = 0$ if both $i$ and $j$ are among the $(v_k)_{1 \le k \le m'}$.
\end{assumption}
To keep notation succinct, we use $\mu$ as shorthand for the Gibbs distribution of the perturbed Ising model $\mu_{J+E, h}$.

In the setting of \Cref{cor:entind-dob}, each $\calP_i$ corresponds to the collection of spins on the vertices of one of the connected components $C_j$ or the spin on one of the isolated vertices $v_j$.
We index our influence matrix by $C_1,\dots,C_m,v_1,\dots,v_{m'}$.
The following is a consequence of a straightforward calculation.

\begin{observation}
    For any tilt $\lambda$, the influence matrix satisfies:
    \begin{align*}
        \dobinf_{\mu\ast\lambda}[C_i, C_j] &\le 1 - \exp(-\eta |C_i|\cdot|C_j|) \\
        \dobinf_{\mu\ast\lambda}[C_i, v_j] &\le 1 - \exp(-\eta |C_i|) \\
        \dobinf_{\mu\ast\lambda}[v_i, C_j] &\le 1 - \exp(-\eta |C_j|) \\
        \dobinf_{\mu\ast\lambda}[v_i, v_j] &= 0\mper
    \end{align*}
\end{observation}
Using $\exp(x)\ge 1+x$, we get the following as a corollary.

\begin{corollary}\label{cor:dobrushin}
    For any tilt $\lambda$,
    \[
        \norm*{\dobinf_{\mu\ast\lambda}}_F \le \sqrt{ \eta^2\cdot\parens*{\sum_i |C_i|^2}^2 + 2\cdot\eta^2 \cdot m'\cdot\sum_i |C_i|^2 }\mper
    \]
    Hence, under \Cref{ass:parameters}, $\norm*{\dobinf_{\mu\ast\lambda}}_{\mathrm{op}} \le \delta + \sqrt{\delta} \le 2\sqrt{\delta}$.
\end{corollary}

Thus, by \Cref{cor:entind-dob}, we get the following.

\begin{corollary}\label{cor:entropy-conservation-perturb}
    Let $k = m + m'$. Under \Cref{ass:parameters}, the perturbed Ising model $\mu$ satisfies the following entropy conservation bound.
    \[
        \E_{b\sim\{C_1,\dots,C_m\}\cup\{v_1,\dots,v_{m'}\} } \E_{\bx_{-b}\sim\mu_{-b}} \Ent_{\mu_b}[f] \ge \frac{1}{k^{\frac{1}{1-2\sqrt{\delta}}}} \cdot \Ent_{\mu}[f]\mper
    \]
\end{corollary}

Next, we observe that based on the assumptions on the parameters, for any individual connected component $C_i$, any Ising model with the perturbed interaction matrix restricted to $C_i$ satisfies an MLSI.

\begin{lemma}\label{lem:perturb-mlsi-centered}
    Under \Cref{ass:parameters}, for any external field $h$, we have 
    \[
        \exp\parens*{-O\parens*{\eta \cdot n^{2\theta}}} \le \frac{ \mu_{(J+E)_{C_i}, h}(x) }{ \mu_{J_{C_i}, h}(x) } \le \exp\parens*{O\parens*{\eta \cdot n^{2\theta}}}\mper
    \]
    Consequently, the Glauber dynamics on $\mu_{(J+E)_{C_i}, h}$ satisfies
    \[
        \MLSI(\mu_{(J+E)_{C_i}, h}) \ge \exp\parens*{-O\parens*{\eta \cdot n^{2\theta}}}\cdot\zeta.
    \]
\end{lemma}
\begin{proof}
    For any $x\in\{\pm1\}^{V(C_i)}$, the first bound easily follows from the bound on $|C_i|$ and the entries of $E$. 
    The second statement then follows from \Cref{lem:mlsi-bdd-density-relate}.
\end{proof}

We are now ready to prove the MLSI for Glauber dynamics for $\mu$. Recall that $n = \sum_{i}|C_i| + m'$, so $\mu$ is a distribution on $\{\pm1\}^n$.
\begin{lemma}\label{lem:mlsi-under-assumptions}
    Under \Cref{ass:parameters}, Glauber dynamics on $\mu$ satisfies
    \[
        \MLSI(\mu) \ge \exp\parens*{-O\parens*{\eta \cdot n^{2\theta}}} \cdot \frac{\zeta}{n^{\frac{1}{1-2\sqrt{\delta}}}} \cdot \left(1 - \frac{\delta}{\eta n}\right).
    \]
\end{lemma}
\begin{proof}
    We can sample a uniform index $i$ from $[n]$ by the following two-step process. First, sample $b$ from a certain distribution $\calD$ on $\{C_1,\dots,C_m\}\cup\{v_1,\dots,v_{m'}\}$, and then choose $i$ as a uniformly random vertex from $b$. In particular, the distribution $\calD$ is defined as:
    \[
        \Pr[b = v_i] = \frac{1}{n} \qquad \qquad \Pr[b = C_i] = \frac{|C_i|}{n}.
    \]
    For any $b$, and any pinning $x_{-b}$ to the spins outside of $b$, the resulting conditional distribution on $x_b$ is an Ising model with interaction matrix $(J+E)_b$ and some external field $h\parens*{x_{-b}}$ depending on the pinning.
    This lets us write the distribution $\mu$ as a mixture:
    \[
        \mu = \E_{b\sim\calD} \E_{x_{-b}\sim\mu_{-b}} \mu_{(J+E)_b, h\parens*{x_{-b}}}\mper
    \]
    We can succinctly express the mixture distribution as $\rho$, the pair $(b, x_{-b})$ as $\bz$, and the resulting Ising model as $\pi_{\bz}$.
    To prove the MLSI, we write the following chain of inequalities.
    For any function $f:\{\pm1\}^n\to\R_{\ge 0}$:
    \begin{align*}
        \calE_{\mu}(f, \log f) &\ge \E_{\bz\sim\rho} \calE_{\pi_{\bz}}(f, \log f) \\
        &\ge \exp\parens*{-O\parens*{\eta \cdot n^{2\theta}}} \cdot \zeta \cdot \E_{\bz\sim\rho} \Ent_{\pi_{\bz}}[f] \tag*{\text{(MLSI for $\pi_{\bz}$, \Cref{lem:perturb-mlsi-centered})}}\\
        &= \exp\parens*{-O\parens*{\eta \cdot n^{2\theta}}} \cdot \zeta \cdot \E_{b\sim\calD} \E_{x_{-b}\sim\mu_{-b}} \Ent_{\pi_{(b,x_{-b})}}[f] \\
        &\ge \exp\parens*{-O\parens*{\eta \cdot n^{2\theta}}} \cdot \zeta \cdot \left(1 - \frac{\delta}{\eta n}\right) \cdot \E_{b\sim\{C_1,\dots,C_m\}\cup\{v_1,\dots,v_{m'}\}} \E_{x_{-b}\sim\mu_{-b}} \Ent_{\pi_{(b,x_{-b})}}[f] \\
        &\ge \exp\parens*{-O\parens*{\eta \cdot n^{2\theta}}} \cdot \frac{\zeta}{n^{\frac{1}{1-2\sqrt{\delta}}}} \cdot \left(1 - \frac{\delta}{\eta n}\right) \cdot \Ent_{\mu}[f],  \tag*{(\Cref{cor:entropy-conservation-perturb})}
    \end{align*}
    where in the fourth line we move from $\calD$ to the uniform distribution over the parts using the fact that for any component $C_i$,
    \[ \frac{|C_i|/n}{1/(m+m')} \ge \frac{m}{n} \ge \frac{n - \sum_{i : |C_i| > 1} |C_i|}{n} \ge 1 - \frac{\delta}{\eta n}. \]
    This completes the proof.
\end{proof}

\parhead{Component sizes of sparse random graphs.}
To prove \Cref{thm:near-forest-junk-mlsi} and obtain explicit parameter dependencies, we will need to verify \Cref{ass:parameters} for sparse random graphs. 
The following results give finer control on the near-forest decomposition; we defer their proofs to \cref{sec:extinction}.
\begin{restatable}{lemma}{componentsize}\label{lem:component-size}
    Let $\bG\sim\SBM(n, d, \lambda)$ for any constant $d\ge 1$. Let $\bG_2$ be the near-forest component guaranteed by \Cref{lem:random-graph-decomposition}, and write $\bG_2$ as the union of connected components $\bG_2 = \bigsqcup_{i} C_i$.

    For $\eps \ge \Omega(\log d/d)^{1/3}$, the following holds with $\alpha \coloneqq \exp(-c\eps^3 d)$ for some absolute constant $c > 0$.
    With probability $1 - o(1)$ over $\bG$, 
    \[
        \sum_{i: \abs{C_i} > 1} \abs{C_i}^2 \le \alpha n.
    \]
\end{restatable}

\begin{restatable}{lemma}{singlecomp}\label{lem:single-component-bound}
    With probability $1-o(1)$, the maximum size of a connected component is at most $n^{\tfrac{c\log(d(1+\eps))}{\eps^3 d}}$ for some absolute constant $c > 0$.
\end{restatable}

Let us put everything together and prove that \Cref{ass:parameters} holds for sparse random graphs; this will also be required later when we bound the covariance of the Ising models obtained along the localization path.

\begin{lemma}\label{lem:parameter-dep}
    Let $\bG \sim \SBM(n, d, \lambda)$. Then $\mu_{J_{R}}$ satisfies \Cref{ass:parameters} with $m' \le n$ and
    \begin{align*}
        (\eta, \delta, \theta, \zeta) \leftarrow (\tfrac{2d}{n}, 8\alpha d^2, \tfrac{c\log(d(1+\eps))}{\eps^3 d}, n^{-o_d(1)}),
    \end{align*}
    where $c$ is a positive universal constant.
    Consequently, for $\eps = \Omega(\log d/d)^{1/3}$, Glauber dynamics on $\mu_{J_R}$ satisfies
    \[
        \MLSI(\mu_{J_R}) \ge \frac{1}{n^{1 + o_d(1)}}\mper
    \]
\end{lemma}
\begin{proof}
    Let $\cc(H)$ denote the connected components of $H$, and write $H = \sqcup_{C \in \cc(H)} C$. Let $\calS = \{C \in \cc(H): \abs{C} > 1\}$ be the nontrivial connected components, and $\calT = \{C \in \cc(H): \abs{C} = 1\}$ be the isolated vertices. 
    In our setup, the error $E_H + E$ has entries whose magnitude are bounded by $\frac{d}{n}\left(1 + \frac{\abs{\lambda}}{\sqrt{d}}\right) \le \frac{2d}{n}$. 
    On the other hand, 
    \Cref{lem:component-size} implies that $\abs{\calS} \le \alpha n$, where $\alpha = \exp(-c\eps^3 d)$, and we trivially have $\abs{\calT} \le n$.
    
    Hence, we will verify the conditions for $\eta = \frac{2d}{n}$ and $m' = n$.
    Plugging these parameters in, we need $\delta$ to satisfy
    \begin{align*}
        \frac{2 d}{n} \cdot \sum_{C \in \calS} \abs{C}^2 &< \delta \\
        2 \cdot \left(\frac{2d}{n}\right)^2 \cdot n \cdot \sum_{C \in \calS} \abs{C}^2 &< \delta.
    \end{align*}
    Applying \Cref{lem:component-size} again, we see that the above inequalities are satisfied for $\delta > 8\alpha d^2$.
    Hence, we pick $\delta = 8\alpha d^2$, and we easily compute $1 - \frac{\delta}{\eta n} = 1 - 4\alpha d = 1 - o_d(1)$.

    Turning now to $\theta$, we need $\theta$ to satisfy 
    \[
        \max_{C \in \calS} \abs{C} \le  n^{\theta}.
    \]
    \Cref{lem:single-component-bound} implies that we can take $\theta = \frac{c\log(d(1+\eps))}{\eps^3 d} = o_d(1)$ for some positive constant $c$. 
    Next, by the proof of \Cref{thm:mlsi-bdd-growth-tree}, we can take $\zeta = n^{-\theta - o_n(1)} = n^{-o_d(1)}$ (observe that we get $n^{-\theta-o(1)}$ instead of $n^{-1-o(1)}$ here because we are considering the Ising model over $\{ \pm 1\}^{V(C_i)}$ instead of $\{\pm 1\}^n$).

    Finally, plugging these parameters into \Cref{lem:mlsi-under-assumptions}, we have 
    \[
        \exp(-O(\eta \cdot n^{2\theta})) = 1 - O\left(\frac{d}{n^{1-2\theta}}\right)
    \]
    By picking the constant factor in $\eps$ large enough, we can ensure that $\theta$ is sufficiently small so that the above term is $1 - o_n(1)$. 
\end{proof}
Clearly, \Cref{lem:parameter-dep} directly proves \Cref{thm:near-forest-junk-mlsi}.

\subsection{Entropy conservation for centered stochastic block models}
\label{sec:centered-adjacency}

In this section, we prove \Cref{thm:centered-adjacency}, which is used to establish entropy conservation for centered Ising models on sparse SBMs.

\centeredadjacency*

Let $B,H$ be a $(\Delta, D)$-decomposition as guaranteed by \Cref{lem:random-graph-decomposition}.
As before, we will use \Cref{lem:interaction-sum-cov-bound} with 
\[
    M_1 = (1-t)J_B + \theta I_B + \frac{\rho}{\Delta} I_{H \setminus \partial H}  \quad \text{ and } M_2 = J_R,
\]
where we can take $\theta = \frac{t}{2} + \eta + 2\beta(1+\eps)$ due to \Cref{lem:bulk-spectral-radius} and $\rho$ as in the proof of \Cref{thm:entropy-conservation}.
As before, we can reduce to the case $t=0$. 

Moreover, since $J_B$ is only supported on the bulk, one can follow the proof of \Cref{thm:entropy-conservation} to argue that to bound $\Cov(\mu_{M_1 + M_2, h})$, one only needs to bound covariances of Ising models on $M_2$. 
Namely, it suffices to establish the following PSD upper bounds for arbitrary external field $h'$:
\begin{align*}
    \left\| \Cov\left( \mu_{M_2,h'} \right)_B \right\| &\le O(1) \\
    \left\|\Cov(\mu_{M_2,h'})_H\right\| &\le O\left(\frac{\Delta}{D}\right).
\end{align*}

The main difficulty stems from the fact that $M_2 = J_R$ has support outside $H$; it is now supported on the complete graph on $H$, as well as the edges between $B$ and $H$.

To analyze $\Cov(\mu_{M_2, h})$, the high level strategy will be to use \Cref{fact:var-cov} and transfer over to the case of having interactions purely supported on $H$. 
To do so, we will prove an optimal Poincar\'{e} inequality for a Markov chain where the states are the connected components of $H$.
Let us illustrate this argument informally; the subsequent lemmas will justify why we can execute these steps.

Recall that $\cc(H)$ denotes the set of connected components of $H$, and let $C \in \cc(H)$ be an arbitrary component.
By a slight abuse of notation, let $M_C = (M_H)_{C}$ denote the restriction of $M_H$ to $C$ and $E_C = (E_H)_{C}$ denote the restriction for the error in the near-forest. 
For any $v \in \R^n$ supported on $S \subseteq [n]$, we have
\begin{align*}
    \Vart_{x \sim \mu_{M_2, h'}}(\angles{v, x}) &\le O(1) \sum_{C \in \cc(H)} \E_{x_{-C}} \Vart_{x \sim \mu_{M_C + E_C, h(x_{-C})}}(\angles{v_C, x_C}) \tag*{(\text{Poincar\'{e}})}\\
    &\le O(1) \sum_{C \in \cc(H)} \E_{x_{-C}} \Vart_{x \sim \mu_{M_C, h(x_{-C})}}(\angles{v_C, x_C}) \tag*{(\text{Perturbation})} \\
    &\le O(1) \sum_{C \in \cc(H)} \norm{v_C}^2 \cdot \E_{x_{-C}} \left[\norm{\Cov(\mu_{M_C, h(x_{-C})})_{S}}\right] \mper \tag*{(\text{$v$ supported on $S$})}
\end{align*}

By selecting $S = B$ and $S = H$, we can transfer over the covariance bounds computed for near-forest Ising models in \Cref{lem:near-forest-cov} to finish off the proof. 

For the remainder of this section, we will justify the above steps and complete the proof of \Cref{thm:centered-adjacency}.

\parhead{Establishing an optimal Poincar\'{e} inequality.}
Let us first recall the definition of the Poincar\'{e} inequality on product domains.
\begin{definition}
    \label{lem:poincare}
    Suppose that $\nu$ is a distribution on a product domain $\calP = \calP_1 \times \cdots \times \calP_k$, and let $\PI(\nu)$ be the Poincar\'{e} constant for Glauber dynamics on $\nu$. 
    Then for any $f : \calP \to \R$, we have
    \[
        \Var_\nu(f) \le \frac{1}{k \cdot \PI(\nu)} \sum_{i \in [k]} \E_{y_{-i}} \Var_{y_i | y_{-i}}(f).
    \]
\end{definition}
We claim that the block Glauber dynamics studied in \Cref{sec:perturbed-near-forest} has an optimal Poincar\'{e} constant. 
In particular, for us $k = m + m'$ is the total number of connected components of $H$ (including isolated vertices).
Indeed, due to \Cref{lem:parameter-dep,cor:dobrushin}, we have  
\begin{align*}
    \norm{\dobinf_{\nu}} &\le 8\alpha d^2 + \sqrt{8\alpha d^2} \\
    &\le \exp(-\Omega(\eps^3 d)) \tag*{(\Cref{lem:component-size})}.
\end{align*}

We now have the following lemma from \cite{Hay06}.

\begin{lemma}[\cite{Hay06}]
    If the influence matrix of a distribution $\nu$ is bounded as $\|\calR_{\nu}\| < 1-\eps$, the mixing time of Glauber dynamics is $O((n \log n)/\eps)$.
\end{lemma}

In conjunction with the previous equations, this implies an optimal mixing time bound for the block dynamics, and thus an $\Omega\left(\frac{1}{k}\right)$ bound on $\PI(\nu)$.

\parhead{Perturbation.}
Next, we justify the transfer from $\mu_{M_C+E_C, h}$ to $\mu_{M_C, h}$, thereby removing the pesky additional weak interactions $E_C$. 
To accomplish this, we will use the simple perturbation bound \Cref{lem:mlsi-bdd-density-relate}.
Indeed, we already proved such a density bound for perturbed Ising models in \Cref{lem:perturb-mlsi-centered}.
Plugging in the explicit parameters for SBMs, as we did in the proof of \Cref{lem:parameter-dep}, we see that we can take $c = 1 + o_n(1)$ in \Cref{lem:mlsi-bdd-density-relate}.
In other words, we barely incur any loss transferring to $M_C$. 

With the above lemmas in hand, we can finish off the proof of the theorem.
\begin{proof}[Proof of \Cref{thm:centered-adjacency}]
    By \Cref{fact:var-cov}, if $v \in \R^n$ is a unit norm vector supported on $S \subseteq [n]$, then there is a positive constant $c > 0$ such that 
    \begin{align*}
        \Vart_{x \sim \mu_{M_R, h'}}(\angles{v, x}) &\le c \sum_{C \in \cc(H)} \E_{x_{-C}} \Vart_{x \sim \mu_{M_C + E_C, h(x_{-C})}}(\angles{v_C, x_C}) \tag*{(\text{\Cref{lem:poincare}})}\\
        &\le c(1+o_n(1)) \sum_{C \in \cc(H)} \E_{x_{-C}} \Vart_{x \sim \mu_{M_C, h(x_{-C})}}(\angles{v_C, x_C}) \tag*{(\text{\Cref{lem:mlsi-bdd-density-relate}})} \\
        &\le c(1+o_n(1)) \sum_{C \in \cc(H)} \norm{v_C}^2 \E_{x_{-C}} \left[ \norm{\Cov(\mu_{M_C, h(x_{-C})})_{S}}\right]\mper \tag*{(\text{$v$ supported on $S$})}
    \end{align*}

    Recall that $M_H = \frac{\beta}{\sqrt{d}}(\ol{A}_H \circ A_H)$.
    Since the centering interaction is $O(\frac{d}{n})$, it follows that 
    the interactions are in $\left[ -\frac{\beta}{\sqrt{d}} , \frac{\beta}{\sqrt{d}} \right] = \left[ - \frac{\beta\sqrt{1+\eps}}{\sqrt{d(1+\eps)}} , \frac{\beta\sqrt{1+\eps}}{\sqrt{d(1+\eps)}} \right]$, so we can apply \Cref{lem:near-forest-cov} with $\gamma = \beta\sqrt{1+\eps}$. 
    
    Taking $S = B$ and applying \Cref{lem:near-forest-cov} uniformly on each component $C$ yields
    \begin{align*}
        \norm{\Cov(\mu_{M_R, h'})_B} \le c(1+o(1)) \cdot \left(\frac{e^{2\beta/\sqrt{d}}}{(1-\beta\sqrt{1+\eps})^2}\right),
    \end{align*}
    and for $S = H$ we obtain that 
    \begin{align*}
        \norm{\Cov(\mu_{M_R, h'})_H} \le c(1+o(1)) \cdot \left(\frac{e^{2\beta/\sqrt{d}}}{(1-\beta\sqrt{1+\eps})^2} \right) \cdot \frac{\Delta}{D} .
    \end{align*}
    By following the same calculations as in \Cref{thm:main-mixing}, the upshot is that for sufficiently small (constant) $\beta > 0$, we have for some constant $C$ independent of $\Delta$ that
    \begin{align*}
        \Cov(\mu_{(1-t)J_B + J_R, h}) \psdle C \left( I_B + \Delta I_{H \setminus \partial H} \right),
    \end{align*}
    as desired.
\end{proof}


\section{Near-forest decomposition for stochastic block models}
\label{sec:extinction}
In this section, we prove \Cref{lem:random-graph-decomposition}, a statement about the existence of a near-forest decomposition for sparse random graphs.
\randomgraphdecomp*
\begin{notation*}[$\eps,\, d,\, \bG$]
Throughout this section, let $\eps$ and $d$ be as in the statement of \Cref{lem:random-graph-decomposition}, and let $\bG\sim\SBM(n, d, \lambda)$.
\end{notation*}

We begin the proof by first defining the construction of $\bG_1$ and $\bG_2$.

\begin{definition}[Excision of a graph]
    For a graph $G$, and for each vertex $v\in V(G)$, define $N_\ell(v)$ to be the set of vertices at distance exactly $\ell$ from $v$ and $B_\ell(v)$ to be the induced subgraph on $G$ on those vertices within distance $\ell$ of $v$. 
    We say that $v$ is \emph{$\ell$-heavy} if $|B_\ell(v)| > (d(1+\eps))^\ell$, and conversely that $v$ is \emph{$\ell$-light} if $|B_{\ell}(v)| \le (d(1+\eps))^{\ell}$, and set
    \[
        \ell(v) = \min\{ \ell \ge 0 : v \text{ is $L$-light for all $L \ge \ell$} \}.
    \]
    Note in particular that if $|B_L(v)| \le (d(1+\eps))^L$ for all $L \ge 0$, then $\ell(v) = 0$ and $B_{\ell(v)}(v)$ has no edges.
    Finally, we define the \emph{excision} of a graph $G$ as
    \[
        H \coloneqq \bigcup_{v \in V} B_{\ell(v)}(v),
    \]
    and refer to the graph after the removal of edges in $H$ as the \emph{bulk} of $G$.
\end{definition}

The ingredients to prove \Cref{lem:random-graph-decomposition} are:
\begin{itemize}
    \item \Cref{obs:excision-bound-pseudorandom}, which says that the bulk of any graph has bounded degree, and the excision of any graph is $(\Delta, D)$-pseudorandom for some $\Delta$ and $D = (1+\eps)d$.
    \item \Cref{lem:components-tree-excess-1}, which says that the excision of a sparse random graph is a near-forest with high probability.
    \item \Cref{lem:log-star-bound}, which shows that $\Delta = o(\log n)$ with high probability.
\end{itemize}

\begin{notation*}[$\bH$]
    We will use $\bH$ to refer to the excision of $\bG$.
\end{notation*}

\begin{observation} \label{obs:excision-bound-pseudorandom}
    For any graph $G$, let $G_1$ and $G_2$ denote its bulk and excision respectively.
    Then, for some $\Delta > 0$ and $D \coloneqq (1+\eps)d$:
    \begin{enumerate}
        \item The maximum degree of $G_1$ is at most $D$.
        \item $G_2$ is $(\Delta,D)$-pseudorandom with respect to $S\coloneqq V'(G_1)\cap V'(G_2)$.
    \end{enumerate}
\end{observation}

\begin{lemma}   \label{lem:components-tree-excess-1}
    With probability $1-o(1)$, $\bH$ is a near-forest.
\end{lemma}
The following statement establishes the claimed bound on $\Delta$.
\begin{restatable}{lemma}{logloglog}    \label{lem:log-star-bound}
    With probability $1-o(1)$, for every vertex $v$ in $\bG$ and $\ell \ge 0$, $\frac{|B_\ell(v)|}{(d(1+\eps))^{\ell}} = o(\log n)$.
\end{restatable}

We are now ready to prove \Cref{lem:random-graph-decomposition}.
\begin{proof}[Proof of \Cref{lem:random-graph-decomposition}]
    Choose $\bG_1$ as the bulk of $\bG$, and $\bG_2$ as the excision of $\bG$. 
    The statement then immediately follows from \Cref{obs:excision-bound-pseudorandom,lem:components-tree-excess-1,lem:log-star-bound}.
\end{proof}

It now remains to prove the above statements.
The statement of \Cref{lem:log-star-bound} is a mild strengthening of \cite[Lemma 9 of arXiv version]{BGGS24} and follows standard arguments involving branching processes, and hence, we defer its proof to \Cref{appendix-a}.

The rest of this section is dedicated to proving \Cref{lem:components-tree-excess-1}, which establishes that $\bH$ is a near forest with high probability.

\subsection{Excisions of sparse random graphs are near-forests}
The key technical ingredient in the proof of \Cref{lem:components-tree-excess-1} is the following.
\begin{lemma}
    \label{lem:components-diam}
    With probability $1-o(1)$, for any connected component $C$ of $\bH$, $\diam(C) \le O\parens*{\tfrac{\log n}{\eps^3 d}}$.
\end{lemma}

As foreshadowed in the technical overview, we will also need the following structural result about small sets in sparse random graphs.
\begin{lemma}[{\cite[Lemma 30]{BLM15}, also see \cite[Lemma 31]{BGGS24}}]
    \label{lem:small-subgraphs-have-excess-1}
    For $\bG\sim\SBM(n,d,\lambda)$, with high probability, there is no connected set $S$ of vertices with $|S| < \frac{1}{2} \cdot \frac{\log n}{\log d}$ that has tree-excess larger than $1$.
\end{lemma}
Assuming the above, let us present a proof of \Cref{lem:components-tree-excess-1}.

\begin{proof}[Proof of \Cref{lem:components-tree-excess-1}]
    Condition on the event in \Cref{lem:components-diam}, so we have that every component $C$ of $H$ has diameter at most $\log n/d$. Suppose instead that $C$ had tree-excess at least $2$, and let $C_1,C_2$ be minimal cycles in $C$ (in that they do not have any chords). Consider the induced subgraph on the set $S$ of vertices in $C$ comprised of $C_1,C_2$, and a shortest path (in $C$) between the two. Note in particular that $S$ is connected, and $|S| \le 5 \diam(C)$. Since $\eps = \Omega_d(\log d/d)^{1/3}$, we may apply \Cref{lem:small-subgraphs-have-excess-1} to conclude that the tree-excess of the induced subgraph on $S$ is at most $1$, a contradiction.
\end{proof}

    We now turn our attention to proving \Cref{lem:components-diam}, which bounds the diameter of each connected component of $\bH$.
    A key auxiliary graph in our proof is the \emph{cluster graph}.
    \begin{definition}\label{def:cluster-graph}
        Given a graph $G$ on $[n]$, the associated \emph{cluster graph} $\Cluster(G)$ has vertex set $[n]$, with an edge between $u$ and $v$ iff $\ell(u),\ell(v) > 0$ and $B_{\ell(u)}(u) \cap B_{\ell(v)}(v) \ne \emptyset$. 
        Furthermore, each vertex $u$ in this graph has associated weight $\wtt(u) = (1+\eps)^{\ell(u)}$, and each edge $uv$ has weight $\wtt(u, v) = \ell(u)+\ell(v)$.
    \end{definition}
    Observe that \Cref{lem:components-diam} follows immediately from the following two statements:
    \begin{enumerate}
        \item \Cref{lem:no-tree-excess-cluster-graph}, which shows that if a connected component of $\bH$ has large diameter, then it causes a high-weight induced path in $\Cluster(\bG)$.
        \item \Cref{lem:no-long-path-cluster}, which establishes that there are no high-weight induced paths in $\Cluster(\bG)$.
    \end{enumerate}
    \begin{lemma}
    \label{lem:no-tree-excess-cluster-graph}
        Suppose the diameter of a connected component of $\bH$ exceeds $\tau$.
        Then there are vertices $u_1,\dots, u_k\in[n]$ such that the induced subgraph in $\Cluster(\bG)$ is the path $u_1\dots u_k$, and
        \[
            \sum_{i\in[k]} \ell(u_i) \ge \frac{\tau}{2}\mper
        \]
    \end{lemma}
    \begin{lemma}   \label{lem:no-long-path-cluster}
        Define $\eta$ as $\min\left\{\frac{\eps^2}{54}, \frac{1}{18}\right\}$.
        With probability $1-o(1)$, for any $U = \{u_1,\dots,u_k\}$ such that the induced subgraph on $U$ in $\Cluster(\bG)$ is the path $u_1\dots u_k$, we have:
        \[
            \sum_{i\in[k]} \ell(u_i) < \frac{32\log n}{\eta\eps d}\mper
        \]
    \end{lemma}

    The proof of \Cref{lem:no-tree-excess-cluster-graph} is fairly short and self-contained, and we present it below.
    \begin{proof}[Proof of \Cref{lem:no-tree-excess-cluster-graph}]
        There must be vertices $a$ and $b$ in a single connected component of $\bH$ such that $\dist_{\bH}(a,b) > \tau$.
        Let $u$ and $v$ be vertices such that $B_{\ell(a)}(a)$ and $B_{\ell(b)}(b)$ contain $u$ and $v$ respectively.

        Let $u_1u_2\dots u_k$ be the weighted shortest path between $u$ and $v$ in $\Cluster(\bG)$ where $u_1 = u$ and $u_k = v$. 
        The path between $a$ and $b$ has length bounded by $2\sum_{i\in[k]}\ell_i$, and hence:
        \[
            \sum_{i\in[k]} \ell_i \ge \frac{\tau}{2}\mper
        \]
        The induced subgraph on $\{u_1,\dots,u_k\}$ certainly contains the claimed path by construction.
        It remains to show that there are no other edges besides the ones in the path.
        If there is an edge between $u_i$ and $u_j$ for $j\ne i\pm 1$ for $i < j$, then the path $u_1\cdots u_i u_j u_{j+1} \cdots u_k$ is a shorter path than the original one, which leads to a contradiction.
        Indeed, since the edge weights are positive and $\wtt(u, v) = \ell(u) + \ell(v)$, any path between $u_i$ and $u_j$, in particular $u_i u_{i+1} \cdots u_{j-1} u_j$, has weight at least $\ell(u_i) + \ell(u_j)$, with equality if and only if the path is just the edge $u_iu_j$.
    \end{proof}

    We now begin the proof of \Cref{lem:no-long-path-cluster} by setting up some terminology.
    \begin{definition}[Violating set and minimal violating set]
        We say a subset of vertices $U$ is a \emph{violating set} if the induced subgraph on $U$ in $\Cluster(\bG)$ is a path of weight exceeding $\tfrac{32\log n}{\eta\eps d}$.
        We say $U$ is a \emph{minimal violating set} if it is a violating set, and no subset of $U$ is a violating set.
    \end{definition}

    \begin{definition}[Signature]
        The \emph{signature} of a minimum violating set is the tuple $\parens*{k,\vec{u},\vec{\ell}}$ where $k$ is the number of vertices in the minimum violating set, $\vec{u} = (u_1,\dots,u_k)$ is the vertices in the minimum violating set in the order they appear in the path, and $\vec{\ell} = (\ell(u_1),\dots,\ell(u_k))$.
    \end{definition}

    \begin{notation*}[{$\calS_{\kappa}$}]
        For any positive integer $\kappa$, define $\calS_{\kappa}$ as all signatures $({k,\vec{u},\vec{\ell}})$ for which $k\ge 1$ and $\vec{\ell}$ is in $[\kappa]^{k}$.
    \end{notation*}

    To prove \Cref{lem:no-long-path-cluster}, it suffices to prove that there are no violating sets in $\Cluster(\bG)$. 
    Our goal is to prove that with high probability, there are no violating sets in $\Cluster(\bG)$.
    We observe that it suffices to prove there are no minimal violating sets (\Cref{obs:min-violating-set}), and then bound the probability of containing a minimal violating set in \Cref{obs:MVS-bound}.
    \begin{observation} \label{obs:min-violating-set}
        If $U\subseteq[n]$ is a violating set in $\Cluster(\bG)$, then it contains a minimal violating set $U'\subseteq U$. 
    \end{observation}
    \begin{observation} \label{obs:MVS-bound}
        For any positive integer $\kappa$, the probability that $\Cluster(\bG)$ contains a minimal violating set is bounded by
        \[
            \sum_{\parens*{k,\vec{u},\vec{\ell}}\in\calS_{\kappa}} \Pr\bracks*{\exists\text{MVS with signature }\parens*{k,\vec{u},\vec{\ell}}\text{ in }\Cluster\parens*{\bG} } + \Pr\bracks*{\exists v\in[n]\text{ s.t. }\ell(v) > \kappa},
        \]
        where MVS is short for ``minimal violating set''.
    \end{observation}

    The following, proved in \Cref{appendix-a}, gives a choice of $\kappa$ for which the second term of the above is negligible.
    \begin{restatable}{lemma}{radiusbound} \label{lem:radius-bound}
        For $\kappa = \frac{4}{\eps}\log\frac{\log n}{\eta d}$, we have:
        \[
            \Pr\bracks*{\exists v\in[n]\text{ s.t. }\ell(v)>\kappa} = o(1)\mper
        \]
    \end{restatable}
    The following gives us a handle on the first term of the bound in \Cref{obs:MVS-bound}.
    \begin{lemma}   \label{lem:single-signature-bound}
        For $\kappa = \frac{4}{\eps}\log\frac{\log n}{\eta d}$ and for any $(k,\vec{u},\vec{\ell})\in\calS_{\kappa}$ where $\vec{\ell} = (\ell_1,\dots,\ell_k)$, we have:
        \begin{align*}
            \Pr&\bracks*{\exists\text{MVS with signature }\parens*{k,\vec{u},\vec{\ell}} \text{ in }\Cluster(\bG) } \\
            &\qquad\le \frac{1}{n^{k-1}}\prod_{i\in[k]} 2\ell_i^2\cdot\parens*{d(1+\eps)}^{2\ell_i}\cdot\exp\parens*{-\frac{\eta d}{2}\parens*{1+\frac{\eps}{4}}^{\ell_i}}\mper
        \end{align*}
    \end{lemma}
    \begin{remark}
        The way to interpret the above is to think of the probability bound as decaying at a \emph{doubly} exponential rate in each $\ell_i$, and for the events $\{u_i\text{ is }\ell_i\text{-heavy}\}$ as being approximately independent.
        See \Cref{lem:ball-tail} and its proof to understand the doubly exponential decay rate.
    \end{remark}
    We first show how to complete the proof of \Cref{lem:no-long-path-cluster}, and then later prove \Cref{lem:single-signature-bound}.

    \begin{proof}[Proof of \cref{lem:no-long-path-cluster}]
        By \Cref{obs:min-violating-set} it suffices to prove that there is no minimal violating set in $\Cluster(\bG)$.
        The definition of a signature $\parens*{k,\vec{u},\vec{\ell}}$ of a minimal violating set implies that $\sum_i\ell_i \ge \frac{32\log n}{\eta\eps d}$.
        Hence, by \Cref{obs:MVS-bound,lem:radius-bound,lem:single-signature-bound}, we get that the probability of containing a violating set is at most:
        \begin{align}\label{eq:long-path-prob}
            n \sum_{k \ge 1} \sum_{\substack{\ell_1, \ldots \ell_k \\ \sum \ell_i \ge \tfrac{32\log n}{\eta\eps d}}} \prod_{i \in [k]} 2\ell_i^2 \cdot (d(1+\eps))^{2\ell_i} \cdot \exp\left(-\frac{\eta d}{2}\left(1+\frac{\eps}{4}\right)^{\ell_i}\right).
        \end{align}
        Consider one of the terms in the product, which is
        \[
            \exp\left(\log(2\ell_i^2) + 2\ell_i \log (d(1+\eps)) - \frac{\eta d}{2}\left(1+\frac{\eps}{4}\right)^{\ell_i}\right) \le \exp\left(\ell_i (2\log (d(1+\eps))+2) - \frac{\eta d}{2}\left(1+\frac{\eps}{4}\right)^{\ell_i}\right).
        \]
        We claim that for
        \begin{equation}
            \label{eq:eps-constraint-main}
            \eps \ge 1500 \left( \frac{100+\log d}{d} \right)^{1/3},
        \end{equation}
        the above is at most $\exp \left( - \ell_i \left( 1 + \frac{\eta \eps d}{16} \right) \right)$.
        Indeed, since $\left(1+\frac{\eps}{4}\right)^{\ell_i} \ge \frac{\eps}{4} \cdot \ell_i$, the above quantity is at most
        \[
            \exp\left( - \ell_i \left( \frac{\eta \eps d}{8} - 2 - 2 \log\left( d (1+\eps) \right) \right) \right).
        \]
        Therefore, it suffices for $\eps$ to satisfy
        \[
            \frac{\eta \eps d}{16} \ge 2 \log \left( d (1+\eps) \right) + 3.
        \]
        It is straightforward to verify that $\eps$ as in~\eqref{eq:eps-constraint-main} satisfies the above.
        Given this, \eqref{eq:long-path-prob} is at most
        \begin{align*}
            n \sum_{k \ge 1} \sum_{\substack{\ell_1, \ldots \ell_k \\ \sum \ell_i \ge \tfrac{32\log n}{\eta\eps d}}} \exp\left(-\left(\frac{\eta\eps d}{16} + 1\right)\sum_{i \in [k]} \ell_i\right) & 
            \le \frac{1}{n} \sum_{k \ge 1} \sum_{\substack{\ell_1, \ldots \ell_k \\ \sum \ell_i \ge \tfrac{32\log n}{\eta\eps d}}} \exp\left(-\sum_{i \in [k]} \ell_i\right) \\
            &\le \frac{1}{n} \sum_{k \ge 1} \sum_{\ell_1, \ldots \ell_k \ge 1} \exp\left(-\sum_{i \in [k]} \ell_i\right) \\
            &\le \frac{1}{n} \sum_{k \ge 1} \exp(-\Omega(k)) = o(1),
        \end{align*}
        as desired.
    \end{proof}

    We now turn our attention to proving \Cref{lem:single-signature-bound}.
    The following two statements (\Cref{lem:gain-if-tree-excess,lem:independent-set-independent}) will be used in the proof.
    We defer their proofs to \Cref{appendix-a}.
    \begin{restatable}{lemma}{lineindependent}
    \label{lem:independent-set-independent}
        Let $u_1,u_2,\ldots,u_r \in [n]$ and $\ell_1,\ell_2,\ldots,\ell_r > 0$. Then, 
        \[ \Pr\left[ \ell(u_i) = \ell_i \,\,\forall i \in [r]  \text{ and }  B_{\ell_i}(u_i) \cap B_{\ell_j}(u_j) = \emptyset \,\,\forall i\neq j \in [r]\right] \le \prod_{i \in [r]} \exp\left(-\eta d \left(1+\frac{\eps}{4}\right)^{\ell_i}\right). \]
    \end{restatable}
    \begin{remark}
        In the above bound, the reader should treat the events $\{\ell_i = \ell(u_i)\}$ as roughly independent.
        Informally, the approximate independence can be justified by the fact that the union of the balls covers only a small fraction of the graph, and hence for a vertex $u_r$ disjoint from $B_{\ell(u_1)}(u_1)\cup\dots\cup B_{\ell(u_{r-1})}(u_{r-1})$, the probability that $B_{\ell(u_r)}(u_r)$ intersects the rest is very small --- it is roughly the probability that two small subsets of $[n]$ intersects.
        Furthermore, when the ball around $u_r$ is disjoint from the remaining balls, its size is not influenced by the other balls, resulting in approximate independence.
    \end{remark}
    
    \begin{restatable}{lemma}{lineedgesgain}
        \label{lem:gain-if-tree-excess}
        Let $S,T_1,T$ be disjoint subsets of vertices, and say $T = T_2 \sqcup T'$. Let $\mathcal{E}$ be an (arbitrary) conditioning of the edges adjacent to vertices in $S,T_1,T$ such that for any $A = S,T_1,T$, we can write $A = A_{\inn} \sqcup A_{\outt}$ such that the following holds.
        \begin{itemize}
            \item $|A| = o(n)$.
            \item \emph{Every} edge incident on vertices in $A_{\inn}$ is fixed in a way that there are no edges between $A_{\outt}$ and $A^c$.
            \item Further, there is no conditioning done on the edges between $A_{\outt}$ and $(S_{\inn} \cup (T_1)_{\inn} \cup T_{\inn})^c$.
        \end{itemize}
        Let $u \in [n]$ and $\ell > 0$. Then, for $n$ sufficiently large,
        \[ \Pr\left[ \text{$u$ is $\ell$-light and } B_{\ell}(u) \cap T_i \ne \emptyset \text{ for } i \in \{1, 2\} \mid \mathcal{E} \right] \le \frac{2\ell^2 \cdot (d(1+\eps))^{2\ell} \cdot |T_1| \cdot |T_2|}{n^2}. \]
    \end{restatable}
    \begin{remark}
    Put more intuitively, \cref{lem:gain-if-tree-excess} is a generalization of the following simple observation: for any small set of vertices of size $C$, then the probability that a random vertex is contained in said set scales like $C/n$. 
    The technical assumptions on $S, T_1,$ and $T$ are present to ensure that the lemma applies throughout the proof of \cref{lem:no-long-path-cluster}; see \Cref{fig:overlapping-balls} for a depiction of these sets.
    \end{remark}

    \begin{figure}[ht]
        \centering
        \begin{tikzpicture}[
            ball/.style={circle, draw=black, fill=white, minimum size=20mm, align=center},
            special ball/.style={circle, draw=blue, dashed, fill=blue!5, opacity=0.5, pattern=north east lines, pattern color=blue!40, minimum size=20mm},
            t2/.style={circle, draw=green, fill=green!20,minimum size=20mm},
            clip ball/.style={circle, draw=black, fill=white, minimum size=20mm, postaction={fill=blue!20, even odd rule, clip}},
            t1/.style={circle, draw=purple, fill=purple!20,minimum size=20mm},
            next ball/.style={circle, draw=red, fill=red!10, minimum size=20mm}]
    
        \pgfdeclarelayer{background}
        \pgfdeclarelayer{foreground}
        \pgfsetlayers{background,main,foreground}
        
        \def\circleT2{(3.5,0) circle (10mm)}
        \def\circleTprime{(1.8,0.5) circle (10mm)}
    
        
        \begin{scope}[shift={(-2.25, 0)}]
            
            \begin{pgfonlayer}{background}
                \node[t2] (C) at (3.5,0) {$B_{\ell_3}(u_3)$};
    
                \node[t1] (E) at (6.5,0) {$B_{\ell_5}(u_5)$};
    
                \node[special ball] (D) at (5,0.5) {};
                \node at (5,0.5) {$B_{\ell_4}(u_4)$};
            
            \end{pgfonlayer}
    
            \begin{pgfonlayer}{foreground}
                \node[ball] (A) at (0,0) {$B_{\ell_1}(u_1)$};
                \node[ball] (B) at (1.8,0.5) {$B_{\ell_2}(u_2)$};
        
                \begin{scope}
                    \clip \circleT2;
                    \draw[green,thick,dashed] \circleTprime;
                \end{scope}
                
                \node[next ball] (F) at (8.75,0) {$B_{\ell_7}(u_7)$};
                \node[next ball] (G) at (10.9,0.5) {$B_{\ell_9}(u_9)$};
            \end{pgfonlayer}
            \draw[thick, rounded corners=10mm](-1.5,-3) -- (-1.5,2) -- (12.25, 2) -- (12.25,-3) -- cycle;
        
            \node at (-0.75,-2) {\huge$G$};
        \end{scope}

        \begin{scope}[shift={(10.5,-1.25)}]
            \draw[thick] (0,-0.75) rectangle (2.1,2.6);
            
            \node[anchor=west] at (0,2.25) {\textbf{Legend}};
    
            \node[anchor=west] at (0.25,1.75) {\tikz\fill[ball] (0,0) rectangle (0.5,0.3); {$T'$}};
            \node[anchor=west] at (0.25,1.25) {\tikz\fill[t2] (0,0) rectangle (0.5,0.3); $T_2$};
            \node[anchor=west] at (0.25,0.75) {\tikz\fill[special ball] (0,0) rectangle (0.5,0.3); $B_{\ell}(u)$};
            \node[anchor=west] at (0.25,0.25) {\tikz\fill[t1] (0,0) rectangle (0.5,0.3); $T_1$};
            \node[anchor=west] at (0.25,-0.25) {\tikz\fill[next ball] (0,0) rectangle (0.5,0.3); $S$};
            
        \end{scope}
        \end{tikzpicture}
        \caption{A depiction of the sets $T_1, T_2, T', S$, and $B_{\ell}(u)$. 
        The event of interest in \Cref{lem:gain-if-tree-excess} is when the random set of interest $B_{\ell}(u)$ (shaded blue) intersects the two target sets $T_1$ (solid purple) and $T_2$ (solid green). 
        The other two sets of vertices $S$ and $T'$ cover the additional edges that need to be conditioned on to apply the lemma inductively to bound the length of the path $u_1\ldots u_k$ in the cluster graph. 
        In our setting, all the sets involved are unions of balls centered at the vertices $u_i$.
        }

        \label{fig:overlapping-balls}
    \end{figure}
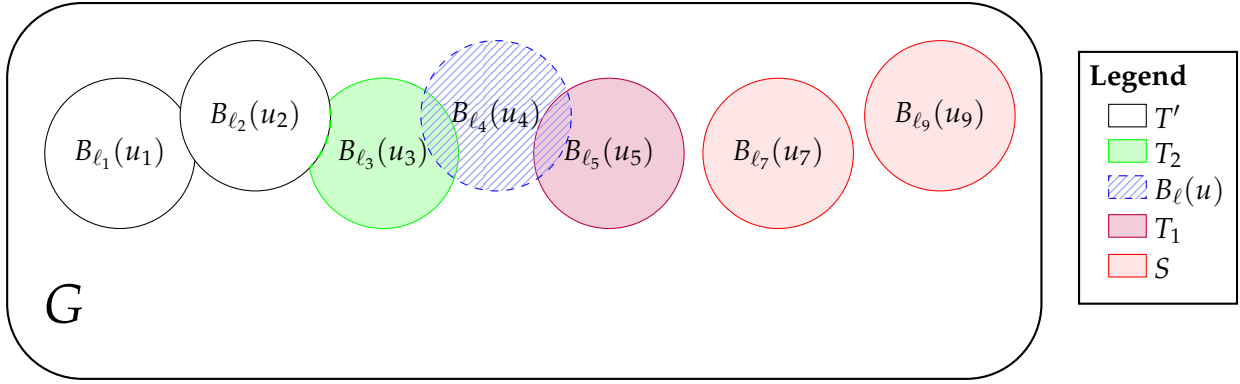

    We now prove \Cref{lem:single-signature-bound}.
    \begin{proof}[Proof of \Cref{lem:single-signature-bound}]
        Assume, without loss of generality, that
        $$\sum_{i\text{ odd}}(1+\eps)^{\ell_i} \ge \sum_{i\text{ even}} (1+\eps)^{\ell_i}\mper$$
        Thus, for the independent set $\{u_i:i\text{ odd}\}$, we have 
        $$\sum_{i\text{ odd}} (1+\eps)^{\ell_i} \ge \frac{1}{2} \sum_{i} (1+\eps)^{\ell_i}\mper$$ 

        Let $\mathcal{E}_{\text{independent}}$ be the event that $\ell(u_i) = \ell_i$ for all odd $i$ and $B_{\ell_i}(u_i) \cap B_{\ell_j}(u_j) = \emptyset$ for distinct odd $i, j$. 
        Let $\mathcal{E}_{\text{cross}}$ be the event that for all even $i$ and $j = i\pm 1$, $B_{\ell_i}(u_i) \cap B_{\ell_{j}}(u_{j}) \ne \emptyset$. Then, 
        \begin{align}
            &\Pr\left[ \exists\text{MVS with signature }\parens*{k,\vec{u},\vec{\ell}} \text{ in }\Cluster(\bG) \right] \nonumber \\
            &\quad\quad \le \Pr\bracks*{\calE_{\text{independent}},\,\calE_{\text{cross}}} \label{eq:cardinality-dropper} \\
            &\quad\quad = \Pr[\mathcal{E}_{\text{independent}}] \cdot \Pr[\mathcal{E}_{\text{cross}} \mid \mathcal{E}_{\text{independent}}] \nonumber \\
            &\quad\quad \le \prod_{i \text{ odd}} \exp(-\eta d (1+\eps)^{\ell_i}) \cdot \prod_{i \text{ even}} 2\ell_i^2 \cdot \frac{(d(1+\eps))^{2\ell_i + \ell_{i-1} + \ell_{i+1}}}{n^2} \label{eq:even-vertices}\\
            &\quad\quad \le \frac{1}{n^{k-1}} \prod_{i \in [k]} 2\ell_i^2 \cdot (d(1+\eps))^{2\ell_i} \cdot \exp\left(-\frac{\eta d}{2} \left(1+\frac{\eps}{4}\right)^{\ell_i}\right) \label{eq:odd-to-all}.
        \end{align}
        To explain the above chain of inequalities, in~\eqref{eq:cardinality-dropper} we dropped the condition that $\ell(u_i) = \ell_i$ for even $i$.
        In~\eqref{eq:even-vertices}, we applied \cref{lem:independent-set-independent} and \cref{lem:gain-if-tree-excess} inductively with the setting
        \[
            T_1 = B_{\ell_{i+1}}(u_{i+1}),\, S = \bigcup_{\substack{i\text{ odd} \\ j > i+1}} B_{\ell_j}(u_j), \, T'=\bigcup_{j < i-1} B_{\ell_j}(u_j), \, T_2 = B_{\ell_{i-1}}(u_{i-1}) \setminus T'\mper
        \]
        When applying \Cref{lem:gain-if-tree-excess}, we use the bound of $\kappa$ on $\ell_i$ to conclude that the size of the union of all these balls is $o(n)$.
        See \Cref{fig:overlapping-balls} for an illustration. In~\eqref{eq:odd-to-all}, we used the fact that the odd vertices have at least half of the total vertex weight.
    \end{proof}

\subsection{Bounding the sizes of components}\label{sec:components-bound}
In this section, we prove high probability bounds on the sizes of connected components in $\bH$, which are of interest when studying the Ising model on  centered interaction matrices.

\Cref{lem:single-component-bound}, which states that the maximum size of a connected component is $n^{o_d(1)}$, is a consequence of \Cref{lem:components-diam}. 
\singlecomp*
We believe that our analysis here is loose, and that the size of the largest connected component is in fact $\poly\log(n)$ with high probability.
\begin{proof}
    By \Cref{lem:components-diam}, we know that for any connected component $C$, there is an absolute positive constant $c$ such that $\diam(C) \le \frac{c\log n}{\eps^3 d}$. 
    Note that any vertex $u \in C$ is $\diam(C)$-light. 
    Furthermore, we have $B_{\diam(C)}(u) \supseteq C$, so 
    \[
        \abs{C} \le (d(1+\eps))^{\diam(C)} \le \exp\left(\log(d(1+\eps)) \cdot \frac{c\log n}{\eps^3 d}\right) = n^{\tfrac{c\log(d(1+\eps))}{\eps^3 d}},
    \]
    as desired.
\end{proof}

Using similar ingredients, we are also able to establish the following quantitative bound on the sum of squared component sizes, from which \Cref{lem:component-size} follows.
\begin{lemma}\label{lem:component-size-precise}
    Write $H = \sqcup_{C \in \cc(H)} C$, where $\cc(H)$ denotes the connected components of $H$. 
    There are absolute constants $c_1, c_2 > 0$ such that for $\eps \ge c_1(\log d/d)^{1/3}$, with probability $1-o(1)$, 
    \[
        \sum_{\substack{C \in \cc(H) \\ \abs{C} > 1}} \abs{C}^2 \le \exp(-c_2 \eps^3 d)n.
    \]
\end{lemma}

We now turn our attention to proving \Cref{lem:component-size-precise}, which gives a high probability bounds on the sizes of component sizes in $H$. 
The proof essentially amounts to bounding the first and second moments. 
For the first moment bound, many of the key ingredients are already contained in the proof of \Cref{lem:components-diam}; for the variance bound, we use the Efron-Stein inequality. 

\begin{proof}[Proof of \Cref{lem:component-size-precise}]
    For succinctness, define the random variable 
    \begin{align*}
        X \coloneqq \sum_{C \in \cc(H)} \abs{C}^2 \cdot \Ind[\abs{C} > 1].
    \end{align*}
    For distinct $u, v \in [n]$, define the random variable $X_{uv} = \Ind[u, v \text{ are in the same component}]$. 
    Notice that 
    \begin{align*}
        X = \sum_{u \neq v} X_{uv}\mper
    \end{align*}
    Hence for the first moment it suffices to upper bound $\E[X_{uv}] = \Pr[u, v \text{ are in the same component}]$. 
    In fact, bounding $\E[X_{uv}]$ uses tools similar to what we used in the proof of \Cref{lem:components-diam}.

    Observe that for distinct $u$ and $v$, they are in the same connected component only if there are $u'$ and $v'$ (which need not be distinct) which satisfy the following:
    \begin{itemize}
        \item $\ell_{u'}, \ell_{v'} > 0$, so that $u'$ and $v'$ belong to nontrivial connected components.
        \item $u \in B_{\ell_{u'}}(u')$ and $v \in B_{\ell_{v'}}(v')$.
        \item There is a path between $u'$ and $v'$ in $\Cluster(\bG)$. 
    \end{itemize}
    Note that if $u' = v'$, we allow for self loops, as long as $\ell_{u'} > 0$; this is consistent with \Cref{def:cluster-graph}.
    Hence, we have 
    \begin{align*}
        \sum_{u \neq v} X_{uv} &\le \sum_{u', v' \in [n]} \abs{B_{\ell_{u'}}(u')} \abs{B_{\ell_{v'}}(v')} \Ind[u', v' \text{ connected in } \Cluster(\bG)].
    \end{align*}
    Fix a path length $k$, positive integers $\ell_i$ for $i \in [k]$, and a path $w_1\cdots w_k$ with $w_1 = u'$ and $w_k = v'$.
    Consider the event $\calE$ whose probability we bounded in \Cref{lem:no-long-path-cluster}:
    \[
        \calE \coloneqq \{\ell_i = \ell(w_i) \text{ for $i \in [k]$ and } B_{\ell_i}(w_i) \cap B_{\ell_{j}}(w_{j}) \ne \emptyset \text{ iff $j = i \pm 1$ for $i \in [k]$}\}
    \]
    For the term at $u', v'$ in the sum to be nonzero, the above event $\calE$ must occur for some choice of $k$, positive $\ell_i$, and path from $u'$ to $v'$. 
    On $\calE$, we can upper bound 
    \[
        \abs{B_{\ell_{u'}}(u')} \cdot \abs{B_{\ell_{v'}}(v')} \cdot \Ind[u', v' \text{ connected in } \Cluster(\bG)] \le (d(1+\eps))^{\ell_1 + \ell_k},
    \]
    and a careful inspection of the proof of \Cref{lem:no-long-path-cluster} establishes that
    \begin{align*}
        (d(1+\eps))^{\ell_1 + \ell_k}\Pr[\calE] \le \prod_{i \in [k]} 2\ell_i^2 \cdot (d(1+\eps))^{2\ell_i} \cdot \exp\left(-\tfrac{\eta d}{2}(1+\tfrac{\eps}{4})^{\ell_i}\right).
    \end{align*}
    Indeed, in \eqref{eq:even-vertices}, only one factor of $(d(1+\eps))^{\ell_i}$ is picked up for $i \in \{1, k\}$, as they have only one neighbor in the chain. 

    Now choose $\eps$ such that 
    \[
        \frac{\eta \eps d}{16} \ge 2 \log \left( d (1+\eps) \right) + 2, 
    \]
    which is satisfied for $\eps \ge 1500 \left( \frac{100+\log d}{d} \right)^{1/3}$. 
    Summing up now over all $k$, $\ell_i$, and paths from $u'$ to $v'$, we obtain 
    \begin{align*}
        &\E\left[\sum_{u', v' \in [n]} \abs{B_{\ell_{u'}}(u')} \cdot \abs{B_{\ell_{v'}}(v')} \cdot \Ind[u', v' \text{ connected in } \Cluster(\bG)] \right] \\
        &\quad\quad \le n \sum_{k \ge 1} \sum_{\ell_1, \ldots \ell_k \ge 1}  \prod_{i \in [k]} 2\ell_i^2 \cdot (d(1+\eps))^{2\ell_i} \cdot \exp\left(-\tfrac{\eta d}{2}(1+\tfrac{\eps}{4})^{\ell_i}\right) \\
        &\quad\quad \le n \sum_{k \ge 1} \sum_{\ell_1, \ldots \ell_k \ge 1}  \exp\left(-\frac{\eta \eps d}{16} \sum_{i \in [k]} \ell_i\right) \\
        &\quad\quad \le n \sum_{k \ge 1} \exp\left(-\frac{\eta \eps kd}{16}\right) \prod_{i \in [k]} \frac{\exp\left(-\frac{\eta \eps d}{16}\right)}{1 - \exp\left(-\frac{\eta \eps d}{16}\right)} \\
        &\quad\quad \le n \sum_{k \ge 1} \exp\left(-\frac{\eta \eps kd}{16}\right) \\
        &\quad\quad \le 2\exp\left(-\frac{\eta \eps d}{16}\right)n,
    \end{align*}
    where in the second-to-last line we have used the fact that $\frac{\eta \eps d}{16} \ge 2$ and $\tfrac{\exp(-2)}{1-\exp(-2)} \le 1$, and in the last line we have used $\tfrac{1}{1-\exp(-2)} \le 2$.
    Hence, $\E[X] \le 2\exp(-\tfrac{\eta \eps d}{16})n$.

    We now show that $\Var(X) = o(\E[X]^2)$, which implies the desired result by applying Chebyshev. 
    To establish this, we appeal to the Efron-Stein inequality.
    Consider rerandomizing an edge $e$ in $\bG$. 
    With probability at most $1 - 2d/n$, the rerandomized edge $e$ agrees with the original edge, so the value of $X$ does not change. 
    If it disagrees, then $X$ changes only if $e$ connects two connected components or if $e$ breaks a connected component into two smaller connected components.
    Either way, we can upper bound the magnitude of the difference by $\max_{C \in \cc(H)} 2\abs{C}^2$. 
    
    By \Cref{lem:single-component-bound}, $\max_{C \in \cc(H)} \abs{C}^2 \le n^{\tfrac{c\log d}{\eps^3 d}}$ for some absolute constant $c > 0$, so 
    \begin{align*}
        \Var(X) &\le \frac{4d}{n} \cdot n^{\tfrac{c\log d}{\eps^3 d}} \cdot n^2 \\
        &\le 4\exp\left(\log d + \frac{c \log d}{\eps^3 d} \log n\right) n.
    \end{align*}
    To conclude, it suffices to verify that for our choice of parameters 
    \[
        \log d + \frac{c \log d}{\eps^3 d} \log n < -\eps^3 d + \frac{1}{2} \log n.
    \]

    Indeed, for $\eps \ge \left(\frac{K \log d}{d}\right)^{1/3}$ and any $d = O(\log n)$, the left-hand-side of the above inequality is $O(\log\log n) + \frac{c}{K} \log n$, whereas the right-hand-side is $\frac{1}{2}\log n - O(\log \log n)$. 
    Hence, the above inequality is satisfied for $K$ a sufficiently large constant, which finishes the proof.  
\end{proof}

\section*{Acknowledgments}
S.M.\ would like to thank Frederic Koehler for helpful conversations about Ising models.

\clearpage

\addcontentsline{toc}{section}{References}
\bibliographystyle{alpha}
\bibliography{refs}


\appendix

\section{Proofs omitted from the main text}
\label{appendix-a}

\subsection{Generalized MLSI for Ising models on pseudorandom near-forests}

We begin with a proof of \Cref{lem:pseudorandom-ferro-mixing}. While the proof we state in the main text which assumes that all nonzero interactions are equal suffices for our purposes, a more general result holds.

\pseudorandomferromixing*

\begin{proof}
    To prove this statement, as in the simpler proof given in the main text, we will use \Cref{lem:entropic-stability,thm:anneal}, albeit with the stochastic localization using a time-varying control matrix. Again, set $s_0 = \frac{\gamma}{\sqrt{D}}$. Consider the graph $H_t$ on $V(H)$ with edge set
    \[ E(H_t) = \left\{ ij \in E(H) : J_{ij} > t \cdot  \frac{\gamma}{\sqrt{D}} \right\}. \]
    In particular, $H_1$ is empty.
    The control matrix at time $t$ is defined by $C_t = \left((1-s_0^2) \cdot \BH_{H_t} \left( - \frac{\gamma}{\sqrt{D}} \right)\right)^{1/2}$. As in the simpler proof, this is a non-negative matrix that is well-defined due to \Cref{cor:bethe-tree-psd}. At time $t$, the interaction matrix $J_t$ is, on the off-diagonal elements, supported on the graph $H_t$, with the interactions being at most $\frac{\gamma}{\sqrt{D}}$. In particular, $J_1$ is a diagonal matrix so $\mu_{1}$ is a product measure. Now, for an arbitrary tilt $h$,
    \[ \Cov\left( \mu_{J_t,h} \right) \le \sum_{\ell \ge 0} \left( \frac{\gamma}{\sqrt{D}} \right)^{\ell} A^{(\ell)}_{H_t} \le \sum_{\ell \ge 0} \left( \frac{\gamma}{\sqrt{D}} \right)^{\ell} A^{(\ell)}_{H} = \BH_{H}(s_0)^{-1}, \]
    where the inequalities mean that the entries of the (non-negative) matrix on the left are at most that of the corresponding entries on the right. Similarly, because $H_t$ is a subgraph of $H$, $C_t^2 \le (1-s_0^2) \BH_H(-s_0)$. Using \Cref{cor:cov-ferro-external-psd,lem:tree-cov}, we have that
    \begin{align*}
        \norm*{ C_t \cdot \Cov\left( \mu_{J_t,h} \right) \cdot C_t } &\le \norm*{ \Cov\left( \mu_{J_t,h} \right) \cdot C_t^2 } \\
            &\le (1-s_0^2) \norm*{ \BH_{H}(s_0)^{-1} \cdot \BH_{H}(-s_0) } \\
            &= (1-s_0^2) \norm*{ \BH_{H}(s_0)^{-1} \cdot \left(\BH_{H}(s_0) + \frac{2s_0}{1-s_0^2} \cdot A_{H} \right) } \\
            &\le (1-s_0^2) + \frac{2s_0}{1-s_0^2} \norm*{ \BH_{H}(s_0)^{-1} \cdot A_{H} } \\
            &\le 1 + \frac{8}{(1-\gamma)^2} \cdot \frac{\Delta}{D}.
    \end{align*}
    The measure at time $1$ is almost surely a product measure, therefore
    \begin{align*}
        \MLSI(\mu_{J,h}) &\ge \frac{1}{n} \cdot \exp \left( - \int_0^1 \sup_{h \in \R^n} \left\| C_t \cdot \Cov\left( \mu_{J_t,h} \right) \cdot C_t \right\| \right) \\
            &\ge \frac{1}{n} \cdot \exp\left( - \left( 1 + \frac{8}{(1-\gamma)^2} \cdot \frac{\Delta}{D} \right) \right) = \frac{1}{ne^{O(\Delta)}}. \qedhere
    \end{align*}
\end{proof}

\subsection{Bounds on various probabilities in sparse random graphs}

Next, we prove numerous lemmas from \Cref{sec:extinction}, including \Cref{lem:gain-if-tree-excess,lem:radius-bound,lem:log-star-bound,lem:independent-set-independent}.
In the sequel, we study random graphs $\bG$ sampled from $\SBM(n,d,\lambda)|\bsigma$ where $\bsigma$ is a ``balanced'' community assignment.
In particular, we assume that $\frac{1}{n}|\angles*{\bone,\bsigma}|<O\parens*{\sqrt{\frac{\log n}{n}}}$.
This holds with probability $1$ if $\bG$ is an \erdos--\renyi graph, and more generally with probability $1 - o_n(1)$ for 2-community SBMs with nontrivial priors.

\begin{restatable}{lemma}{lballtail}
    \label{lem:ball-tail}
    For any vertex $v$ in $V(\bG)$ and $r \ge 1$, for $\eta = \min\left\{ \frac{\eps^2}{54} , \frac{1}{18} \right\}$,
    \[
        \Pr[ |B_r(v)| \ge (d(1+\eps))^r] \le \exp\left( - \eta d \left( 1 + \frac{\eps}{4} \right)^{r} \right).
    \]
\end{restatable}

\begin{proof}
    Consider a Galton-Watson tree $T$ with branching random variable $\Poi(d(1+o_n(1))$. It is well-known (see e.g. \cite[Example 4.2.8]{Roch24}) that $\Poi(d)$ stochastically dominates $\Bin(n-1, \frac{d}{n})$, so one can construct a monotone coupling between the local neighborhoods of $G$ and the Galton-Watson tree $T$ in the following way.
    Suppose $\alpha$ is the fraction of vertices $v$ such that $\bsigma(v) = +1$.
    Then, the number of neighbors of a given vertex $v$ is distributed as $\Bin\left(\alpha n - 1, \frac{d+\lambda\sqrt{n}}{n}\right) + \Bin\left((1-\alpha)n, \frac{d-\lambda\sqrt{n}}{n} \right)$, which is dominated by $\Poi\parens*{\alpha\cdot\frac{d+\lambda \sqrt{d}}{2} + (1-\alpha)\cdot\frac{d-\lambda\sqrt{d}}{2}}$.
    Since $\alpha = \frac{1}{2}\pm O\parens*{\sqrt{\frac{\log n}{n}}}$, the number of neighbors is stochastically dominated by $\Poi\parens*{d\parens*{1 + o_n(1)}}$. 
    For simplicity of notation, for the rest of the proof we replace $d(1+o_n(1))$ with $d$. 
    Since $\eps = \Omega(\log d/d)^{1/3}$, we can increase the implicit constant to ensure that the original tail bound as stated holds.

    Let $u$ be the root of $T$. Set $N^{\GW}_{\ell}(u)$ to be the set of vertices at depth exactly $\ell$ in the Galton-Watson tree, and $B^{\GW}_{\ell}(u)$ to be the set of vertices at depth at most $\ell$.
    By the above monotone coupling,  $|B_L(v)|$ is stochastically dominated by $|B^\GW_L(u)|$, so it suffices to show the relevant tail bound for $|B_\ell^{\GW}(u)|$.\\
    For ease of notation, set $R_t = |N^{\GW}_{t}(u)|$ and $B_t = |B^{\GW}_{t}(u)|$. Also let $B_t = 0$ for $t < 0$. Markov's inequality implies that for any $s > 0$,
    \[ \Pr\left[ |B_r| \ge (d(1+\eps))^r \right] \le e^{-s(d(1+\eps))^r} \E\left[ e^{sB_r} \right]. \]
    Now, for $t \ge 2$, using the recursion $B_t = B_{t-1} + R_t$, we have 
    \begin{align*}
        \E\left[ e^{sB_t} \right] &= \E\left[ e^{sB_{t-1}} \cdot e^{sR_t} \right] \\
            &= \E\left[ e^{sB_{t-1}} \cdot e^{d(e^s-1)R_{t-1}} \right] \tag{Poisson MGF} \\
            &= \E\left[ e^{sB_{t-2}} \cdot e^{(d(e^s-1) + s)R_{t-1}} \right]. 
    \end{align*}
    Define $g : \R_{\ge 0}^2 \to \R_{\ge 0}$ by
    \[ g(x,s) = d(e^x - 1) + s, \]
    so for any $r \ge 1$, $\E[e^{sB_r}] \le e^{g^{\circ r}(s,s)}$, where $g^{\circ 1}(x,y) = g(x,y)$ and $g^{\circ (t+1)}(x,y) = g(g^{\circ t}(x,y),y)$. Now, define $s_k = \delta (d+2\delta)^{-k}$ for some $\delta \ge 1/d$ that we shall set later. Then, for any $1 \le t \le r$,
    \begin{align*}
        g(s_t,s_r) &= d(e^{s_t} - 1) + s_r \\
            &\le d s_t e^{s_t} + s_r \\
            &\le s_{t-1} \cdot \left( \frac{de^{s_t}}{d+2\delta} + \frac{s_r}{s_{t-1}} \right) \\
            &\le s_{t-1} \cdot \left(\frac{de^{s_t}}{d+2\delta} + \frac{1}{d+2\delta}\right) \\
            &\le s_{t-1} \cdot \left(\frac{d\exp\left(\frac{\delta}{d+2\delta}\right)}{d+2\delta} + \frac{d}{d+2\delta}\right) \tag{$t \ge 1$} \\
            &\le s_{t-1}.
    \end{align*}
    Here, the final inequality follows from the fact that
    \[ \frac{\exp\left( \frac{x}{1+2x} \right)}{1+2x} + \frac{1}{1+2x} \le 1 \]
    for $x \ge 1$, which we have applied with $x = \frac{\delta}{d}$.
    We also have
    \[ g(s_0,s_r) = d(e^{s_0} - 1) + s_r \le d\delta e^{\delta} + s_r \le d\delta e^{\delta} + \delta. \]
    Consequently,
    \[ \E\left[ e^{s_{r-1} B_r} \right] \le e^{d\delta e^{\delta} + \delta}. \]
    Therefore,
    \begin{align*}
        \Pr\left[ |B_r| \ge (d(1+\eps))^r \right] &\le e^{d\delta e^{\delta} + \delta} \cdot e^{-s_{r-1}(d(1+\eps))^{r}} \\
            &= \exp\left(d\delta e^{\delta} + \delta\right) \cdot \exp\left( - \delta \cdot \frac{(d(1+\eps))^r}{(d+2\delta)^{r-1}} \right) \\
            &\le \exp\left(d\delta e^{\delta} + \delta\right) \cdot \exp\left( - \delta d \cdot \left( \frac{1+\eps}{1+\frac{2\delta}{d}} \right)^r \right).
    \end{align*}
    Now, choose $\delta = \min\left\{\frac{\eps}{4+2\eps} , \log\left( 1 + \frac{\eps}{8} \right)\right\}$. Recall that $\eps = \Omega\left( \frac{\log d}{d} \right)^{1/3}$, so this is indeed at least $1/d$. This implies that $\frac{1+\eps}{1+\frac{2\delta}{d}} \ge 1 + \frac{\eps}{2}$. Furthermore,
    \begin{align*}
        \frac{\exp\left( d \delta e^{\delta} + \delta \right) \cdot \exp\left( - \delta d \left( 1 + \frac{\eps}{2} \right)^r \right)}{ \exp\left( - \frac{\delta\eps d}{4+\eps} \left( 1 + \frac{\eps}{4} \right)^r \right) } &= \exp\left( d \delta e^{\delta} + \delta - d \delta \left(\left( 1 + \frac{\eps}{2} \right)^r - \frac{\eps}{4+\eps} \left( 1 + \frac{\eps}{4} \right)^r \right) \right) \\
            &\le \exp\left( d \delta e^{\delta} + \delta - d \delta \left( 1 + \frac{\eps}{2} - \frac{\eps}{4+\eps} \left( 1 + \frac{\eps}{4} \right) \right) \right) \\
            &= \exp\left( d \delta e^{\delta} + \delta - d \delta \left( 1 + \frac{\eps}{4} \right) \right) \\
            &= \exp\left( d \delta \left( e^{\delta} + \frac{1}{d} - \left( 1 + \frac{\eps}{4} \right) \right) \right) \\
            &\le \exp\left( d \delta \left( e^{\delta} - \left( 1 + \frac{\eps}{8} \right) \right) \right) \le 1,
    \end{align*}
    where in the last line we used $\eps \ge \frac{8}{d}$.
    To conclude, we set
    \begin{align*}
        \eta \coloneqq \frac{\delta \epsilon}{4+\eps} &= \frac{\eps}{4+\eps} \cdot \min \left\{ \frac{\eps}{4+2\eps} , \log \left( 1 + \frac{\eps}{8} \right) \right\} \\
        &\ge \min \left\{ \frac{\eps}{6} , \frac{1}{3} \right\} \cdot \min \left\{ \frac{\eps}{9} , \frac{1}{6} \right\} \\
            &\ge \min\left\{ \frac{\eps^2}{54} , \frac{1}{18} \right\}. \qedhere
    \end{align*}
\end{proof}

First, observe that \Cref{lem:radius-bound}, restated below, is an immediate corollary of the above.
\radiusbound*

Now, we proceed to fill in other missing proofs.
\logloglog*

\begin{proof}
    As in the proof of \Cref{lem:ball-tail}, set $R_\ell$ to be the number of vertices at depth $\ell$ in a Galton-Watson tree with branching random variable $\Poi(d)$ rooted at $v$. We shall show that
    \[ \Pr\left[ R_\ell \ge (d(1+\eps))^{\ell} \cdot \Delta \right] \le \frac{1}{n^3} \]
    for some $\Delta = o(\log n)$ ($\Delta$ is $o(\log n)$ uniformly for all $\ell$). Given this, we are done by taking a union bound over $0 \le \ell \le n$ and then all $n$ vertices $v$. Again, define
    \[ g(x) = d(e^x - 1), \]
    so for any $s > 0$,
    \[ \Pr \left[ R_\ell \ge (d(1+\eps))^{\ell} \cdot \Delta \right] \le e^{-s \Delta (d(1+\eps))^{\ell}} \cdot e^{g^{\circ \ell}(s)}. \]
    Alternately, set
    \[ h(x) = g^{-1}(x) = \log \left( \frac{x}{d} + 1 \right), \]
    so we get that for any $q > 0$,
    \[ \Pr \left[ R_\ell \ge (d(1+\eps))^{\ell} \cdot \Delta \right] \le e^{-h^{\circ \ell}(q) \Delta (d(1+\eps))^{\ell}} \cdot e^{q}. \]
    Now, set $q = \log n$ and
    \[ \Delta = \frac{4 \log n}{(d(1+\eps))^{\ell} \cdot h^{\circ \ell}(\log n)}. \]
    Clearly, this yields the $1/n^3$ bound claimed earlier. Further, we have
    \[ \min_{\ell \ge 0} \left\{(d(1+\eps))^{\ell} \cdot h^{\circ \ell}(\log n)\right\} = \omega(1), \]
    so $\Delta = o(\log n)$.
\end{proof}

The key idea in the proof of \Cref{lem:gain-if-tree-excess} is the following observation.

\begin{observation}
    \label{obs:uniformity-nbds}
    Let $S \subseteq V$ be any set of vertices in $G$ and $\ell$ be a positive integer. For any vertex $v \in V$ and conditioning on the induced subgraph on $S$, we have that $N_{\ell}(v) \setminus S$ is a uniformly random subset of $V \setminus S$ of size $\abs{N_{\ell}(v) \setminus S}$. 
\end{observation}

\lineedgesgain*

\begin{proof}[Proof of \cref{lem:gain-if-tree-excess}] 
    The idea is that due to \Cref{obs:uniformity-nbds} --- since the sets $S, T_1, T$ are all $o(n)$ in size, we are essentially picking uniformly random small subsets of vertices on $V$ and checking whether they intersect.

    To formalize this, consider growing out the ball around $u$ incrementally. 
    Let us first introduce some additional notation. 
    Set $\calL$ to be the event that $u$ is $\ell$-light, and let $\calF$ be the event in the lemma statement whose conditional probability we wish to bound: 
    \[
        \calF \coloneqq \calL \cap \{B_{\ell}(u) \cap T_i \ne \emptyset \text{ for } i \in \{1, 2\}\}
    \]For $r \ge 1$ and $i \in [2]$, let $\calE_{r, i}$ denote the event that $B_r(u) \cap T_i \ne \emptyset$ but $B_{r-1}(u) \cap T_i = \emptyset$, that is, $r$ is the smallest radius at which the ball intersects $T_i$. Noting that $\mathcal{E}_{r,i}$ are disjoint events for fixed $i$, we have
    \begin{equation}
        \Pr\left[ \calF \mid \calE \right] = \sum_{r_1,r_2=1}^{\ell} \Pr\left[ \mathcal{L} \cap \mathcal{E}_{r_1,1} \cap \mathcal{E}_{r_2,2} \mid \calE \right]. \label{eq:4.6-main}
    \end{equation}
    In words, we are splitting this into cases depending on when $T_1$ and $T_2$ are hit.
    To complete the proof, we shall bound each of the terms in this summation by
    \[ \frac{2(d(1+\eps))^{2\ell} \cdot |T_1| \cdot |T_2|}{n^2}. \]
    Fix $r_1,r_2$ with $r_1 \le r_2$. Let $\mathcal{L}_{r}$ be the event that $|B_{r}(u)| \le (d(1+\eps))^{\ell}$. We have
    \begin{align*}
        \Pr\left[ \mathcal{E}_{r_1,1} \cap \mathcal{E}_{r_2,2} \cap \mathcal{L} \mid \calE \right] &\le \Pr\left[ \mathcal{E}_{r_1,1} \cap \mathcal{E}_{r_2,2} \cap \mathcal{L}_{r_2} \mid \calE \right] \\
            &\le \Pr\left[ \mathcal{E}_{r_1,1} \cap \mathcal{L}_{r_2} \mid \calE \right] \cdot \Pr\left[ \mathcal{E}_{r_2,2} \mid \mathcal{L}_{r_2} \cap \mathcal{E}_{r_1,1} \cap \calE \right] \\
            &\le \Pr\left[ \mathcal{E}_{r_1,1} \mid \mathcal{L}_{r_1} \cap \calE \right] \cdot \Pr\left[ \mathcal{E}_{r_2,2} \mid \mathcal{L}_{r_2} \cap \mathcal{E}_{r_1,1} \cap \calE\right].
    \end{align*}
    Let us start by bounding the first of these two probabilities; refer to \Cref{fig:overlapping-balls} for a depiction of the sets involved. For convenience, we will denote $V' = V \setminus (T_2 \cup S)$. 
    Suppose that $B_{r_1-1}(u)$, $|N_{r_1}(u)|$ and $|N_{r_1}(u) \cap V'|$ were fixed in a way that $B_{r_1-1}(u)$ does not intersect $T_1$ or $T'$, and $\mathcal{L}_{r_1}$ is satisfied. Then, $N_{r_1}(u) \cap V'$ is in fact a uniformly random subset of $V' \setminus B_{r_1-1}(u)$ of the appropriate size. Since $T_1$ is assumed to be contained in $V'$, it follows that
    \[ \Pr\left[ \mathcal{E}_{r_1,1} \mid \mathcal{L}_{r_1} \cap \calE  \right] \le \frac{\E\left[|N_{r_1}(u)| \Big\vert \mathcal{L}_{r_1} \cap \calE\right] \cdot |T_1|}{n - o(n)} \le \frac{2(d(1+\eps))^{\ell} \cdot |T_1|}{n} . \]
    Similarly, we also get that
    \[ \Pr\left[ \mathcal{E}_{r_2,2} \mid \mathcal{L}_{r_2} \cap \mathcal{E}_{r_1,1} \cap \calE \right] \le \frac{2(d(1+\eps))^{\ell} \cdot |T_2|}{n} . \]
    Plugging these two bounds into \eqref{eq:4.6-main}, we get that
    \[ \Pr\left[ \calF \mid \calE \right] \le \frac{2\ell^2 \cdot (d(1+\eps))^{2\ell} \cdot |T_1| \cdot |T_2|}{n^2}, \]
    as claimed.
\end{proof}

\lineindependent*

The key observation for this proof is the following.

\begin{observation}
    \label{obs:tail-bounds-better-after-conditioning}
    Let $S \subseteq V$ be a (small) subset of vertices in $G$. Consider a partial conditioning of the edges for which
    \begin{itemize}
        \item the edges within $S$ are arbitrarily conditioned on.
        \item for any potential edge $e$ from $S$ to $S^c$, either $e$ is not conditioned on at all, or is conditioned to be absent.
    \end{itemize}
    Now, fix some vertex $u$, and suppose we grow out a radius $r$ ball from $v$, conditioned on it not intersecting $S$. Then, this is equivalent to growing out a radius $r$ ball from $v$ within a smaller completely random graph on $V \setminus S$, and then adding back the edges corresponding to $S$, with a further conditioning that there are no edges between $B_{r-1}(u)$, the interior vertices of the ball, and $V \setminus B_{r}(u)$ (which, in particular, includes $S$).
    
    Consequently, any monotonically nondecreasing function of the sizes of neighborhoods in the conditioned graph are stochastically dominated by their versions without conditioning. In particular, the random variable $\ell(v)$ in the conditioned graph is stochastically dominated by $\ell(v)$ in $G(n,d/n)$.
\end{observation}

\begin{proof}[Proof of \cref{lem:independent-set-independent}]
    We will upper bound the probability inductively; the base case immediately follows from \Cref{lem:ball-tail}.

    Let $\calA_r$ denote the event that $\ell(u_i) = \ell_i$ for every $i \in [r]$, i.e. the first $r$ balls are heavy. 
    Further, let $\calB_r$ denote the event that $B_{\ell_i}(u_i) \cap B_{\ell_j}(u_j) = \emptyset$ for all distinct $i, j \in [r]$, i.e. the first $r$ balls are pairwise disjoint.
    Noting that $\calB_r = \calB_{r-1} \cap \{B_{\ell_r}(u_r) \cap B_{\ell_i}(u_i) = \emptyset \,\, \forall i \in [r-1]\}$, we have
    \begin{align*}
        \Pr\left[ \calA_r \cap \calB_r \right] &= \Pr[\calA_{r-1} \cap \calB_{r-1}] \cdot \Pr\left[B_{\ell_r}(u_r) \cap B_{\ell_i}(u_i) = \emptyset \,\, \forall i \in [r-1] \Big\vert \calA_{r-1} \cap \calB_{r-1}\right] \\
        &\quad\quad \cdot \Pr\left[\ell(u_r) = \ell_r \Big\vert \mathcal{A}_{r-1} \cap \calB_{r}\right] \\
        &\le \prod_{i \in [r-1]} \exp(-\eta d (1+\eps)^{\ell_i}) \cdot \Pr\left[\ell(u_r) = \ell_r \Big\vert \mathcal{A}_{r-1} \cap \calB_{r}\right] \\
        &\le \prod_{i \in [r]} \exp(-\eta d (1+\eps)^{\ell_i}),
    \end{align*}
    where in the last line we have used \cref{obs:tail-bounds-better-after-conditioning} and \Cref{lem:ball-tail} with $S = \bigcup_{i =1}^{r-1} B_{\ell_i}(u_i)$.
\end{proof}

\end{document}